\documentclass[11pt,reqno]{amsart}
\usepackage{a4wide}
\usepackage{amssymb}
\usepackage{mathrsfs}
\usepackage{enumerate}
\usepackage{esint}
\usepackage{titletoc}
\usepackage[colorlinks=true, urlcolor=red, linkcolor=red, citecolor=blue]{hyperref}
\usepackage[nameinlink]{cleveref}
\usepackage{soul}
\usepackage{xcolor}
\setstcolor{blue} 

\theoremstyle{plain}
\newtheorem{theorem}{Theorem}[section]
\newtheorem{lemma}[theorem]{Lemma}

\theoremstyle{definition}
\newtheorem{definition}[theorem]{Definition}

\numberwithin{equation}{section}

\DeclareMathOperator{\supp}{supp}
\DeclareMathOperator*{\osc}{osc}
\DeclareMathOperator*{\essosc}{ess\,osc}

\DeclareMathOperator{\dist}{dist}
\newcommand{\R}{{\mathbb{R}}}

\newcommand{\diff}{\mathop{}\!\mathrm{d}}
\makeatletter
\@namedef{subjclassname@2020}{\textup{2020} Mathematics Subject Classification}
\makeatother

\newcommand{\dd}{\,\mathrm d}

\newcommand{\Tail}{\operatorname{Tail}}
\newcommand{\PV}{\operatorname{P.V.}}
\newcommand{\Wpot}{\mathbf W}

\title[Gradient estimates for the fractional $p$-Laplacian]{Gradient estimates for the fractional $p$-Laplacian in the superquadratic regime}
\author{Ying Li, Chao Zhang$^*$}
\address{Ying Li\newline
	School of Mathematics, Harbin Institute of Technology, Harbin 150001, China\newline
	\texttt{lymath@hit.edu.cn}}
\address{Chao Zhang\newline
	School of Mathematics and Institute for Advanced Study in Mathematics, Harbin Institute of Technology, Harbin 150001, China
	\newline
	\texttt{czhangmath@hit.edu.cn}}
\thanks{$^*$ Corresponding author.}
\thanks{{\bf Keywords}: fractional $p$-Laplacian, measure data, gradient estimate,
	Wolff potential, affine excess, harmonic replacement.}
\thanks{{\bf MSC 2020}: Primary 35R09, 35J60. Secondary 35B65, 35D30, 31C45}

\begin{document}

\begin{abstract}
	We establish gradient potential estimates for solutions obtained
	as limits of approximations (SOLA) to the fractional $p$-Laplace
	equation with finite signed Radon measure data.
	Under the assumptions $n\ge2$, $p>2$, $0<s<1$, and $sp>p-1$,
	every SOLA belongs to $W^{1,p-1}_{\mathrm{loc}}$,
	and its weak gradient satisfies a Wolff potential estimate
	at every Lebesgue point of the weak gradient.
	Under the additional condition $sp>n$, the solution has a
	continuous representative that is Fr\'echet differentiable
	at every point where the potential is finite.
	These results give a partial answer to questions raised by
	Diening and Nowak [Ann. PDE, \textbf{11}(2025)] and by
	Diening, Kim, Lee and Nowak~[J. Eur. Math. Soc. (JEMS), 2025]
	concerning gradient potential estimates for the fractional
	$p$-Laplacian.  The proof combines homogeneous affine decay with constants
	independent of the affine slope and comparison estimates
	to obtain an affine excess recurrence. Iteration then yields
	the Wolff potential bound.
\end{abstract}
	\maketitle
	
	\section{Introduction}
	Let $\Omega\subset\R^n$, $n\ge2$, be open and bounded. We consider the
Dirichlet problem
	\begin{equation}\label{eq:main}
		\begin{cases}
			(-\Delta_p)^s u=\mu & \text{in }\Omega,\\
			u=0 & \text{in }\R^n\setminus\Omega,
		\end{cases}
	\end{equation}
	where $\mu\in\mathcal M(\Omega)$ is a finite signed Radon measure,
	and
	\[
	(-\Delta_p)^s u(x)
	:=2\,\PV\int_{\R^n}
	\frac{|u(x)-u(y)|^{p-2}(u(x)-u(y))}{|x-y|^{n+sp}}\,\diff y.
	\]
	For general measure data, an energy weak solution need not exist, since the measure does not in general define an element of \((W^{s,p}_0(\Omega))'\). We therefore work with SOLA in the sense of Kuusi, Mingione and Sire \cite{KuusiMingioneSire2015}. Their existence theory covers the fractional \(p\)-Laplace equation considered in \eqref{eq:main}.  We recall the specialized definition in	Definition~\ref{def:sola}.  We claim here that all estimates below are interior. The zero exterior condition simplifies the approximation procedure and the passage to the limit in the tails.
	
	Throughout the paper, we assume
	\begin{equation}\label{eq:gradient-parameter-range}
		n\ge2,\qquad p>2,\qquad 0<s<1,\qquad sp>p-1.
	\end{equation}
	We define
	\[
	m:=p-1,\qquad \beta:=sp-p+1\in(0,1),\qquad
	\gamma:=\frac\beta p=s-\frac{p-1}{p}.
	\]
	The condition $sp>p-1$ guarantees that nonconstant affine functions
	have finite $(p-1,sp)$-tail and that the linear equation obtained after
	affine subtraction has order
	\[
	sp-p+2=1+\beta>1.
	\] 
	The first main result  Theorem~\ref{thm:mean-gradient} holds throughout the range
    \eqref{eq:gradient-parameter-range}. The additional assumption $sp>n$
    is needed only in the second theorem~\ref{thm:wolff}.

	For $x_0\in\Omega$ and $r>0$,  we define
	\[
	D_\mu(x_0,r)
	:=\left(\frac{|\mu|(B_r(x_0))}{r^{n-\beta}}\right)^{1/(p-1)},
	\]
	and 
	\[
	\Wpot_{\gamma,p}^{|\mu|}(x_0,R)
	:=\int_0^R
	\left(\frac{|\mu|(B_\rho(x_0))}{\rho^{n-\gamma p}}\right)^{1/(p-1)}
	\frac{\diff\rho}{\rho}
	=\int_0^R D_\mu(x_0,\rho)\frac{\diff\rho}{\rho}.
	\]
	This potential has the scaling of a first derivative. Indeed, if
	$\widetilde u(z)=[u(x_0+rz)-c]/r$, then
	\[
	(-\Delta_p)^s\widetilde u=\widetilde\mu,\qquad
	\widetilde\mu(E)=r^{\beta-n}\mu(x_0+rE),\qquad
	D_{\widetilde\mu}(0,\rho)=D_\mu(x_0,r\rho).
	\]
	
	\subsection{Brief background and main difficulties}
	
Nonlinear potential theory provides a direct link
between the data of a PDE and the pointwise behavior of  its solutions. It is a central tool in regularity theory, measure data problems, and the study of fine properties of solutions.  We refer to
    \cite{HeinonenKilpelainenMartio2006,KimLeeLee2023,KuusiMingione2014}.
	
	For local quasilinear equations, pointwise potential estimates go back
    to the work of Kilpel\"{a}inen and Mal\'{y}~\cite{KilpelainenMaly1994}.
    See also \cite{KilpelainenMaly1992}. They proved Wolff potential estimates for solutions of
	equations of $p$-Laplace type with nonnegative measure data. A different
	approach, which also applies to subelliptic operators, was developed by
	Trudinger and Wang \cite{TrudingerWang2002}, while signed measure data were treated
	in \cite{TrudingerWang2009}. Related results on solutions with measure data and on
	the relation between different notions of solution can be found in
	\cite{KilpelainenKuusiTuholaKujanpaa2011,KorteKuusi2010}. For equations with Orlicz growth, we refer to
	\cite{ChlebickaGiannettiZatorskaGoldstein2024,LeeLee2021}.  At the gradient level, Mingione \cite{Mingione2011} established pointwise Riesz potential estimates for uniformly elliptic quasilinear equations with quadratic growth. For possibly degenerate   equation with $p$-growth, $p\ge 2$,  Duzaar and Mingione
\cite{DuzaarMingione2011} obtained pointwise gradient estimates in terms of Wolff potentials.  Related estimates in terms of
	Riesz potentials  were obtained in \cite{DuzaarMingione2010,KuusiMingione2013}. The singular case was further developed in \cite{DongZhu2022,DongZhu2024,NguyenPhuc2020,NguyenPhuc2023,XuZhao2026}. Extensions to  systems, obstacle problems, and equations with Orlicz growth can be found in \cite{Baroni2015,ByunSongYoun2023,ChlebickaKimWeidner2026,Scheven2012,XiongZhangMa2026}.
	
	In the nonlocal setting, Kuusi, Mingione and Sire
    \cite{KuusiMingioneSire2015} initiated the potential theory for equations
    with measure data. They introduced solutions obtained as limits of
    approximations (SOLA) for nonlinear nonlocal equations and proved
    pointwise Wolff potential estimates for these solutions.    Kim, Lee and Lee \cite{KimLeeLee2023} subsequently established the Wiener criterion for  nonlocal equations with standard $p$-growth, and later extended these results to equations with Orlicz growth~\cite{KimLeeLee2025}.  More recently, Nguyen, Ok and Song  \cite{NguyenOkSong2026} extended the existence and Wolff potential estimates for
	SOLA to the strongly singular range $ 1<p\le 2-\frac{s}{n}$.

	The results above concern pointwise estimates for the solution itself.
	Potential estimates at the gradient level are less developed in the
	nonlocal setting. For a class of linear nonlocal equations with measure
	data, Kuusi, Nowak and Sire \cite{KuusiNowakSire2024} established
	gradient regularity and pointwise estimates in terms of Riesz
	potentials. In the nonlinear setting, Diening, Kim, Lee and Nowak
	\cite{DieningKimLeeNowak2025c,DieningKimLeeNowak2025a} developed a
	potential theory for gradients of solutions to a class of nonlinear
	nonlocal equations with measure data. Their structural assumptions
	require a globally Lipschitz and strongly monotone nonlinearity, see
	\cite[Assumption~1.1]{DieningKimLeeNowak2025c}. These assumptions
	exclude the power nonlinearity
	\[
	J_p(t)=|t|^{p-2}t,\qquad p\ne2.
	\]
	Thus their results do not cover the degenerate fractional
	\(p\)-Laplacian considered here.
	
	For the fractional \(p\)-Laplacian, local boundedness, Harnack
	inequalities, and H\"older continuity were established in
	\cite{DiCastroKuusiPalatucci2014,DiCastroKuusiPalatucci2016,
		IannizzottoMosconiSquassina2016}. Higher Sobolev and H\"older
	regularity, as well as higher differentiability, were subsequently
	obtained in
	\cite{BogeleinDuzaarLiaoMolicaBisciServadei2025,
		BrascoLindgren2017,BrascoLindgrenSchikorra2018,
		DieningKimLeeNowak2025b,Schikorra2016}.
	Diening and Nowak \cite{DieningNowak2025} developed
	Calder\'on--Zygmund estimates for fractional \(p\)-Laplace type
	equations with VMO coefficients, yielding higher fractional
	differentiability and integrability of weak solutions.
	B\"ogelein, Duzaar, Liao and Moring
	\cite{BogeleinDuzaarLiaoMoring2025} established
	Calder\'on--Zygmund type estimates at the gradient level for the
	fractional \(p\)-Poisson equation. In particular, for
	\(s>(p-1)/p\), they proved
	\[
	f\in L_{\mathrm{loc}}^{qp/(p-1)}
	\quad\Longrightarrow\quad
	\nabla u\in L_{\mathrm{loc}}^{qp},
	\qquad q>1.
	\]
	Their subsequent work \cite{BogeleinDuzaarLiaoMoring2026} gives sharp
	gradient integrability results. For bounded right hand sides, local
	Lipschitz regularity was proved in \cite{BiswasTopp2025}.
	Giovagnoli, Jesus and Silvestre
	\cite{GiovagnoliJesusSilvestre2025} proved interior
	\(C^{1,\alpha}\) regularity for fractional \(p\)-harmonic functions in
	the range
	\[
	2\le p<\frac{2}{1-s}.
	\]
	This condition is equivalent to \(sp>p-2\). Under
	\eqref{eq:gradient-parameter-range}, we have
	\(sp>p-1>p-2\), and hence their homogeneous regularity result applies
	in our setting.
	
	Despite these developments, pointwise gradient potential estimates
	for the fractional \(p\)-Laplacian with measure data have remained
	open. This question was explicitly raised by Diening and Nowak
	\cite[Section~1.5]{DieningNowak2025}. Diening, Kim, Lee and Nowak
	\cite[Section~1.5]{DieningKimLeeNowak2025c} also listed the extension
	of their gradient potential estimates to nonlinearities with
	\(p\)-growth among their open questions, explicitly including the
	fractional \(p\)-Laplacian for \(p\ne2\).
	
	The present paper addresses this question for SOLA with finite signed
	measure data in the superquadratic range \(p>2\) and \(sp>p-1\).
	Theorem~\ref{thm:mean-gradient} establishes local
	\(W^{1,p-1}\) regularity and a pointwise Wolff potential estimate at
	every Lebesgue point of the weak gradient. An affine expansion in
	\(L^{p-1}\)-average at every point of finite potential is proved in
	Lemma~\ref{lem:mean-affine-expansion}. Under the additional condition
	\(sp>n\), Theorem~\ref{thm:wolff} gives Fr\'echet differentiability at
	every such point.

	The main difficulty is to prove affine decay for fractional
    $p$-harmonic functions with constants independent of the minimizing
    affine slope. The affine approximation changes with the scale, and
    its slope is not known a priori. The decay estimate must therefore
    remain uniform with respect to this slope. After affine subtraction, the remainder
	satisfies a linear nonlocal equation whose kernel may vanish in some
	directions. Thus a pointwise lower bound on the kernel is unavailable.
	
	\subsection{Main results}
	For a measurable function $w$, we define
	\[
	\Tail_{m,sp}(w;x_0,R)
	:=\left(R^{sp}\int_{\R^n\setminus B_R(x_0)}
	\frac{|w(y)|^m}{|y-x_0|^{n+sp}}\,\diff y\right)^{1/m},
	\qquad m=p-1,
	\]
and introduce 
	\begin{equation*}
		\mathcal A_m(u;x_0,R)
		:=\frac1R\left(\fint_{B_R(x_0)}
		|u-(u)_{B_R(x_0)}|^m\,\diff x\right)^{1/m}
		+\frac1R\Tail_{m,sp}
		\bigl(u-(u)_{B_R(x_0)};x_0,R\bigr).
	\end{equation*}
	Here $(u)_B=\fint_Bu\,\diff x$. The exponent $m=p-1$ is also used in the comparison estimate and the affine excess.

The first theorem concerns the integrability and pointwise bounds of the weak gradient.
\begin{theorem}[Weak gradient potential estimate for SOLA]
    \label{thm:mean-gradient}
    Assume \eqref{eq:gradient-parameter-range}. Let
    $\mu\in\mathcal M(\Omega)$ and let $u$ be a SOLA of
    \eqref{eq:main} in the sense of Definition~\ref{def:sola}. Then
    \begin{equation}\label{eq:weak-gradient-membership}
        u\in W^{1,p-1}_{\mathrm{loc}}(\Omega).
    \end{equation}
    For every Lebesgue point $x_0$ of $\nabla u$ and every $R>0$ such
    that $B_{2R}(x_0)\Subset\Omega$,
    \begin{equation}\label{eq:mean-gradient-bound}
        |\nabla u(x_0)|
        \le C\left[\mathcal A_m(u;x_0,2R)
        +\Wpot_{\gamma,p}^{|\mu|}(x_0,2R)\right],
    \end{equation}
    where $C=C(n,p,s)$. In particular, the estimate holds for almost
    every $x_0\in\Omega$.
\end{theorem}
Here $\nabla u(x_0)$ denotes the Lebesgue value of $\nabla u$ at $x_0$, that is,
\[
\lim_{r\to0}\fint_{B_r(x_0)}|\nabla u(x)-\nabla u(x_0)|\,\diff x=0.
\]
We  emphasise that \eqref{eq:weak-gradient-membership} is part of the
conclusion and is not assumed in the definition of SOLA. We prove it by
estimating the gradients of mollifications. The argument uses
Lemma~\ref{lem:mean-affine-expansion}, which gives an affine approximation
of $u$ in $L^{p-1}$ at every point where the Wolff potential is finite.

\begin{theorem}[Classical differentiability in the Morrey range]
    \label{thm:wolff}
    Let $u$ be as in Theorem~\ref{thm:mean-gradient} and assume in
    addition that $sp>n$. Then $u$ has a unique continuous
    representative on $\Omega$. If $B_{2R}(x_0)\Subset\Omega$ and
    $\Wpot_{\gamma,p}^{|\mu|}(x_0,2R)<\infty$, this representative is
    Fr\'echet differentiable at $x_0$ and
    \begin{equation}\label{eq:wolff-gradient}
        |\nabla u(x_0)|\le C\left[
        \mathcal A_m(u;x_0,2R)+\Wpot_{\gamma,p}^{|\mu|}(x_0,2R)\right],
    \end{equation}
    where $C=C(n,p,s)$. Here $\nabla u(x_0)$ denotes the classical
    gradient. At Lebesgue points of the weak gradient, the two values
    agree.
\end{theorem}
The second theorem uses an $L^\infty$ comparison to strengthen the averaged expansion obtained in Section~\ref{sec:mean-gradient}. The additional
condition $sp>n$ is used only for this classical conclusion.

\subsection{Strategy of proof}
Our argument combines comparison estimates with a
potential iteration in the spirit of Duzaar and Mingione
\cite{DuzaarMingione2011}. Our use of affine approximation is inspired by the work of
Kuusi, Nowak and Sire \cite{KuusiNowakSire2024}
on linear nonlocal equations.  Diening, Kim, Lee and Nowak
\cite[Section~1.4]{DieningKimLeeNowak2025c} explain the difficulty
of obtaining suitable regularity estimates for affine remainders
in the nonlinear nonlocal setting.  For the fractional $p$-Laplacian, we obtain the required control of
the affine remainder by treating the controlled and large slope
regimes separately.

We first consider the homogeneous equation. The key step is the decay estimate for
the uniform affine excess $\Psi_\infty$ defined in
\eqref{eq:uniform-affine-excess} and established in
Lemma~\ref{lem:full-affine-decay-high}, which gives
\[
\Psi_\infty(v;x_0,\rho)
\le C\left(\frac\rho r\right)^\sigma\Psi_\infty(v;x_0,r),
\qquad 0<\rho\le r/32,
\]
with constants independent of the slope of the minimizing affine function.
In the controlled slope regime, this follows from the homogeneous
$C^{1,\alpha}$ estimate of Giovagnoli, Jesus and Silvestre
\cite[Theorem~1.1]{GiovagnoliJesusSilvestre2025}. In the large slope
regime, exact affine subtraction gives a linear equation for the
remainder. After normalization, we compare its kernel with the model
of order $1+\beta$ given by
\[
(p-1)|e\cdot\omega|^{p-2}|x-y|^{-n-(1+\beta)}.
\]
Here $e$ is the direction of the affine slope and
$\omega=(x-y)/|x-y|$. We verify the integrated ellipticity and
regularity assumptions of the divergence form Schauder estimate of
Fern\'andez-Real and Ros-Oton
\cite[Theorem~1.4]{FernandezRealRosOton2024}. This gives the required
bound in the large slope regime. A local interpolation argument then
transfers the uniform decay to the averaged excess $\Psi_m$ in Lemma~\ref{lem:mean-homogeneous-decay}.

We next combine the decay of $\Psi_m$ with the $L^{p-1}$
comparison between the energy approximations and their fractional
$p$-harmonic replacements. Passing to the SOLA limit, Lemma~\ref{lem:mean-one-step} gives
\[
\Psi_m(u;x_0,\theta r)
\le \frac12\Psi_m(u;x_0,r)+CD_\mu(x_0,2r)
\]
for a fixed $\theta\in(0,1)$.
At each point where the Wolff potential is finite, iteration
gives an affine expansion in averaged $L^{p-1}$ and a bound
for its slope $a_*(x_0)$, as proved in
Lemma~\ref{lem:mean-affine-expansion}.

The existence of the weak gradient still has to be proved.
In Lemma~\ref{lem:sola-weak-gradient}, we mollify $u$ and show
that $\nabla u_\varepsilon(x_0)$ converges to $a_*(x_0)$ at
points of finite potential. The affine estimates and the
local $L^{p-1}$ integrability of the potential, established in
Lemma~\ref{lem:wolff-local-Lm}, also bound these gradients by
a fixed $L^{p-1}$ function on each compact subset, see
\eqref{eq:mollified-gradient-domination}.
This allows us to pass to the limit in the integration by
parts formula and obtain the weak gradient of $u$.
At its Lebesgue points where the potential is finite, this
gradient agrees with $a_*$.
The bound for $a_*$ in \eqref{eq:mean-all-scale-bound} then gives
Theorem~\ref{thm:mean-gradient}.

Under the additional condition $sp>n$,
Lemma~\ref{lem:sola-local-weak} shows that the SOLA has a
continuous representative which is a local weak solution.
The $L^\infty$ comparison in
Lemma~\ref{lem:uniform-replacement-comparison} therefore
applies to $u$ and gives \eqref{eq:sola-mean-to-uniform}.
Applying this estimate to the same affine approximations
gives a uniform first-order expansion at every point where
the potential is finite. The continuous representative of
$u$ is therefore Fr\'echet differentiable at these points,
as stated in Theorem~\ref{thm:wolff}.

The paper is organized as follows. In
Section~\ref{sec:preliminaries},  we give some preliminary lemmas. 
Section~\ref{sec:harmonic-replacement} is devoted to  the fractional
$p$-harmonic replacement.  In Section~\ref{sec:affine-excess},  we	establish the homogeneous affine decay.  Section~\ref{sec:mean-gradient}
passes the averaged recurrence to SOLA and proves
Theorem~\ref{thm:mean-gradient}. Section~\ref{sec:classical-gradient}
contains the additional argument in the Morrey range and proves
Theorem~\ref{thm:wolff}.

	\section{Preliminaries}
	\label{sec:preliminaries}
	
	We fix the notation and collect the auxiliary results used in the proofs.
	Throughout the paper, $C>0$ denotes a constant that may change from line
	to line, with its dependence indicated when relevant. We set
	\[
	(f)_{B_R(x_0)}
	:=
	\fint_{B_R(x_0)}f(y)\,\diff y.
	\]
	
	\subsection{Function spaces and notation}
	For the basic properties and embeddings of fractional Sobolev spaces,
	we refer to \cite{DiNezzaPalatucciValdinoci2012}.
	
	For an open set $U\subset\R^n$ and $s\in(0,1)$, $p\in(1,\infty)$, let
	\[
	[w]_{W^{s,p}(U)}^p:=\iint_{U\times U}\frac{|w(x)-w(y)|^p}{|x-y|^{n+sp}}\,\diff x\diff y.
	\]
	We define
	\[
	L_{sp}^{p-1}(\R^n):=\left\{w\in L^{p-1}_{\mathrm{loc}}(\R^n):
	\int_{\R^n}\frac{|w(y)|^{p-1}}{(1+|y|)^{n+sp}}\,\diff y<\infty\right\}.
	\]
	For an open set $U\subset\R^n$, set
	\[
	C_c^\infty(U)
	:=\{\varphi\in C_c^\infty(\R^n):
	\supp\varphi\Subset U\},
	\qquad
	W^{s,p}_0(U)
	:=\overline{C_c^\infty(U)}^{\,W^{s,p}(\R^n)}.
	\]
	We also set
	\[
	X_0^{s,p}(U)
	:=
	\{w\in W^{s,p}(\R^n):
	w=0\ \text{a.e. in }\R^n\setminus U\}.
	\]
	In general,  $
	W_0^{s,p}(U)\subset X_0^{s,p}(U)$. If $B$ is a ball, then
	\[
	W_0^{s,p}(B)=X_0^{s,p}(B),
	\]
	as follows by a dilation about the centre followed by mollification.	
	We use this identification for all local comparison balls.	In addition, we define
	\[
	p':=\frac p{p-1},\qquad J_p(t):=|t|^{p-2}t,\qquad \delta u(x,y):=u(x)-u(y).
	\]

	\subsection{Energy approximations and SOLA}
	We first define the energy weak solutions used for the approximating problems in the definition of SOLA.
	\begin{definition}[Energy weak solution]\label{def:weak}
		Let $f\in(W^{s,p}_0(\Omega))'$. An energy weak solution of the zero
		exterior problem is a function $u\in X_0^{s,p}(\Omega)$ satisfying
		\begin{equation*}
			\iint_{\R^n\times\R^n}
			\frac{J_p(\delta u(x,y))\delta\varphi(x,y)}{|x-y|^{n+sp}}
			\,\diff x\diff y
			=\langle f,\varphi\rangle
			\qquad\text{for every }\varphi\in C_c^\infty(\Omega).
		\end{equation*}
	\end{definition}
A density argument shows that the same identity holds for every
\(\varphi\in W_0^{s,p}(\Omega)\).

We specialize \cite[Definitions~1 and~2]{KuusiMingioneSire2015}
to zero exterior data, with $g_j=g=0$. Define
	\[
	\overline q:=\min\left\{\frac{n(p-1)}{n-s},p\right\}.
	\]
	\begin{definition}[SOLA]\label{def:sola}
		Let $\mu\in\mathcal M(\Omega)$ be extended by zero to $\R^n$.
		A function $u$, equal to zero almost everywhere outside $\Omega$,
		is a SOLA of \eqref{eq:main} if the following conditions hold.
		For some $h\in(0,s)$ and $q\in[m,\overline q)$, one has
		$u\in W^{h,q}(\Omega)$, and $u$ solves the distributional identity
		\begin{equation*}
			\iint_{\R^n\times\R^n}
			\frac{J_p(\delta u(x,y))\delta\varphi(x,y)}{|x-y|^{n+sp}}
			\,\diff x\diff y
			=\int_\Omega\varphi\,\diff\mu
			\qquad\text{for every }\varphi\in C_c^\infty(\Omega).
		\end{equation*}
In addition, there is a sequence
$\{u_j\}\subset W^{s,p}(\R^n)$  solving the approximate Dirichlet problems
 \begin{equation*}
	\begin{cases}
		(-\Delta_p)^s u_j=\mu_j & \text{in }\Omega,\\
		u_j=0 & \text{in }\R^n\setminus\Omega,
	\end{cases}
\end{equation*}
in the sense of Definition~\ref{def:weak}, such that
\begin{equation}\label{eq:sola-convergence}
	u_j\to u\quad\text{a.e. in }\R^n
	\quad\text{and locally in }L^q(\R^n).
\end{equation}
Here $\mu_j\in C_c^\infty(\R^n)$ converge weakly to $\mu$ in the sense
of measures in $\Omega$ and satisfy
\begin{equation}\label{eq:sola-measure-control}
\limsup_{j\to\infty}|\mu_j|(B)\le|\mu|(\overline B)
\qquad\text{for every ball }B\subset\R^n.
\end{equation}
\end{definition}
Since $p>2$, we have $m<\overline q$ and $p>2-s/n$.
Theorem~1.1 of \cite{KuusiMingioneSire2015} therefore yields a
SOLA for every finite signed measure $\mu$. The solution constructed
there belongs to $W^{h,q}(\Omega)$ for every $0<h<s$ and
$m\le q<\overline q$.  Since $q\ge m$, \eqref{eq:sola-convergence} gives strong $L^m$
convergence on every bounded subset of $\R^n$. The boundedness of
$\Omega$ and the common zero exterior values then imply
	\begin{equation}\label{eq:sola-tail-class}
		u_j\to u\text{ in }L^m(\R^n),\qquad
		u\in L_{sp}^m(\R^n).
	\end{equation}
	For SOLA, the comparison estimates are first proved for the
	energy approximations \(u_j\) and then passed to the limit.

	For homogeneous equations, we use the following local notion of solution.
	\begin{definition}[Local weak solution]\label{def:local-weak}
		Let $B\subset\R^n$ be a ball. A function
		\[
		v\in W^{s,p}_{\mathrm{loc}}(B)\cap L_{sp}^{p-1}(\R^n)
		\]
		is a local weak solution of $(-\Delta_p)^sv=0$ in $B$ if
		\[
		\iint_{\R^n\times\R^n}
		\frac{J_p(\delta v(x,y))\delta\varphi(x,y)}{|x-y|^{n+sp}}
		\,\diff x\diff y=0
		\qquad\text{for every }\varphi\in C_c^\infty(B).
		\]
	\end{definition}
	
	\subsection{Algebraic inequalities}
	
	We recall the following inequalities from \cite[Lemma~2.2]{BogeleinDuzaarLiaoMoring2025}.
	\begin{lemma}\label{lem:algebra}
		Let \(p\ge2\). There exists \(C=C(p)\ge1\) such that, for every
		\(a,b\in\R\),
		\begin{equation*}
			C^{-1}(|a|+|b|)^{p-2}|a-b|^2
			\le
			\bigl(J_p(a)-J_p(b)\bigr)(a-b)
			\le
			C(|a|+|b|)^{p-2}|a-b|^2.
		\end{equation*}
		In particular,
		\begin{equation}\label{eq:pcoercive}
			|a-b|^p
			\le
			C\bigl(J_p(a)-J_p(b)\bigr)(a-b).
		\end{equation}
	\end{lemma}

	\subsection{Poincar\'e and tail estimates}
	
	\begin{lemma}\label{lem:fractional-poincare}
		Let $0<s<1$, $1\le p<\infty$, and $B_r=B_r(x_0)\subset\R^n$. For every
		$w\in W^{s,p}(B_r)$,
		\begin{equation*}
			\|w-(w)_{B_r}\|_{L^p(B_r)}
			\le Cr^s[w]_{W^{s,p}(B_r)}.
		\end{equation*}
		If $w\in W^{s,p}_0(B_r)$, then, with $C=C(n,p,s)$, we have 
		\begin{equation*}
			\|w\|_{L^p(B_r)}
			\le Cr^s[w]_{W^{s,p}(\R^n)}.
		\end{equation*}
	\end{lemma}
	\begin{proof}
		Jensen's inequality and $|x-y|\le2r$ on $B_r\times B_r$ give
		\[
		\int_{B_r}|w-(w)_{B_r}|^p\,\diff x
		\le\frac1{|B_r|}\iint_{B_r\times B_r}|w(x)-w(y)|^p\,\diff x\diff y
		\le Cr^{sp}[w]_{W^{s,p}(B_r)}^p.
		\]
		If $w\in W^{s,p}_0(B_r(x_0))$, use its zero extension and integrate
		 over $y\in B_{3r}(x_0)\setminus B_{2r}(x_0)$, we  obtain
		\[
		[w]_{W^{s,p}(\R^n)}^p
		\ge\int_{B_r(x_0)}|w(x)|^p
		\int_{B_{3r}(x_0)\setminus B_{2r}(x_0)}
		|x-y|^{-n-sp}\,\diff y\diff x
		\ge cr^{-sp}\|w\|_{L^p(B_r(x_0))}^p.
		\]
		Taking the $p$-th root proves both estimates.
	\end{proof}

	The next lemma combines a local Taylor estimate with a tail bound to obtain affine decay at smaller radii.
	
	\begin{lemma}\label{lem:taylor-affine-tail}
		Let $p>2$, $s\in(0,1)$, and $\beta:=sp-p+1>0$.
		Let $g:\R^n\to\R$ be measurable and let $L$ be affine. Suppose that for
		some $\alpha\in(0,1]$ and $H\ge1$,
		\begin{equation}\label{eq:abstract-taylor-bound}
			|g(x)-L(x)|\le H|x|^{1+\alpha}\quad \text{for}~x\in B_{1/4}~\text{and}\quad \int_{\R^n\setminus B_{1/4}}
			\frac{|g(y)-L(y)|^{p-1}}{|y|^{n+sp}}\,\diff y
			\le H^{p-1}.
		\end{equation}
	Take $
		\sigma:=\frac12\min\left\{\alpha,\frac\beta{p-1}\right\}>0 $.
		Then there exists $C=C(n,p,s,\alpha)$ such that
		\begin{equation}\label{eq:abstract-affine-tail-decay}
			\frac1\rho\|g-L\|_{L^\infty(B_\rho )}
			+\frac1\rho\Tail_{p-1,sp}(g-L;0,\rho)
			\le CH\rho^\sigma
			\qquad(0<\rho\le1/16).
		\end{equation}
	\end{lemma}
	
	\begin{proof}
		Set $h:=g-L$ and $m:=p-1$. Then $sp=m+\beta$. By \eqref{eq:abstract-taylor-bound},
		\[
		\frac1\rho\|h\|_{L^\infty(B_\rho)}
		\le H\rho^\alpha
		\le H\rho^\sigma,
		\qquad 0<\rho\le\frac1{16}.
		\]
		For the tail, split
		\[
		\int_{\R^n\setminus B_\rho}
		\frac{|h(y)|^m}{|y|^{n+sp}}\,\diff y
		=
		\int_{B_{1/4}\setminus B_\rho}
		\frac{|h(y)|^m}{|y|^{n+sp}}\,\diff y
		+
		\int_{\R^n\setminus B_{1/4}}
		\frac{|h(y)|^m}{|y|^{n+sp}}\,\diff y.
		\]
		The choice of $\sigma$ ensures that $\sigma\le\alpha$ and $m\sigma<\beta$.
        Hence \eqref{eq:abstract-taylor-bound} gives
		\begin{equation*}
			\begin{split}
				\frac1\rho\Tail_{p-1,sp}(h;0,\rho)
				&\le	\rho^{\beta/m}
				\left(
				\int_{B_{1/4}\setminus B_\rho}
				\frac{|h(y)|^m}{|y|^{n+sp}}\,\diff y
				\right)^{1/m}
				+
				H\rho^{\beta/m}\\
				&\le  C	\rho^{\beta/m}\left( H^m \int_\rho^{1/4} t^{\alpha m-\beta-1}\,\diff t	\right)^{1/m}+
				H\rho^{\beta/m}\\
				&\le C	\rho^{\beta/m} H
				\left( \int_\rho^{1/4}
				t^{m\sigma-\beta-1}\,\diff t	\right)^{1/m}+
				H\rho^{\beta/m}\\
				&\le
				CH\rho^\sigma
				+
				H\rho^{\beta/m}
				\le
				CH\rho^\sigma.
			\end{split}
		\end{equation*}
		Together with the local bound this gives \eqref{eq:abstract-affine-tail-decay}.
	\end{proof}
	
	The following lemma compares	$\sum_{k=0}^\infty D_\mu(x_0,\theta^kR)$ with the corresponding Wolff
	potential.
	
	\begin{lemma}\label{lem:dyadic-potential}
		Assume \eqref{eq:gradient-parameter-range}, let $\theta\in(0,1/2)$,
		and set $r_k=\theta^kR$. Then
		\begin{equation}\label{eq:dyadic-wolff}
			C^{-1}\sum_{k=1}^\infty D_\mu(x_0,r_k)
			\le\Wpot_{\gamma,p}^{|\mu|}(x_0,R)
			\le C\sum_{k=0}^\infty D_\mu(x_0,r_k).
		\end{equation}
		If $B_{2R}(x_0)\Subset\Omega$, then also
		\begin{equation}\label{eq:full-dyadic-wolff}
			\sum_{k=0}^\infty D_\mu(x_0,r_k)
			\le C\Wpot_{\gamma,p}^{|\mu|}(x_0,2R),
		\end{equation}
		Here $C=C(n,p,s,\theta)$.
	\end{lemma}
	\begin{proof}
		Take $a=n-\beta>0$ and $m=p-1$. For $r_{k+1}\le\rho\le r_k$,
		monotonicity of the measure gives
		\[
		\theta^aD_\mu(x_0,r_{k+1})^m
		\le D_\mu(x_0,\rho)^m
		\le\theta^{-a}D_\mu(x_0,r_k)^m.
		\]
		Integration against $\diff\rho/\rho$ and summation prove
		\eqref{eq:dyadic-wolff}. For $R\le\rho\le2R$, the same argument gives
		$D_\mu(x_0,R)\le2^{a/m}D_\mu(x_0,\rho)$. Thus
		\[
		D_\mu(x_0,R)
		\le\frac{2^{a/m}}{\log2}\int_R^{2R}
		D_\mu(x_0,\rho)\frac{\diff\rho}{\rho}.
		\]
		Together with the first inequality in \eqref{eq:dyadic-wolff}, this
		proves \eqref{eq:full-dyadic-wolff}.
		
	\end{proof}
	
	\subsection{A divergence form Schauder estimate}
	\label{subsec:preliminary-schauder}
	
	We recall the divergence form Schauder estimate used in the large slope argument. We first state the kernel assumptions and the corresponding weak formulation.
	
	Let $2\mathfrak s\in(0,2)$ and let $K:\R^n\times\R^n\setminus\{x=y\}\to[0,\infty)$ be a measurable kernel satisfying
	\begin{equation*}
		K(x,y)=K(y,x)\qquad\text{for a.e. }(x,y)\in\R^n\times\R^n.
	\end{equation*}
	Given $\alpha\in(0,1]$ and $\lambda,\Lambda,M>0$,  we define
	\[
	K\in\mathscr K(2\mathfrak s,\alpha;\lambda,\Lambda,M)
	\]
	if, for every $x,h\in\R^n$, $\rho>0$, and $\xi\in\R^n$,
	\begin{align}
		\rho^{2\mathfrak s}
		\int_{B_{2\rho}(x)\setminus B_\rho(x)}K(x,y)\,\diff y
		&\le \Lambda,
		\label{eq:linear-kernel-upper}\\
		\rho^{2\mathfrak s-2}
		\int_{B_\rho(x)}|\xi\cdot(x-y)|^2K(x,y)\,\diff y
		&\ge \lambda|\xi|^2,
		\label{eq:linear-kernel-lower}\\
		\int_{B_{2\rho}(x)\setminus B_\rho(x)}
		|K(x+h,y+h)-K(x,y)|\,\diff y
		&\le M|h|^\alpha\rho^{-2\mathfrak s},
		\label{eq:linear-kernel-translation}\\
		\int_{B_{2\rho}\setminus B_\rho}
		|K(x,x+z)-K(x,x-z)|\,\diff z
		&\le M\rho^{\alpha-2\mathfrak s}.
		\label{eq:linear-kernel-almost-even}
	\end{align}
	For a bounded open set $D\subset\R^n$, we set
	\[
	\mathcal Q(D):=(\R^n\times\R^n)
	\setminus\bigl((\R^n\setminus D)\times(\R^n\setminus D)\bigr)
	\]
	and define
	\[
	\mathcal H_K(D):=
	\left\{W:\
	\iint_{\mathcal Q(D)}|W(x)-W(y)|^2K(x,y)\,\diff x\diff y<\infty
	\right\}.
	\]
	Whenever the integral is absolutely convergent, we write
	\[
	\mathcal E_K(V,\varphi)
	:=\iint_{\R^n\times\R^n}
	\delta V(x,y)\delta\varphi(x,y)K(x,y)\,\diff x\diff y.
	\]
	For $W\in\mathcal H_K(D)$ and $F\in L^1_{\mathrm{loc}}(D)$,
	we say that $\mathcal L_KW=F$ weakly in $D$ if
	\[
	\mathcal E_K(W,\varphi)=\int_D F\varphi\,\diff x
	\qquad\text{for every }\varphi\in C_c^\infty(D).
	\]
	
	The following a priori estimate is the scaled form of
	\cite[Theorem~1.4]{FernandezRealRosOton2024} for absolutely continuous
	symmetric kernels. We take the target regularity exponent in that
	reference to be $1+\alpha$ and its auxiliary parameter to be
	$\varepsilon=\alpha/2$. Since $1+\alpha<2\mathfrak s$, the additional
	kernel assumption required there for exponents above $2\mathfrak s$
	is not needed. 
	
	The data also satisfy the integrability condition in
	\cite[Definition~4.1]{FernandezRealRosOton2024}, since
	\[
	q_\alpha=\frac n{2\mathfrak s-1-\alpha}
	>\frac n{2\mathfrak s}>\frac{2n}{n+2\mathfrak s}.
	\]

	\begin{lemma}[Theorem~1.4,~\cite{FernandezRealRosOton2024}]
		\label{lem:fixed-radius-schauder}
		Let $n>2\mathfrak s$, $2\mathfrak s\in(1,2)$, and
		\[
		0<\alpha<2\mathfrak s-1,
		\qquad
		q_\alpha
		:=
		\frac{n}{2\mathfrak s-(1+\alpha)}.
		\]
		Suppose	$ K\in\mathscr K (2\mathfrak s,\alpha;\lambda,\Lambda,M)$.	Let $W\in
		\mathcal H_K(B_{1/2})
		\cap C_{\mathrm{loc}}^{1,\alpha}(B_{1/2})
		\cap L^\infty(\R^n)$
		be a weak solution of
		\[
		\mathcal L_KW=F
		\qquad\text{in }B_{1/2},
		\qquad
		F\in L^{q_\alpha}(B_{1/2}).
		\]
		Then, with $C=C(n,\mathfrak s,\alpha,\lambda,\Lambda,M)$,
		\begin{equation*}
			\|W\|_{C^{1,\alpha}(B_{1/4})}
			\le
			C\left(
			\|W\|_{L^\infty(\R^n)}
			+
			\|F\|_{L^{q_\alpha}(B_{1/2})}
			\right).
		\end{equation*}
	\end{lemma}
	
	\subsection{Interior regularity for fractional \texorpdfstring{$p$}{p}-harmonic functions}
	\label{subsec:preliminary-homogeneous-C1a}
	
	The following interior regularity estimate will be used for
    fractional $p$-harmonic functions. 
	
	\begin{lemma}[Theorem~1.1,~\cite{GiovagnoliJesusSilvestre2025}]
		\label{lem:fixed-radius-homogeneous-C1a}
		Let $s\in(0,1)$ and $2\le p<\frac{2}{1-s}$. Let
		$U$	be a local weak solution of
		\[
		(-\Delta_p)^sU=0
		\qquad\text{in }B_2.
		\]
		Then there exist $\alpha\in(0,1)$ and $C\ge1$, depending only on $n,p,s$, such that $U\in C^{1,\alpha}(B_1)$ and
		\begin{equation*}
			\|U\|_{C^{1,\alpha}(B_1)}
			\le
			C\left[
			\|U\|_{L^\infty(B_2)}
			+
			\left(
			\int_{\R^n\setminus B_2}
			\frac{|U(z)|^{p-1}}{|z|^{n+sp}}\,\diff z
			\right)^{1/(p-1)}
			\right].
		\end{equation*}
	\end{lemma}
	
	\section{Harmonic replacement and comparison estimates}
	\label{sec:harmonic-replacement}
	
	We introduce the fractional $p$-harmonic replacement and present its
	interior estimates. We then derive the equation for an affine
	remainder and prove the comparison estimate in $L^{p-1}$.
	
	For $1<p<\infty$, $s\in(0,1)$, $u\in W^{s,p}(\R^n)$, and
	$B_r(x_0)\Subset\Omega$,   we set
	\[
	\mathcal X_u(B_r(x_0))
	:=
	u+W^{s,p}_0(B_r(x_0))
	\]
	and
	\[
	\mathscr E(w)
	:=
	\frac1p
	\iint_{\R^n\times\R^n}
	\frac{|\delta w(x,y)|^p}{|x-y|^{n+sp}}
	\,\diff x\,\diff y.
	\]
	For a minimizing sequence $v_j=u+w_j$, the triangle inequality for
	the Gagliardo seminorm and Lemma~\ref{lem:fractional-poincare} bound
	$w_j$ in $W^{s,p}_0(B_r(x_0))$. Weak compactness and lower
	semicontinuity therefore give a minimizer. Strict convexity in the
	differences, together with the fixed exterior values, gives uniqueness.
	Thus there is a unique minimizer
	\[
	v\in\mathcal X_u(B_r(x_0))
	\]
	of \(\mathscr E\). We call \(v\) the fractional \(p\)-harmonic
	replacement of \(u\) in \(B_r(x_0)\). In particular,  it satisfies
	\begin{equation}\label{eq:harmonic-replacement}
		\begin{cases}
			(-\Delta_p)^s v=0
			&\text{weakly in }B_r(x_0),\\
			v=u
			&\text{a.e. in }\R^n\setminus B_r(x_0).
		\end{cases}
	\end{equation}
	
	\subsection{Homogeneous estimates}
	
	We begin with a local bound for the harmonic replacements.
	\begin{lemma}
		\label{lem:local-boundedness}
		Let $
		v\in W^{s,p}(\R^n)$
		be a weak solution of \eqref{eq:harmonic-replacement}.
		Then, for every $c\in\R$, the following estimate holds with $C=C(n,p,s)$,
		\begin{equation}\label{eq:local-boundedness}
			\|v-c\|_{L^\infty(B_{r/2}(x_0))}
			\le C\left[
			\left(\fint_{B_r(x_0)}|v-c|^p\,\diff x\right)^{1/p}
			+\Tail_{p-1,sp}(v-c;x_0,r)
			\right].
		\end{equation}
	\end{lemma}
	
	\begin{proof}
		We set $
		m:=p-1$, $
		z_+:=(v-c)_+$ and $ z_-:=(c-v)_+$.
		The proof of \cite[Theorem~1.1]{DiCastroKuusiPalatucci2016} allows an
        arbitrary truncation level. Applying that estimate to $v$ at level $c$ gives
		\begin{equation*}
			\|z_+\|_{L^\infty(B_{r/2}(x_0))}
			\le \Tail_{p-1,sp}(z_+;x_0,r/2)
			+C\left(\fint_{B_r(x_0)}z_+^p\,\diff x\right)^{1/p}.
		\end{equation*}
		The same estimate applied to $-v$ at level $-c$ gives the analogous bound for $z_-$. Thus, 	for either choice $z=z_+$ or $z=z_-$, we have
		\begin{align*}
			&\Tail_{p-1,sp}(z;x_0,r/2)^m\\
			&\quad =\left(\frac r2\right)^{sp}
			\int_{B_r(x_0)\setminus B_{r/2}(x_0)}
			\frac{|z(y)|^m}{|y-x_0|^{n+sp}}\,\diff y+
			\left(\frac r2\right)^{sp}
			\int_{\R^n\setminus B_r(x_0)}
			\frac{|z(y)|^m}{|y-x_0|^{n+sp}}\,\diff y\\
			&\quad \le C r^{-n}\int_{B_r(x_0)}|v-c|^m\,\diff y
			+2^{-sp}\Tail_{p-1,sp}(v-c;x_0,r)^m\\
			&\quad \le C\left(\fint_{B_r(x_0)}|v-c|^p\,\diff y\right)^{m/p}
			+2^{-sp}\Tail_{p-1,sp}(v-c;x_0,r)^m.
		\end{align*}
		The estimate \eqref{eq:local-boundedness} follows because
		\[
		\|v-c\|_{L^\infty(B_{r/2}(x_0))}
		=\max\left\{
		\|z_+\|_{L^\infty(B_{r/2}(x_0))},
		\|z_-\|_{L^\infty(B_{r/2}(x_0))}
		\right\}.
		\]
	\end{proof}
	
	The condition $sp>p-1$ implies $p(1-s)<1<2$, so the range in Lemma~\ref{lem:fixed-radius-homogeneous-C1a} is satisfied. Scaling that result gives the following interior estimate.
	
	\begin{lemma}
		\label{lem:homogeneous-C1a}
		Assume $p>2$ and $sp>p-1$. Suppose that
		$
		v\in
		W_{\mathrm{loc}}^{s,p}(B_r(x_0))
		\cap
		L_{sp}^{p-1}(\R^n)
		\cap
		L^\infty(B_r(x_0))$
		is a local weak solution of $(-\Delta_p)^s v=0$ in $B_r(x_0)$. We define
		\[
		\mathcal H(v;x_0,r)
		:=
		\frac1r
		\essosc_{B_r(x_0)}v
		+
		\frac1r
		\Tail_{p-1,sp}
		\bigl(
		v-(v)_{B_r(x_0)};x_0,r
		\bigr).
		\]
		Then there exist \(\alpha\in(0,1)\) and \(C_h\ge1\),
		depending only on \(n,p,s\), such that \(v\) has a representative in
		\(C^{1,\alpha}(B_{r/2}(x_0))\) and
		\begin{equation}\label{eq:homogeneous-full-C1a}
			\|\nabla v\|_{L^\infty(B_{r/2}(x_0))}
			+
			r^\alpha
			[\nabla v]_{C^{0,\alpha}(B_{r/2}(x_0))}
			\le
			C_h\mathcal H(v;x_0,r).
		\end{equation}
	\end{lemma}
	\begin{proof}
		We define
		\[
		c:=(v)_{B_r(x_0)},
		\qquad
		r_0:=\frac r2,
		\qquad
		U(z):=\frac{v(x_0+r_0z)-c}{r}.
		\]
		Scaling shows that $U$ is a local weak solution of
		\[
		(-\Delta_p)^sU=0
		\qquad\text{in }B_2.
		\]
		Moreover, we have
		\[
		\|U\|_{L^\infty(B_2)}
		=\frac1r\|v-c\|_{L^\infty(B_r(x_0))}
		\le\frac1r\essosc_{B_r(x_0)}v.
		\]
		Since
		$\R^n\setminus B_2$ is mapped onto
		$\R^n\setminus B_r(x_0)$ by $y=x_0+r_0z$,  we have
		\begin{align*}
			\left(
			\int_{\R^n\setminus B_2}
			\frac{|U(z)|^{p-1}}{|z|^{n+sp}}\,\diff z
			\right)^{1/(p-1)}
			&=
			\frac{r_0^{sp/(p-1)}}r
			\left(
			\int_{\R^n\setminus B_r(x_0)}
			\frac{|v(y)-c|^{p-1}}{|y-x_0|^{n+sp}}\,\diff y
			\right)^{1/(p-1)}\\
			&=
			2^{-sp/(p-1)}\frac1r
			\Tail_{p-1,sp}(v-c;x_0,r).
		\end{align*}
		Applying Lemma~\ref{lem:fixed-radius-homogeneous-C1a}, we obtain
		\[
		\|\nabla U\|_{L^\infty(B_1)}
		+[\nabla U]_{C^{0,\alpha}(B_1)}
		\le C\mathcal H(v;x_0,r).
		\]
		Rescaling gives \eqref{eq:homogeneous-full-C1a}, since
		\[
		\|\nabla U\|_{L^\infty(B_1)}
		=\frac12
		\|\nabla v\|_{L^\infty(B_{r/2}(x_0))},\quad \text{and} \quad
		[\nabla U]_{C^{0,\alpha}(B_1)}
		=\frac{r^\alpha}{2^{1+\alpha}}
		[\nabla v]_{C^{0,\alpha}(B_{r/2}(x_0))}.\]
	\end{proof}
	
	\subsection{Exact affine subtraction}
	
	Subtracting an affine function leads to a linear equation for the
    remainder. The next lemma gives its coefficient through the
    fundamental theorem of calculus.
	
	\begin{lemma}
		\label{lem:affine-linearization}
		Assume $p>2$ and $sp>p-1$. Let $v\in W_{\mathrm{loc}}^{s,p}(B)
		\cap L_{sp}^{p-1}(\R^n)$ be a local weak solution of $(-\Delta_p)^s v=0$ in $B$.  For
		\[
		\ell(x):=b+a\cdot x,
		\qquad
		w:=v-\ell,
		\]
		one has
		\[
		w\in W_{\mathrm{loc}}^{s,p}(B)
		\cap L_{sp}^{p-1}(\R^n),
		\]
		and, for every $\varphi\in C_c^\infty(B)$,
		\begin{equation}\label{eq:affine-linearization}
			\iint_{\R^n\times\R^n}
			\frac{\mathbb A_{a,w}(x,y)\,\delta w(x,y)\,\delta\varphi(x,y)}
			{|x-y|^{n+sp}}\,\diff x\diff y=0,
		\end{equation}
		where
		\begin{equation*}
			\mathbb A_{a,w}(x,y)
			:=(p-1)\int_0^1
			|a\cdot(x-y)+t\delta w(x,y)|^{p-2}\,\diff t.
		\end{equation*}
		The integral in \eqref{eq:affine-linearization} is absolutely convergent and
		\[
		\mathbb A_{a,w}(x,y)=\mathbb A_{a,w}(y,x)\ge0.
		\]
	\end{lemma}
	\begin{proof}
		Let $U\Subset\R^n$ be bounded. Since $p-sp-1=-\beta$, we have
		\begin{align*}
			[\ell]_{W^{s,p}(U)}^p
			=\iint_{U\times U}
			\frac{|a\cdot(x-y)|^p}{|x-y|^{n+sp}}\,\diff x\diff y
			\le C|U|\,|a|^p
			\int_0^{\operatorname{diam}U}r^{p-sp-1}\,\diff r<\infty,
		\end{align*}
		while for the weighted $L^{p-1}$-norm we estimate
		\begin{align*}
			&\int_{\R^n}
			\frac{|\ell(y)|^{p-1}}{(1+|y|)^{n+sp}}\,\diff y\\
			&\quad\le C\int_{B_1}|\ell(y)|^{p-1}\,\diff y
			+C|b|^{p-1}\int_1^\infty r^{-1-sp}\,\diff r
			+C|a|^{p-1}\int_1^\infty r^{p-sp-2}\,\diff r\\
			&\quad=C\int_{B_1}|\ell(y)|^{p-1}\,\diff y
			+\frac{C|b|^{p-1}}{sp}
			+C|a|^{p-1}\int_1^\infty r^{-1-\beta}\,\diff r<\infty,
		\end{align*}
		so that
		\[
		\ell\in W_{\mathrm{loc}}^{s,p}(\R^n)
		\cap L_{sp}^{p-1}(\R^n),
		\qquad
		w\in W_{\mathrm{loc}}^{s,p}(B)
		\cap L_{sp}^{p-1}(\R^n).
		\]
		
		Fix $\varphi\in C_c^\infty(B)$, set $K:=\supp\varphi$, and choose $U$ with $K\Subset U\Subset B$.  Let $d:=\dist(K,\R^n\setminus U)>0$. H\"older's inequality gives
		\begin{align*}
			\iint_{U\times U}
			\frac{|\delta v(x,y)|^{p-1}|\delta\varphi(x,y)|}
			{|x-y|^{n+sp}}\,\diff x\diff y
			\le
			[v]_{W^{s,p}(U)}^{p-1}
			[\varphi]_{W^{s,p}(U)}<\infty,
		\end{align*}
		whereas for $x\in K$ and $y\in\R^n\setminus U$ we have
		\[
		|x-y|^{-n-sp}
		\le C_{K,U}(1+|y|)^{-n-sp},
		\]
		and consequently
		\begin{align*}
			\int_K\int_{\R^n\setminus U}
			\frac{|\delta v(x,y)|^{p-1}|\varphi(x)|}
			{|x-y|^{n+sp}}\,\diff y\diff x
			\le C_{\varphi,K,U}
			\left[
			\|v\|_{L^{p-1}(K)}^{p-1}
			+\int_{\R^n}
			\frac{|v(y)|^{p-1}}{(1+|y|)^{n+sp}}\,\diff y
			\right]<\infty.
		\end{align*}
		Since $\delta\varphi=0$ on $(\R^n\setminus K)\times(\R^n\setminus K)$, we conclude that
		\begin{equation}\label{eq:v-weak-form-absolute}
			\iint_{\R^n\times\R^n}
			\frac{|J_p(\delta v)\delta\varphi|}
			{|x-y|^{n+sp}}\,\diff x\diff y<\infty.
		\end{equation}
		
		Now consider the affine part. The estimates
		\[
		|J_p(\delta\ell)|
		\le |a|^{p-1}|x-y|^{p-1},
		\qquad
		|\delta\varphi|
		\le\min\left\{
		\|\nabla\varphi\|_{L^\infty}|x-y|,
		2\|\varphi\|_{L^\infty}
		\right\}
		\]
		lead  to
		\begin{align*}
			&\iint_{\R^n\times\R^n}
			\frac{|J_p(\delta\ell)\delta\varphi|}
			{|x-y|^{n+sp}}\,\diff x\diff y\\
			&\quad\le C|K|\,|a|^{p-1}
			\left[
			\|\nabla\varphi\|_{L^\infty}
			\int_0^1r^{p-sp-1}\,\diff r
			+\|\varphi\|_{L^\infty}
			\int_1^\infty r^{p-sp-2}\,\diff r
			\right]<\infty.
		\end{align*}
		
		For $\varepsilon>0$,  we define
		\[
		\mathcal E_\varepsilon(\ell,\varphi)
		:=\iint_{\{|x-y|>\varepsilon\}}
		\frac{J_p(a\cdot(x-y))\delta\varphi(x,y)}
		{|x-y|^{n+sp}}\,\diff x\diff y.
		\]
		For each fixed $\varepsilon>0$, the inner integral is absolutely convergent, since
		\[
		\int_{|h|>\varepsilon}
		\frac{|J_p(-a\cdot h)|}{|h|^{n+sp}}\,\diff h
		\le C|a|^{p-1}\int_\varepsilon^\infty r^{-1-\beta}\,\diff r
		<\infty.
		\]
		Interchanging $x$ and $y$ in the term containing $\varphi(y)$ gives
		\begin{align*}
			\mathcal E_\varepsilon(\ell,\varphi)
			&=2\int_{\R^n}\varphi(x)
			\int_{\R^n\setminus B_\varepsilon(x)}
			\frac{J_p(a\cdot(x-y))}{|x-y|^{n+sp}}\,\diff y\diff x\\
			&=2\int_{\R^n}\varphi(x)
			\int_{\{|h|>\varepsilon\}}
			\frac{J_p(-a\cdot h)}{|h|^{n+sp}}\,\diff h\diff x=0,
		\end{align*}
		so that dominated convergence yields
		\begin{equation}\label{eq:affine-weak-harmonic}
			\iint_{\R^n\times\R^n}
			\frac{J_p(\delta\ell)\delta\varphi}
			{|x-y|^{n+sp}}\,\diff x\diff y
			=\lim_{\varepsilon\to 0}
			\mathcal E_\varepsilon(\ell,\varphi)=0.
		\end{equation}
		Subtracting \eqref{eq:affine-weak-harmonic} from the weak formulation for $v$, we obtain
		\begin{equation}\label{eq:subtracted-affine-equations}
			\iint_{\R^n\times\R^n}
			\frac{[J_p(\delta v)-J_p(\delta\ell)]\delta\varphi}
			{|x-y|^{n+sp}}\,\diff x\diff y=0.
		\end{equation}
		
		Finally, for $\eta,\zeta\in\R$, we have the elementary identity
		\[
		J_p(\eta+\zeta)-J_p(\eta)
		=(p-1)\zeta\int_0^1|\eta+t\zeta|^{p-2}\,\diff t.
		\]
		Taking
		\[
		\eta:=\delta\ell(x,y)=a\cdot(x-y),
		\qquad
		\zeta:=\delta w(x,y),
		\]
		we obtain
		\[
		J_p(\delta v)-J_p(\delta\ell)
		=\mathbb A_{a,w}(x,y)\delta w(x,y).
		\]
		From \eqref{eq:v-weak-form-absolute} and the absolute convergence of the affine form, it follows that
		\begin{align*}
			\iint_{\R^n\times\R^n}
			\frac{|\mathbb A_{a,w}(x,y)\delta w(x,y)
				\delta\varphi(x,y)|}{|x-y|^{n+sp}}\,\diff x\diff y
			\le
			\iint_{\R^n\times\R^n}
			\frac{(|J_p(\delta v)|+|J_p(\delta\ell)|)|\delta\varphi|}
			{|x-y|^{n+sp}}\,\diff x\diff y<\infty,
		\end{align*}
		so that \eqref{eq:affine-linearization} follows from \eqref{eq:subtracted-affine-equations}. Finally, since
		\[
		\mathbb A_{a,w}(y,x)
		=(p-1)\int_0^1
		|-a\cdot(x-y)-t\delta w(x,y)|^{p-2}\,\diff t,
		\]
		we have $\mathbb A_{a,w}(y,x)=\mathbb A_{a,w}(x,y)\ge0$.
	\end{proof}
	
	\subsection{Comparison in \texorpdfstring{$L^{p-1}$}{L(p-1)}}
	The next estimate uses only monotonicity and the shared exterior values. 
	\begin{lemma}\label{lem:mean-measure-comparison}
		Assume \eqref{eq:gradient-parameter-range}. Let
		$u\in W^{s,p}(\R^n)$ solve $(-\Delta_p)^su=\mu$ in $B_r(x_0)$,
		where $\mu$ is a finite signed Radon measure. Let $v$ be its
		fractional $p$-harmonic replacement in that ball and set $w=u-v$.
		Then
		\begin{equation}\label{eq:mean-measure-comparison}
			\frac1r\left(\fint_{B_r(x_0)}|w|^m\,\diff x\right)^{1/m}
			\le CD_\mu(x_0,r),\qquad m=p-1.
		\end{equation}
		Moreover, for $0<\rho\le r$,
		\begin{equation}\label{eq:mean-measure-comparison-small}
			\frac1\rho\left(\fint_{B_\rho(x_0)}|w|^m\,\diff x\right)^{1/m}
			+\frac1\rho\Tail_{m,sp}(w;x_0,\rho)
			\le C\left(\frac r\rho\right)^{1+n/m}D_\mu(x_0,r).
		\end{equation}
		The constants depend only on $n,p,s$.
	\end{lemma}
	\begin{proof}
		Write $B_r=B_r(x_0)$ and $M=|\mu|(B_r)$. For $k>0$, set
		\[
		T_k(t)=\max\{-k,\min\{t,k\}\},\qquad S_k(t)=\frac{T_k(t)}k.
		\]
		Since \(w=u-v\in W^{s,p}_0(B_r)\), we have
		\[
		S_k(w)\in W^{s,p}_0(B_r),
		\qquad
		|S_k(w)|\le1.
		\]
		
		We first approximate $S_k(w)$ by smooth functions while preserving
        its bound. Set $g=S_k(w)$ and, for $0<\lambda<1$, define
        \[
        g_\lambda(x)=g\left(x_0+\frac{x-x_0}{\lambda}\right).
        \]
        Then $g_\lambda=0$ outside $B_{\lambda r}(x_0)$,
        $|g_\lambda|\le1$, and $g_\lambda\to g$ in
        $W^{s,p}(\R^n)$ as $\lambda\to1$.
        To see the last convergence, approximate $g$ in the global
        $W^{s,p}$ norm by smooth compactly supported functions and use
        the continuity and uniform boundedness of these dilations for
        $\lambda$ near one. Convolve $g_\lambda$ with a nonnegative
        mollifier whose support radius is less than $(1-\lambda)r/2$.
        Choosing the convolution radii along a sequence
        $\lambda_j\to1$ gives
        \[
        \varphi_j\in C_c^\infty(B_r),\qquad
        |\varphi_j|\le1,\qquad
        \varphi_j\to S_k(w)\quad\text{in }W^{s,p}(\R^n).
        \]
        Subtracting the weak formulations for \(u\) and \(v\) and using
		\(\varphi_j\) as a test function gives
		\begin{equation}\label{eq:bounded-test-approximation}
		\iint_{\R^n\times\R^n}
		\frac{
			[J_p(\delta u)-J_p(\delta v)]
			\delta\varphi_j
		}{|x-y|^{n+sp}}
		\,\diff x\diff y
		=
		\int_{B_r}\varphi_j\,\diff\mu .
		\end{equation}
		Since \(|\varphi_j|\le1\),
		\[
		\left|
		\int_{B_r}\varphi_j\,\diff\mu
		\right|
		\le M.
		\]
		Passing to the limit in the left hand side by H\"older's
		inequality and \eqref{eq:bounded-test-approximation}, we obtain
		\begin{equation}\label{eq:bounded-comparison-test}
			\iint_{\R^n\times\R^n}
			\frac{
				[J_p(\delta u)-J_p(\delta v)]
				\delta S_k(w)
			}{|x-y|^{n+sp}}
			\,\diff x\diff y
			\le M.
		\end{equation}
		
        The integrand in \eqref{eq:bounded-comparison-test} is
		nonnegative.  Indeed,
		\[
		[J_p(a)-J_p(b)](a-b)\ge0
		\]
		for all \(a,b\in\R\), while \(S_k\) is nondecreasing.
		Hence we may retain only the pairs
		\[
		x\in B_r,\qquad y\in\R^n\setminus B_r.
		\]
		For such pairs,
		\[
		w(y)=0,
		\qquad
		\delta u(x,y)-\delta v(x,y)=w(x),
		\qquad
		\delta S_k(w)(x,y)=S_k(w(x)).
		\]
		Moreover, \eqref{eq:pcoercive} implies
		\[
		[J_p(a)-J_p(b)]
		\operatorname{sign}(a-b)
		\ge c|a-b|^{p-1}.
		\]
		Therefore,
		\[
		\begin{aligned}
			\relax[J_p(\delta u(x,y))-J_p(\delta v(x,y))]
			S_k(w(x))\ge
			c|w(x)|^m
			\min\left\{\frac{|w(x)|}{k},1\right\}.
		\end{aligned}
		\]
		On the other hand, uniformly for \(x\in B_r\),
		\[
		\begin{aligned}
			\int_{\R^n\setminus B_r}
			\frac{\diff y}{|x-y|^{n+sp}}
\ge
			\int_{B_{3r}(x_0)\setminus B_{2r}(x_0)}
			\frac{\diff y}{|x-y|^{n+sp}}\ge cr^{-sp}.
		\end{aligned}
		\]
		Consequently, \eqref{eq:bounded-comparison-test} gives
		\[
		r^{-sp}
		\int_{B_r}
		|w|^m
		\min\left\{\frac{|w|}{k},1\right\}
		\,\diff x
		\le CM.
		\]
		Letting \(k\to0\) and applying the monotone convergence
		theorem, we obtain
		\begin{equation}\label{eq:comparison-exterior-mass}
			\int_{B_r}|w|^m\,\diff x
			\le
			Cr^{sp}|\mu|(B_r).
		\end{equation}
		
		Since \(sp=m+\beta\), it follows that
		\[
		\begin{aligned}
			\frac1r
			\left(
			\fint_{B_r}|w|^m\,\diff x
			\right)^{1/m}
			\le
			C
			\left(
			\frac{|\mu|(B_r)}
			{r^{n-\beta}}
			\right)^{1/m}=
			CD_\mu(x_0,r),
		\end{aligned}
		\]
		which proves \eqref{eq:mean-measure-comparison}.

		Finally, let \(0<\rho\le r\). Since \(B_\rho(x_0)\subset B_r(x_0)\)
		and \(w=0\) a.e. in \(\R^n\setminus B_r(x_0)\), we have
		\[
		\begin{aligned}
			\left(\fint_{B_\rho(x_0)}|w|^m\,\diff x\right)^{1/m}
			+\Tail_{m,sp}(w;x_0,\rho)\le
			C\rho^{-n/m}
			\left(\int_{B_r(x_0)}|w|^m\,\diff x\right)^{1/m}.
		\end{aligned}
		\]
		Hence, by \eqref{eq:comparison-exterior-mass} and
		\(sp=m+\beta\),
		\[
		\begin{aligned}
			\frac1\rho
			\left(\fint_{B_\rho(x_0)}|w|^m\,\diff x\right)^{1/m}
			+\frac1\rho\Tail_{m,sp}(w;x_0,\rho)&\le
			C\rho^{-1-n/m}r^{sp/m}|\mu|(B_r)^{1/m}\\
			&=
			C\left(\frac r\rho\right)^{1+n/m}
			D_\mu(x_0,r),
		\end{aligned}
		\]
		which proves \eqref{eq:mean-measure-comparison-small}.
	\end{proof}

	\section{Homogeneous affine approximation}\label{sec:affine-excess}
	Throughout this section, we assume \eqref{eq:gradient-parameter-range}
    and  take $m=p-1$. We first establish affine decay for
    $\Psi_\infty$, with constants independent of the slope of the
    minimizing affine function. We then derive the corresponding estimates
    for  $\Psi_m$, which will be used in
    the comparison argument for measure data.
	
	\subsection{The two affine excesses}
	For $\ell_{a,b}(y)=b+a\cdot(y-x_0)$, define
	\begin{equation}\label{eq:mean-affine-excess}
		\Psi_m(u;x_0,r):=\inf_{a\in\R^n,\,b\in\R}
		\left[\frac1r\left(\fint_{B_r(x_0)}|u-\ell_{a,b}|^m
		\,\diff x\right)^{1/m}
		+\frac1r\Tail_{m,sp}(u-\ell_{a,b};x_0,r)\right]
	\end{equation}
	and
	\begin{equation}\label{eq:uniform-affine-excess}
		\Psi_\infty(u;x_0,r):=\inf_{a\in\R^n,\,b\in\R}
		\left[\frac1r\|u-\ell_{a,b}\|_{L^\infty(B_r(x_0))}
		+\frac1r\Tail_{m,sp}(u-\ell_{a,b};x_0,r)\right].
	\end{equation}
	All affine functions have finite tail because $sp>m$. In particular,
	$\Psi_m$ is finite for $u\in L_{sp}^m(\R^n)$, and $\Psi_\infty$
	is finite when $u$ is also bounded in $B_r(x_0)$. We always have
	$\Psi_m\le\Psi_\infty$.
	
	\begin{lemma}\label{lem:affine-minimizer}
		For $u\in L_{sp}^m(\R^n)$, the infimum defining $\Psi_m(u;x_0,r)$
		is attained. If also $u\in L^\infty(B_r(x_0))$, the infimum defining
		$\Psi_\infty(u;x_0,r)$ is attained. Every affine function
		$L(x)=b+a\cdot(x-x_0)$ satisfies
		\begin{equation}\label{eq:mean-affine-norm}
			|b|+r|a|\le C(n,m)
			\left(\fint_{B_r(x_0)}|L|^m\,\diff x\right)^{1/m}.
		\end{equation}
	\end{lemma}	
	\begin{proof}
		We first prove \eqref{eq:mean-affine-norm}. Since
		\[
		\fint_{B_r(x_0)}(x-x_0)\,\diff x=0,
		\]
		we have
		\[
		b=\fint_{B_r(x_0)}L\,\diff x,
		\]
		and hence
		\[
		|b|
		\le
		\left(
		\fint_{B_r(x_0)}|L|^m\,\diff x
		\right)^{1/m}.
		\]
		Moreover, by a change of variables and rotational symmetry,
		\[
		\left(
		\fint_{B_r(x_0)}
		|a\cdot(x-x_0)|^m\,\diff x
		\right)^{1/m}
		=c_{n,m}r|a|,
		\]
		where \(c_{n,m}>0\). Since
		\[
		a\cdot(x-x_0)=L(x)-b,
		\]
		Minkowski's inequality gives
		\[
		r|a|
		\le
		C
		\left[
		\left(
		\fint_{B_r(x_0)}|L|^m\,\diff x
		\right)^{1/m}
		+|b|
		\right].
		\]
		Combining the last two estimates proves \eqref{eq:mean-affine-norm}.
		
		We next consider \(\Psi_m\). Define
		\[
		\mathcal J_m(\ell)
		:=
		\frac1r
		\left(
		\fint_{B_r(x_0)}|u-\ell|^m\,\diff x
		\right)^{1/m}
		+
		\frac1r
		\Tail_{m,sp}(u-\ell;x_0,r).
		\]
		Let
		\[
		\ell_j(x)=b_j+a_j\cdot(x-x_0)
		\]
		be a minimizing sequence. We may assume
		\[
		\mathcal J_m(\ell_j)
		\le
		\Psi_m(u;x_0,r)+1.
		\]
		Then
		\[
		\left(
		\fint_{B_r(x_0)}
		|u-\ell_j|^m\,\diff x
		\right)^{1/m}
		\le
		r\bigl(\Psi_m(u;x_0,r)+1\bigr).
		\]
		By Minkowski's inequality,
		\[
		\begin{aligned}
			\left(
			\fint_{B_r(x_0)}
			|\ell_j|^m\,\diff x
			\right)^{1/m}
			\le
			\left(
			\fint_{B_r(x_0)}
			|u|^m\,\diff x
			\right)^{1/m}+
			r\bigl(\Psi_m(u;x_0,r)+1\bigr).
		\end{aligned}
		\]
		Hence \eqref{eq:mean-affine-norm} gives
		\[
		\sup_j\bigl(|b_j|+r|a_j|\bigr)<\infty.
		\]
		After passing to a subsequence,
		\[
		a_j\to a,
		\qquad
		b_j\to b.
		\]
	Take
		\[
		\ell(x):=b+a\cdot(x-x_0).
		\]
		Then
		\[
		\|\ell_j-\ell\|_{L^\infty(B_r(x_0))}
		\le
		|b_j-b|+r|a_j-a|
		\longrightarrow0.
		\]
		Consequently, by the reverse triangle inequality, we obtain 
		\[
		\left(
		\fint_{B_r(x_0)}
		|u-\ell_j|^m\,\diff x
		\right)^{1/m}
		\longrightarrow
		\left(
		\fint_{B_r(x_0)}
		|u-\ell|^m\,\diff x
		\right)^{1/m}.
		\]
		Since \(\ell_j(y)\to\ell(y)\) for every \(y\in\R^n\),
		Fatou's lemma yields
		\[
		\Tail_{m,sp}(u-\ell;x_0,r)
		\le
		\liminf_{j\to\infty}
		\Tail_{m,sp}(u-\ell_j;x_0,r).
		\]
		Therefore
		\[
		\mathcal J_m(\ell)
		\le
		\liminf_{j\to\infty}
		\mathcal J_m(\ell_j)
		=
		\Psi_m(u;x_0,r).
		\]
		Since \(\Psi_m(u;x_0,r)\) is the infimum over all affine
		functions, the reverse inequality is immediate. Hence
		\[
		\mathcal J_m(\ell)=\Psi_m(u;x_0,r),
		\]
		and \(\ell\) is a minimizer.
		
		For \(\Psi_\infty\), let \(\ell_j\) be a minimizing sequence.
		Since \(u\in L^\infty(B_r(x_0))\),
		\[
		\|\ell_j\|_{L^\infty(B_r(x_0))}
		\le
		\|u\|_{L^\infty(B_r(x_0))}
		+
		\|u-\ell_j\|_{L^\infty(B_r(x_0))}
		\le C.
		\]
		Hence \eqref{eq:mean-affine-norm} again gives
		\[
		\sup_j\bigl(|b_j|+r|a_j|\bigr)<\infty.
		\]
		After passage to a convergent subsequence of the coefficients, the affine functions converge uniformly in the ball. The local \(L^\infty\) term therefore converges, and the tail is lower semicontinuous by Fatou's lemma. Thus the infimum defining
		\(\Psi_\infty(u;x_0,r)\) is also attained.
	\end{proof}

	We present the scaling used below. Let \(E>0\), \(b\in\mathbb R\), and define
	\[
	V(z):=\frac{v(x_0+rz)-b}{rE}.
	\]
	For an affine function \(\ell\), define its rescaling by
	\[
	\widehat\ell(z)
	:=
	\frac{\ell(x_0+rz)-b}{rE}.
	\]
	A change of variables in the local term and in the tail therefore gives
	\begin{equation}\label{eq:affine-excess-scaling}
		\Psi_q(V;0,t)
		=
		E^{-1}\Psi_q(v;x_0,rt),
		\qquad
		q\in\{m,\infty\},\quad t>0.
	\end{equation}
	The homogeneous equation is preserved by the same change of variables.
	
	In particular, if
	\[
	\ell_{a,b}(y)=b+a\cdot(y-x_0)
	\]
	is a minimizing affine function for \(\Psi_q(v;x_0,r)\), then
	\[
	\widehat\ell(z)
	=
	\frac{\ell_{a,b}(x_0+rz)-b}{rE}
	=
	\frac{a}{E}\cdot z
	\]
	is a minimizing affine function for \(\Psi_q(V;0,1)\). Hence, if
	\(E=\Psi_q(v;x_0,r)\), then
	\[
	\Psi_q(V;0,1)=1.
	\]
	
	\subsection{The controlled slope regime}
	\begin{lemma}[Affine decay in the controlled slope regime]
		\label{lem:controlled-affine-decay}
		Assume $p>2$ and $sp>p-1$.  Let
		\[
		v\in W^{s,p}_{\mathrm{loc}}(B_{2r}(x_0))
		\cap L_{sp}^{p-1}(\R^n)\cap L^\infty(B_{2r}(x_0))
		\]
		be a local weak solution of $(-\Delta_p)^sv=0$ in $B_{2r}(x_0)$.
		Let $M\ge1$, set $E:=\Psi_\infty(v;x_0,r)$, and let
		$\ell_{a,b}$ be a minimizing affine function in
		\eqref{eq:uniform-affine-excess}. If
		\[
		|a|\le M E,
		\]
		then there are $\sigma=\sigma(n,p,s)\in(0,1)$ and
		$C_M=C(n,p,s,M)\ge1$ such that
		\begin{equation}\label{eq:controlled-affine-decay}
			\Psi_\infty(v;x_0,\rho)
			\le C_M\left(\frac\rho r\right)^\sigma
			E
			\qquad(0<\rho\le r/16).
		\end{equation}
	\end{lemma}

	\begin{proof}
		Let $E=\Psi_\infty(v;x_0,r)$. If $E=0$, a minimizing affine
		function agrees with $v$ almost everywhere, and the conclusion follows. Suppose $E>0$.
		Using \eqref{eq:affine-excess-scaling} with a minimizing affine
		function \(\ell_{a,b}\), we may assume $x_0=0$, $r=E=1$ and
		\[
		v=a\cdot x+z,\qquad |a|\le M,\qquad
		\|z\|_{L^\infty(B_1)}+\Tail_{m,sp}(z;0,1)=1.
		\]
		Since $(v)_{B_1}=(z)_{B_1}$, we have
		\[
		\|v\|_{L^\infty(B_1)}\le M+1,\qquad |(v)_{B_1}|\le1.
		\]
		Moreover,
		\[
		\Tail_{m,sp}(a\cdot x;0,1)^m
		\le C|a|^m\int_1^\infty t^{m-sp-1}\,\diff t
		=\frac{C|a|^m}{\beta}.
		\]
		Minkowski's inequality therefore gives
		$\mathcal H(v;0,1)\le C_M$. Lemma~\ref{lem:homogeneous-C1a}
		yields
		\begin{equation}\label{eq:controlled-C1a-bound}
			\|\nabla v\|_{L^\infty(B_{1/2})}
			+[\nabla v]_{C^{0,\alpha}(B_{1/2})}\le C_M.
		\end{equation}
		Set $L(x)=v(0)+\nabla v(0)\cdot x$. Taylor's formula and
		\eqref{eq:controlled-C1a-bound} give
		\[
		|v(x)-L(x)|
		\le |x|\int_0^1|\nabla v(tx)-\nabla v(0)|\,\diff t
		\le C_M|x|^{1+\alpha},\qquad x\in B_{1/4}.
		\]
		On $B_1\setminus B_{1/4}$, both $v$ and $L$ are bounded by $C_M$.
		On $\R^n\setminus B_1$, use
		\[
		v(y)-L(y)=z(y)-v(0)+(a-\nabla v(0))\cdot y.
		\]
		The tail of $z$ and the bounds $|v(0)|\le1$,
		$|a-\nabla v(0)|\le C_M$ imply
		\[
		\int_{\R^n\setminus B_{1/4}}
		\frac{|v(y)-L(y)|^m}{|y|^{n+sp}}\,\diff y
		\le C_M\left(1+\int_1^\infty
		(t^{-1-sp}+t^{-1-\beta})\,\diff t\right)\le C_M.
		\]
		Lemma~\ref{lem:taylor-affine-tail}, with
		$\sigma=\tfrac12\min\{\alpha,\beta/m\}$, now gives
		\[
		\Psi_\infty(v;0,\rho)\le C_M\rho^\sigma,
		\qquad 0<\rho\le1/16.
		\]
		Rescale by \eqref{eq:affine-excess-scaling} to obtain
		\eqref{eq:controlled-affine-decay}.
	\end{proof}
	
	\subsection{The kernel \texorpdfstring{$K_e^0$}{Ke0} and its extension}
	\label{subsec:linearized-kernel}
	We turn to the large slope regime.  Set
	\[
	A:=|a|,
	\qquad
	e:=\frac{a}{A}\in\mathbb S^{n-1}.
	\]
	Lemma~\ref{lem:affine-linearization} gives a linear equation for the affine
	remainder. Multiplying that equation by $A^{2-p}$ yields a rescaled
	coefficient kernel whose value at the affine background is $K_e^0$. We first verify
	the structural estimates for $K_e^0$ with constants independent of the
	direction $e$. We then extend small local perturbations of $K_e^0$ to globally
	defined kernels to which Lemma~\ref{lem:fixed-radius-schauder} applies.
	
	The comparison of the rescaled coefficient with $K_e^0$ uses the
	following inequality.
	
	\begin{lemma}
		\label{lem:power-modulus}
		Let $p>2$ and $\kappa:=\min\{1,p-2\}$.  For every $L\ge1$,  there is
		$C=C(p,L)$ such that
		\begin{equation}\label{eq:power-modulus}
			\bigl||\tau_1|^{p-2}-|\tau_2|^{p-2}\bigr|
			\le C|\tau_1-\tau_2|^\kappa
			\qquad(\tau_1,\tau_2\in[-L,L]).
		\end{equation}
	\end{lemma}
	
	\begin{proof}
		Set \(d:=p-2>0\). For \(0<d\le1\),
		\[
		\bigl||\tau_1|^d-|\tau_2|^d\bigr|
		\le
		\bigl||\tau_1|-|\tau_2|\bigr|^d
		\le
		|\tau_1-\tau_2|^d,
		\]
		while for \(d\ge1\), the mean value theorem on \([0,L]\) yields
		\[
		\bigl||\tau_1|^d-|\tau_2|^d\bigr|
		\le
		dL^{d-1}|\tau_1-\tau_2|.
		\]
		Hence \eqref{eq:power-modulus} follows with
		\(\kappa=\min\{1,d\}\).
	\end{proof}
	From the definition of $\mathbb A_{a,w}$,  we have
	\[
	\mathbb A_{a,0}(x,y)=(p-1)|a\cdot(x-y)|^{p-2}.
	\]
	Therefore
	\begin{align*}
		A^{2-p}\frac{\mathbb A_{a,0}(x,y)}{|x-y|^{n+sp}}
		&=(p-1)
		\left|e\cdot\frac{x-y}{|x-y|}\right|^{p-2}
		|x-y|^{-n-(sp-p+2)}.
	\end{align*}
	The right hand side is the kernel $K_e^0$ considered next.
	
	\begin{lemma}
		\label{lem:affine-model-kernel}
		Let $n\ge2$, $p>2$, $s\in(0,1)$, and $sp>p-1$. Set
		\[
		\beta:=sp-p+1,
		\qquad
		2\mathfrak s:=1+\beta=sp-p+2.
		\]
		Then $\beta\in(0,1)$ and $2\mathfrak s\in(1,2)$. In particular,
		$n>2\mathfrak s$. For
		$e\in\mathbb S^{n-1}$, define
		\begin{equation}\label{eq:affine-model-kernel}
			K_e^0(x,y)
			:=(p-1)
			\left|
			e\cdot\frac{x-y}{|x-y|}
			\right|^{p-2}
			|x-y|^{-n-2\mathfrak s},
			\qquad x\ne y.
		\end{equation}
		Then $K_e^0$ is a nonnegative measurable symmetric kernel satisfying
		\[
		0\le K_e^0(x,y)
		\le(p-1)|x-y|^{-n-2\mathfrak s},
		\qquad x\ne y,
		\]
		and
		\[
		K_e^0(x+h,y+h)=K_e^0(x,y),\quad
		K_e^0(x,x+z)=K_e^0(x,x-z)\]
		for every $x,y,h,z\in\mathbb R^n$ with $x\ne y$ and $z\ne0$.
		Moreover, there exist
		\[
		\lambda_0=\lambda_0(n,p,s)>0,
		\qquad
		\Lambda_0=\Lambda_0(n,p,s)>0,
		\]
		independent of $e$, such that, for every $\alpha\in(0,1]$,
		\[
		K_e^0\in
		\mathscr K(2\mathfrak s,\alpha;\lambda_0,\Lambda_0,1).
		\]
	\end{lemma}
	
	\begin{proof}
		From \eqref{eq:affine-model-kernel}, the kernel \(K_e^0\) is nonnegative, measurable, and symmetric, and satisfies
		\[
		0\le K_e^0(x,y)
		\le (p-1)|x-y|^{-n-2\mathfrak s}.
		\]
		Since \(K_e^0(x,y)\) depends only on the difference \(x-y\) and is invariant under replacing \(x-y\) by \(y-x\), we have
		\[
		K_e^0(x+h,y+h)=K_e^0(x,y),
		\qquad
		K_e^0(x,x+z)=K_e^0(x,x-z),
		\]
		so that the left hand sides of \eqref{eq:linear-kernel-translation} and \eqref{eq:linear-kernel-almost-even} both vanish.
		
		Using the fact that \(|e\cdot\omega|\le1\), we obtain
		\begin{align*}
			&\rho^{2\mathfrak s}
			\int_{B_{2\rho}(x)\setminus B_\rho(x)}
			K_e^0(x,y)\,\diff y\\
			&\qquad=(p-1)
			\left(\int_1^2t^{-1-2\mathfrak s}\,\diff t\right)
			\left(
			\int_{\mathbb S^{n-1}}
			|e\cdot\omega|^{p-2}\,\diff\omega
			\right)\le
			C
			\int_1^2t^{-1-2\mathfrak s}\,\diff t
			=:\Lambda_0.
		\end{align*}
		For the lower bound, define for \(e\in\mathbb S^{n-1}\) and \(\xi\in\mathbb R^n\)
		\[
		F(e,\xi)	:=
		\int_{\mathbb S^{n-1}}
		|e\cdot\omega|^{p-2}
		|\xi\cdot\omega|^2\,\diff\omega.
		\]
		The integrand is continuous and bounded by \(1\) when $e,\xi,\omega \in \mathbb S^{n-1}$, so the restriction of \(F\) to \(\mathbb S^{n-1}\times\mathbb S^{n-1}\) is continuous. Fixing \(e,\xi\in\mathbb S^{n-1}\) and choosing
		\[
		\omega_0:=
		\begin{cases}
			e,&e\cdot\xi\ne0,\\[2mm]
			\dfrac{e+\xi}{\sqrt2},&e\cdot\xi=0,
		\end{cases}
		\]
		we have \(\omega_0\in\mathbb S^{n-1}\) and
		\[
		|e\cdot\omega_0|^{p-2}|\xi\cdot\omega_0|^2>0.
		\]
		By continuity, the integrand is positive on a spherical neighborhood of \(\omega_0\) with positive surface measure, whence \(F(e,\xi)>0\). Compactness of \(\mathbb S^{n-1}\times\mathbb S^{n-1}\) then gives
		\[
		c_{n,p}
		:=
		\min_{e,\xi\in\mathbb S^{n-1}}F(e,\xi)>0.
		\]
		For \(\xi\ne0\), setting \(\widehat\xi:=\xi/|\xi|\in\mathbb S^{n-1}\) yields
		\[
		F(e,\xi)
		=
		|\xi|^2F(e,\widehat\xi)
		\ge c_{n,p}|\xi|^2,
		\]
		and the same inequality is immediate when \(\xi=0\). Therefore
		\[
		F(e,\xi)\ge c_{n,p}|\xi|^2
		\qquad
		(e\in\mathbb S^{n-1},\ \xi\in\mathbb R^n).
		\]
		
		Using this estimate, we obtain
		\begin{align*}
			\rho^{2\mathfrak s-2}
			\int_{B_\rho(x)}
			|\xi\cdot(x-y)|^2K_e^0(x,y)\,\diff y&=
			(p-1)\rho^{2\mathfrak s-2}
			\int_0^\rho t^{1-2\mathfrak s}\,\diff t\,
			F(e,\xi)\\
			&=
			\frac{p-1}{2-2\mathfrak s}F(e,\xi)\ge
			\frac{(p-1)c_{n,p}}{2-2\mathfrak s}|\xi|^2.
		\end{align*}
		Since \(2-2\mathfrak s=p(1-s)>0\), we may set
		\[
		\lambda_0
		:=
		\frac{(p-1)c_{n,p}}{2-2\mathfrak s}>0.
		\]
		Thus \eqref{eq:linear-kernel-upper} and \eqref{eq:linear-kernel-lower} hold with constants \(\Lambda_0\) and \(\lambda_0\), respectively. Combining the preceding estimates yields
		\[
		K_e^0\in
		\mathscr K(2\mathfrak s,\alpha;\lambda_0,\Lambda_0,1)
		\qquad
		\text{for every }\alpha\in(0,1].
		\]
	\end{proof}
	
	The large slope argument gives control of the coefficient kernel on
    an interior ball. To apply Lemma~\ref{lem:fixed-radius-schauder}, we
    need a kernel defined on the whole space. The following extension
    preserves the required structural bounds for a local perturbation of $K_e^0$.
	
	\begin{lemma}
		\label{lem:global-kernel-extension}
		Fix $e\in\mathbb S^{n-1}$, and let $K_0:=K_e^0$,
		$\lambda_0$, and $\Lambda_0$ be as in
		Lemma~\ref{lem:affine-model-kernel}.
		Let $\eta\in(0,1)$ and let $K$ be a nonnegative
		measurable symmetric kernel on	$(B_{2/3}\times B_{2/3})\setminus\{x=y\}$, and set $H:=K-K_0$ there.
		Assume that, for some $\varepsilon,M\ge0$,
		\begin{align*}
			|H(x,y)|&\le\varepsilon|x-y|^{-n-2\mathfrak s}
			&&(x,y\in B_{2/3},\ x\ne y),\\
			|H(x+h,y+h)-H(x,y)|
			&\le M|h|^\eta|x-y|^{-n-2\mathfrak s}
			&&(x,y,x+h,y+h\in B_{2/3}),\\
			|H(x,x+z)-H(x,x-z)|
			&\le M|z|^{\eta-n-2\mathfrak s}
			&&(x,x+z,x-z\in B_{2/3}).
		\end{align*}
		Choose
		\[
		\chi\in C_c^\infty(B_{2/3}),
		\qquad 0\le\chi\le1,
		\qquad \chi\equiv1\quad\text{on }B_{5/8},
		\]
		and, for $x\ne y$, define
		\[
		\widetilde K(x,y):=K_0(x,y)+
		\begin{cases}
			\chi(x)\chi(y)H(x,y),&x,y\in B_{2/3},\\
			0,&\text{otherwise}.
		\end{cases}
		\]
		There is $\varepsilon_*=\varepsilon_*(n,p,s)>0$ such that, if
		$\varepsilon\le\varepsilon_*$, then
		\begin{equation*}
			\widetilde K\in
			\mathscr K(2\mathfrak s,\eta;\lambda_0/2,\Lambda_*,M_*),
		\end{equation*}
		where $\Lambda_*$ and $M_*$ depend only on
		$n,p,s,\eta,M$ and the fixed function $\chi$.  Moreover,
		\[
		\widetilde K=K\quad\text{on }B_{5/8}\times B_{5/8},
		\qquad
		|\widetilde K(x,y)-K_0(x,y)|
		\le\varepsilon|x-y|^{-n-2\mathfrak s}.
		\]
	\end{lemma}
	\begin{proof}
		We extend $H$ by zero to $\R^n\times\R^n$ and define
		\[
		c(x,y):=\chi(x)\chi(y),
		\qquad
		G(x,y):=\widetilde K(x,y)-K_0(x,y)=c(x,y)H(x,y).
		\]
		Since $0\le c\le1$, for $x,y\in B_{2/3}$ we have
		\[
		\widetilde K
		=K_0+c(K-K_0)
		=cK+(1-c)K_0,
		\]
		whereas $\widetilde K=K_0$ whenever at least one of $x,y$ lies outside $B_{2/3}$. Hence $\widetilde K$ is nonnegative and symmetric. Moreover, since $\chi\equiv1$ on $B_{5/8}$,
		\[
		\widetilde K(x,y)
		=K_0(x,y)+H(x,y)
		=K(x,y)
		\qquad
		(x,y\in B_{5/8},\ x\ne y),
		\]
		and we have the pointwise bound
		\begin{equation}\label{eq:G-pointwise-bound}
			|G(x,y)|
			\le \varepsilon |x-y|^{-n-2\mathfrak s}
			\qquad
			(x,y\in\mathbb R^n,\ x\ne y).
		\end{equation}
		
		We turn to the ellipticity conditions \eqref{eq:linear-kernel-upper} and \eqref{eq:linear-kernel-lower}. From \eqref{eq:G-pointwise-bound} and the corresponding upper bound for $K_0$, we obtain
		\begin{align*}
			\rho^{2\mathfrak s}
			\int_{B_{2\rho}(x)\setminus B_\rho(x)}
			\widetilde K(x,y)\,\diff y
			&\le
			\rho^{2\mathfrak s}
			\int_{B_{2\rho}(x)\setminus B_\rho(x)}
			K_0(x,y)\,\diff y
			+
			\rho^{2\mathfrak s}
			\int_{B_{2\rho}(x)\setminus B_\rho(x)}
			|G(x,y)|\,\diff y
			\\
			&\le
			\Lambda_0
			+
			\varepsilon\rho^{2\mathfrak s}
			\int_{B_{2\rho}(x)\setminus B_\rho(x)}
			|x-y|^{-n-2\mathfrak s}\,\diff y
			\le
			\Lambda_0+C\varepsilon.
		\end{align*}
		Similarly, for every $\xi\in\R^n$,
		\begin{align*}
			&\rho^{2\mathfrak s-2}
			\int_{B_\rho(x)}
			|\xi\cdot(x-y)|^2
			\widetilde K(x,y)\,\diff y
			\\
			&\quad\ge
			\rho^{2\mathfrak s-2}
			\int_{B_\rho(x)}
			|\xi\cdot(x-y)|^2
			K_0(x,y)\,\diff y
			- \rho^{2\mathfrak s-2}
			\int_{B_\rho(x)}
			|\xi\cdot(x-y)|^2
			|G(x,y)|\,\diff y
			\\
			&\quad\ge
			\lambda_0|\xi|^2
			-
			\varepsilon|\xi|^2
			\rho^{2\mathfrak s-2}
			\int_{B_\rho(x)}
			|x-y|^{2-n-2\mathfrak s}\,\diff y
			\ge
			(\lambda_0-C\varepsilon)|\xi|^2.
		\end{align*}
		Choosing $\varepsilon_*=\varepsilon_*(n,p,s)>0$ with $C\varepsilon_*\le\lambda_0/2$, we have, whenever $\varepsilon\le\varepsilon_*$,
		\[
		\rho^{2\mathfrak s-2}
		\int_{B_\rho(x)}
		|\xi\cdot(x-y)|^2
		\widetilde K(x,y)\,\diff y
		\ge
		\frac{\lambda_0}{2}|\xi|^2.
		\]
		Thus \eqref{eq:linear-kernel-upper} and \eqref{eq:linear-kernel-lower} hold with
		\[
		\lambda=\frac{\lambda_0}{2},
		\qquad
		\Lambda_*:=\Lambda_0+C\varepsilon_*.
		\]
		
		It remains to verify the translation and almost even estimates. Set
		\[
		d:=\dist(\supp\chi,\R^n\setminus B_{2/3})>0,
		\qquad
		h_0:=\min\{d/2,1\}.
		\]
		For $|h|<h_0$, if at least one of $c(x+h,y+h)$ and $c(x,y)$ is nonzero, then
		\[
		x,y,x+h,y+h\in B_{2/3},
		\]
		so the translation estimate for $H$ applies. From the identity
		\[
		\begin{aligned}
			G(x+h,y+h)-G(x,y)
			&=
			c(x+h,y+h)
			\bigl[H(x+h,y+h)-H(x,y)\bigr]
			\\
			&\quad+
			\bigl[c(x+h,y+h)-c(x,y)\bigr]H(x,y),
		\end{aligned}
		\]
		we deduce
		\[
		\begin{aligned}
			|G(x+h,y+h)-G(x,y)|
			&\le
			M|h|^\eta|x-y|^{-n-2\mathfrak s}
			\\
			&\quad+
			|c(x+h,y+h)-c(x,y)|
			|H(x,y)|.
		\end{aligned}
		\]
		Since $\chi$ is Lipschitz, $|h|\le1$, and $\eta\in(0,1)$, we have
		\[
		|c(x+h,y+h)-c(x,y)|
		\le C_\chi |h|\le C_\chi |h|^{\eta},
		\]
		and together with the first bound for $H$, this gives
		\[
		|G(x+h,y+h)-G(x,y)|
		\le
		C(M+\varepsilon)
		|h|^\eta|x-y|^{-n-2\mathfrak s}.
		\]
		If both $c(x+h,y+h)$ and $c(x,y)$ vanish, the left hand side is zero, so the same estimate is valid.
		
		For $|h|\ge h_0$, \eqref{eq:G-pointwise-bound} yields
		\[
		\begin{aligned}
			|G(x+h,y+h)-G(x,y)|
			&\le
			|G(x+h,y+h)|+|G(x,y)|
			\\
			&\le
			2\varepsilon|x-y|^{-n-2\mathfrak s}\le
			2\varepsilon h_0^{-\eta}
			|h|^\eta|x-y|^{-n-2\mathfrak s}.
		\end{aligned}
		\]
		Consequently, for all $h\in\R^n$,
		\[
		|G(x+h,y+h)-G(x,y)|
		\le
		C|h|^\eta|x-y|^{-n-2\mathfrak s},
		\]
		where $C$ depends only on $\eta,M,\varepsilon_*$ and $\chi$. Since
		\[
		K_0(x+h,y+h)-K_0(x,y)=0,
		\]
		we conclude that
		\begin{align*}
			\int_{B_{2\rho}(x)\setminus B_\rho(x)}
			|\widetilde K(x+h,y+h)-\widetilde K(x,y)|\,\diff y
			&\le
			C|h|^\eta
			\int_{B_{2\rho}(x)\setminus B_\rho(x)}
			|x-y|^{-n-2\mathfrak s}\,\diff y\\
			&\le
			C|h|^\eta\rho^{-2\mathfrak s}.
		\end{align*}
		
		Finally, to verify the almost even estimate \eqref{eq:linear-kernel-almost-even}, define
		\[
		c_\pm:=\chi(x)\chi(x\pm z),
		\qquad
		H_\pm:=H(x,x\pm z),
		\]
		so that
		\[
		\begin{aligned}
			G(x,x+z)-G(x,x-z)
			&=
			c_+(H_+-H_-)
			+(c_+-c_-)H_-.
		\end{aligned}
		\]
		Suppose first that $|z|<h_0$. If at least one of $c_+$ and $c_-$ is nonzero, then
		\[
		x,\ x+z,\ x-z\in B_{2/3},
		\]
		so the almost even assumption on $H$ gives
		\[
		|H_+-H_-|
		\le
		M|z|^{\eta-n-2\mathfrak s}.
		\]
		Moreover, since
		\[
		|c_+-c_-|
		\le C_\chi|z|,
		\]
		and $|z|\le1$, it follows that
		\[
		\begin{aligned}
			|(c_+-c_-)H_-|
			\le
			C\varepsilon
			|z|^{1-n-2\mathfrak s}
			\le
			C\varepsilon
			|z|^{\eta-n-2\mathfrak s}.
		\end{aligned}
		\]
		Thus
		\[
		|G(x,x+z)-G(x,x-z)|
		\le
		C(M+\varepsilon)
		|z|^{\eta-n-2\mathfrak s}.
		\]
		If $c_+=c_-=0$, the left hand side vanishes.
		
		For $|z|\ge h_0$, \eqref{eq:G-pointwise-bound} gives
		\[
		\begin{aligned}
			|G(x,x+z)-G(x,x-z)|
			\le
			2\varepsilon |z|^{-n-2\mathfrak s}
			\le
			2\varepsilon h_0^{-\eta}
			|z|^{\eta-n-2\mathfrak s}.
		\end{aligned}
		\]
		Hence, for every $z\ne0$,
		\[
		|G(x,x+z)-G(x,x-z)|
		\le
		C|z|^{\eta-n-2\mathfrak s}.
		\]
		Since
		\[
		K_0(x,x+z)=K_0(x,x-z),
		\]
		we obtain
		\begin{align*}
			\int_{B_{2\rho}\setminus B_\rho}
			|\widetilde K(x,x+z)-\widetilde K(x,x-z)|\,\diff z
			&\le
			C
			\int_{B_{2\rho}\setminus B_\rho}
			|z|^{\eta-n-2\mathfrak s}\,\diff z\\
			&\le
			C\rho^{\eta-2\mathfrak s}.
		\end{align*}
		Therefore
		\[
		\widetilde K\in
		\mathscr K
		(2\mathfrak s,\eta;\lambda_0/2,\Lambda_*,M_*),
		\]
		with $M_*$ depending only on
		$n,p,s,\eta,M$ and the fixed cutoff $\chi$.
	\end{proof}
	
	We will apply the Schauder estimate with an exponent smaller than the regularity exponent obtained for the kernel. We therefore record that the kernel assumptions remain valid when the exponent is decreased.
	\begin{lemma}
		\label{lem:kernel-exponent-downgrade}
		Let $0<\alpha\le\eta\le1$ and
		\[
		K\in\mathscr K(2\mathfrak s,\eta;\lambda,\Lambda,M).
		\]
		Then
		\begin{equation}\label{eq:kernel-exponent-downgrade}
			K\in\mathscr K
			\bigl(2\mathfrak s,\alpha;\lambda,\Lambda,
			\max\{M,2\Lambda\}\bigr).
		\end{equation}
	\end{lemma}
	
	\begin{proof}
		The ellipticity conditions are unchanged.  If $|h|\le1$, then
		$|h|^\eta\le|h|^\alpha$, so
		\eqref{eq:linear-kernel-translation} with exponent $\eta$ implies the same
		estimate with exponent $\alpha$.  If $|h|>1$,
		\eqref{eq:linear-kernel-upper} at $x$ and $x+h$ gives
		\[
		\int_{B_{2\rho}(x)\setminus B_\rho(x)}
		|K(x+h,y+h)-K(x,y)|\,\diff y
		\le2\Lambda\rho^{-2\mathfrak s}
		\le2\Lambda|h|^\alpha\rho^{-2\mathfrak s}.
		\]
		Similarly, when $\rho\le1$, \eqref{eq:linear-kernel-almost-even} with exponent $\eta$
		implies the one with exponent $\alpha$.  When $\rho>1$,
		\eqref{eq:linear-kernel-upper}, applied to the two terms, yields
		\[
		\int_{B_{2\rho}\setminus B_\rho}
		|K(x,x+z)-K(x,x-z)|\,\diff z
		\le2\Lambda\rho^{-2\mathfrak s}
		\le2\Lambda\rho^{\alpha-2\mathfrak s}.
		\]
		This proves \eqref{eq:kernel-exponent-downgrade}.
	\end{proof}
	
	\subsection{Uniform affine decay}
	
	We now combine the controlled slope argument with the preceding
	large slope construction. This yields homogeneous affine decay with
	constants independent of the minimizing slope.
	\begin{lemma}
		\label{lem:full-affine-decay-high}
		Let $n\ge2$, $p>2$, $s\in(0,1)$, and assume $sp>p-1$. There exist
		$\sigma\in(0,1)$ and $C\ge1$, depending only on $n,p,s$, such that every
		\(
		v\in W^{s,p}_{\mathrm{loc}}(B_{2r}(x_0))
		\cap L_{sp}^{p-1}(\R^n)
		\cap L^\infty(B_{2r}(x_0))
		\)
		that is a local weak solution of $(-\Delta_p)^sv=0$
		in $B_{2r}(x_0)$ satisfies
		\begin{equation*}
			\Psi_\infty(v;x_0,\rho)
			\le
			C\left(\frac{\rho}{r}\right)^\sigma
			\Psi_\infty(v;x_0,r),
			\qquad
			0<\rho\le\frac r{32}.
		\end{equation*}
	\end{lemma}
	
	\begin{proof}
		Let $\alpha\in(0,1)$ be the exponent in Lemma~\ref{lem:homogeneous-C1a}. We divide the proof into nine steps.
		
		\smallskip
		\noindent\textbf{Step 1: rescaling.}
		We take $
		E:=\Psi_\infty(v;x_0,r)$. By Lemma~\ref{lem:affine-minimizer}, there exists a minimizing affine function
		\[
		\ell_{a,b}(y):=b+a\cdot(y-x_0).
		\]
		The conclusion is immediate when $E=0$, so assume $E>0$ and define
		\[
		\widehat v(x):=\frac{v(x_0+rx)-b}{rE},
		\qquad
		\widehat a:=\frac{a}{E}.
		\]
		Then $\widehat v$ is fractional $p$-harmonic in $B_2$, and
		\[
		\Psi_\infty(\widehat v;0,\rho)
		=
		\frac{1}{E}\Psi_\infty(v;x_0,r\rho),
		\qquad 0<\rho\le1.
		\]
		Moreover, $\widehat a\cdot x$ minimizes $\Psi_\infty(\widehat v;0,1)$, whence
		\[
		\Psi_\infty(\widehat v;0,1)
		=
		\|\widehat v-\widehat a\cdot x\|_{L^\infty(B_1)}
		+
		\Tail_{p-1,sp}
		(\widehat v-\widehat a\cdot x;0,1)
		=1.
		\]
		Relabeling $(\widehat v,\widehat a)$ as $(v,a)$, we may assume
		\[
		x_0=0,\qquad r=E=1,
		\]
		and
		\begin{equation}\label{eq:large-slope-normalization}
			\|v-a\cdot x\|_{L^\infty(B_1)}
			+
			\Tail_{p-1,sp}(v-a\cdot x;0,1)
			=1.
		\end{equation}
		
		Let $M_0=M_0(n,p,s)\ge 24^{1+\alpha}$ with an additional lower bound to be specified in Step~5. If $|a|\le M_0$, the conclusion follows from Lemma~\ref{lem:controlled-affine-decay}. We may therefore assume
		\[
		A:=|a|>M_0,
		\qquad
		e:=\frac{a}{A}\in\mathbb S^{n-1},
		\qquad
		z:=v-a\cdot x,
		\]
		so that
		\[
		\|z\|_{L^\infty(B_1)}
		+
		\Tail_{p-1,sp}(z;0,1)
		=1.
		\]
		
		\smallskip
		\noindent\textbf{Step 2: uniform estimates after slope rescaling.}
		Let \[
		v_A:=\frac{v}{A}
		=e\cdot x+\frac{z}{A}.
		\]
		By homogeneity, $v_A$ is a local weak solution of
		\[
		(-\Delta_p)^s v_A=0
		\qquad\text{in }B_2.
		\]
		From \eqref{eq:large-slope-normalization} we have
		\[
		\|v_A\|_{L^\infty(B_1)}
		\le
		\|e\cdot x\|_{L^\infty(B_1)}
		+\frac1A\|z\|_{L^\infty(B_1)}
		\le2,
		\]
		and Minkowski's inequality, together with $|e|=1$, gives
		\[
		\Tail_{p-1,sp}(v_A;0,1)
		\le
		\Tail_{p-1,sp}(e\cdot x;0,1)
		+\frac1A\Tail_{p-1,sp}(z;0,1)\le C.
		\]
		Hence
		\begin{equation*}
			\|v_A\|_{L^\infty(B_1)}
			+
			\Tail_{p-1,sp}(v_A;0,1)
			\le C.
		\end{equation*}
		
		We next establish a uniform $C^{1,\alpha}$ estimate on $B_{3/4}$. Fix $\xi\in B_{3/4}$ and define
		\[
		c_\xi:=(v_A)_{B_{1/4}(\xi)}.
		\]
		Since $B_{1/4}(\xi)\subset B_1$ and $\|v_A\|_{L^\infty(B_1)}\le2$, we have
		\[
		|c_\xi|\le2,
		\qquad
		\osc_{B_{1/4}(\xi)}v_A\le4.
		\]
		For $|y|\ge1$ and $\xi\in B_{3/4}$, we have $|y-\xi|\ge |y|/4$. Splitting the tail at $B_1$ therefore gives
		\[
		\begin{aligned}
			\Tail_{p-1,sp}(v_A-c_\xi;\xi,1/4)^{p-1}
			&\le
			C\|v_A-c_\xi\|_{L^\infty(B_1)}^{p-1}
			+
			C\int_{\R^n\setminus B_1}
			\frac{|v_A(y)|^{p-1}+|c_\xi|^{p-1}}
			{|y|^{n+sp}}\,\diff y\\
			&\le C.
		\end{aligned}
		\]
		Thus
		\[
		\mathcal H(v_A;\xi,1/4)\le C
		\qquad
		(\xi\in B_{3/4}),
		\]
		and applying Lemma~\ref{lem:homogeneous-C1a} yields
		\begin{equation}\label{eq:local-C1a-vA}
			\|\nabla v_A\|_{L^\infty(B_{1/8}(\xi))}
			+
			[\nabla v_A]_{C^{0,\alpha}(B_{1/8}(\xi))}
			\le C.
		\end{equation}
		In particular, we have
		\[
		\|\nabla v_A\|_{L^\infty(B_{3/4})}\le C.
		\]
		Moreover, for $x,y\in B_{3/4}$, \eqref{eq:local-C1a-vA} gives
		\[
		|\nabla v_A(x)-\nabla v_A(y)|
		\le
		\begin{cases}
			C|x-y|^\alpha, & |x-y|<1/8,\\[1mm]
			2C \le 2C8^\alpha|x-y|^\alpha, & |x-y|\ge1/8.
		\end{cases}
		\]
		Combining these two bounds yields
		\begin{equation}\label{eq:large-slope-C1a}
			\|\nabla v_A\|_{L^\infty(B_{3/4})}
			+
			[\nabla v_A]_{C^{0,\alpha}(B_{3/4})}
			\le C.
		\end{equation}
		
		We next estimate the gradient of the normalized remainder. Set
\[
h_A:=\frac{z}{A} =v_A-e\cdot x,
		\qquad
		\tau:=\frac{\alpha}{1+\alpha},
		\qquad
		t_A:=A^{-1/(1+\alpha)}.
		\]
		From \eqref{eq:large-slope-normalization} and \eqref{eq:large-slope-C1a} we have
		\[
		\|h_A\|_{L^\infty(B_{3/4})}
		\le A^{-1},
		\qquad
		[\nabla h_A]_{C^{0,\alpha}(B_{3/4})}
		\le C.
		\]
		Since $M_0\ge24^{1+\alpha}$ and $A>M_0$, we have $t_A\le1/24$, so that for every $x\in B_{2/3}$ and $\nu\in\mathbb S^{n-1}$, the segment $[x,x+t_A\nu]$ is contained in $B_{3/4}$. By the fundamental theorem of calculus,
		\[
		\begin{aligned}
			\left|
			\partial_\nu h_A(x)
			-
			\frac{h_A(x+t_A\nu)-h_A(x)}{t_A}
			\right|
			&=
			\left|
			\frac1{t_A}
			\int_0^{t_A}
			\left[
			\partial_\nu h_A(x)
			-
			\partial_\nu h_A(x+s\nu)
			\right]\,\diff s
			\right|\\
			&\le
			\frac{C}{t_A}
			\int_0^{t_A}s^\alpha\,\diff s
			\le
			Ct_A^\alpha.
		\end{aligned}
		\]
		Using $\|h_A\|_{L^\infty(B_{3/4})}\le A^{-1}$ together with the definitions of $t_A$ and $\tau$, we obtain
		\begin{equation*}
			|\partial_\nu h_A(x)|
			\le
			\frac{|h_A(x+t_A\nu)-h_A(x)|}{t_A}
			+
			Ct_A^\alpha\le
			\frac{2}{At_A}
			+
			Ct_A^\alpha\le  CA^{-\tau},
		\end{equation*}
		and taking the supremum over $\nu\in\mathbb S^{n-1}$ gives
		\begin{equation*}
			\|\nabla h_A\|_{L^\infty(B_{2/3})}
			=
			\left\|
			\nabla\left(\frac{z}{A}\right)
			\right\|_{L^\infty(B_{2/3})}
			=
			\|\nabla v_A-e\|_{L^\infty(B_{2/3})}
			\le
			CA^{-\tau}.
		\end{equation*}
		Finally, for $x,y\in B_{2/3}$ with $x\ne y$, convexity of $B_{2/3}$ yields
		\[
		\frac{\delta z(x,y)}{A}
		=
		h_A(x)-h_A(y)
		=
		\int_0^1
		\nabla h_A\bigl(y+t(x-y)\bigr)
		\cdot(x-y)\,\diff t,
		\]
		and therefore
		\begin{equation}\label{eq:qA-small}
			\left|
			\frac{\delta z(x,y)}
			{A|x-y|}
			\right|
			\le
			\|\nabla h_A\|_{L^\infty(B_{2/3})}
			\le
			CA^{-\tau}.
		\end{equation}
		
		\smallskip
		\noindent\textbf{Step 3: the equation for $z$.}
		Applying Lemma~\ref{lem:affine-linearization} gives, for every $\varphi\in C_c^\infty(B_2)$,
		\begin{equation*}
			\iint_{\R^n\times\R^n}
			\delta z(x,y)\delta\varphi(x,y)K_A(x,y)\,\diff x\diff y=0,
		\end{equation*}
		with
		\begin{equation}\label{eq:large-slope-kernel}
			K_A(x,y)
			:=
			A^{2-p}(p-1)
			\left(
			\int_0^1
			|a\cdot(x-y)+t\delta z(x,y)|^{p-2}\diff t
			\right)
			|x-y|^{-n-sp}.
		\end{equation}
		Set
		\[
		2\mathfrak s:=sp-p+2=1+\beta\in(1,2).
		\]
		For $x\ne y$, define
		\[
		\omega:=\frac{x-y}{|x-y|},
		\qquad
		q_A(x,y):=\frac{\delta z(x,y)}{A|x-y|}.
		\]
		Since $a=Ae$, we have
		\begin{equation*}
			K_A(x,y)
			=
			k_A(x,y)|x-y|^{-n-2\mathfrak s},
		\end{equation*}
		where
		\begin{equation*}
			k_A(x,y)
			:=
			(p-1)\int_0^1
			|e\cdot\omega+tq_A(x,y)|^{p-2}\diff t.
		\end{equation*}
		
		\smallskip
		\noindent\textbf{Step 4: estimates for $K_A$.}
		Set
		\[
		\kappa:=\min\{1,p-2\},
		\qquad
		\eta:=\frac{\alpha\kappa}{2}.
		\]
		For $t\in[0,1]$, define
		\[
		\Theta_t(x,y)
		:=
		e\cdot\omega+tq_A(x,y)
		=
		(1-t)e\cdot\omega
		+
		t\frac{v_A(x)-v_A(y)}{|x-y|}.
		\]
		For $x,y\in B_{2/3}$,
		\[
		\frac{v_A(x)-v_A(y)}{|x-y|}
		=
		\int_0^1
		\nabla v_A(y+\lambda(x-y))
		\cdot\omega\,\diff\lambda,
		\]
		and consequently
		\begin{equation*}
			|\Theta_t(x,y)|\le C,
			\qquad
			0\le k_A(x,y)\le C.
		\end{equation*}
		If $x,y,x+h,y+h\in B_{2/3}$, then
		\[
		q_A(x,y)
		=\frac{\delta z(x,y)}{A|x-y|}=\frac{h_A(x)-h_A(y)}{|x-y|}
		=
		\int_0^1
		\nabla h_A\bigl(y+\lambda(x-y)\bigr)
		\cdot\omega\,\diff\lambda,
		\]
		and hence
		\begin{equation*}
			\begin{split}
				|q_A(x+h,y+h)-q_A(x,y)|
				&\le
				\int_0^1
				\left|
				\nabla h_A\bigl(y+h+\lambda(x-y)\bigr)
				\cdot\omega-
				\nabla h_A\bigl(y+\lambda(x-y)\bigr)
				\cdot\omega
				\right|\diff\lambda\\
				&\le C|h|^\alpha.
			\end{split}
		\end{equation*}
		Combining this with Lemma~\ref{lem:power-modulus} yields
		\[
		|k_A(x+h,y+h)-k_A(x,y)|
		\le C|h|^{\alpha\kappa}
		\le C|h|^\eta,
		\]
		where the last inequality follows from $|h|<4/3$ and $\eta=\alpha\kappa/2$. Consequently,
		\begin{equation}\label{eq:KA-pointwise-translation}
			|K_A(x+h,y+h)-K_A(x,y)|
			\le
			C|h|^\eta|x-y|^{-n-2\mathfrak s}.
		\end{equation}
		Next, for $x,x+\zeta,x-\zeta\in B_{2/3}$ with $\zeta\ne0$, symmetry of $K_A$ and \eqref{eq:KA-pointwise-translation} applied to the pair $(x-\zeta,x)$ with shift $h=\zeta$ give
		\begin{equation}\label{eq:KA-pointwise-even}
			\begin{aligned}
				|K_A(x,x+\zeta)-K_A(x,x-\zeta)|
				&=
				|K_A((x-\zeta)+\zeta,x+\zeta)-K_A(x-\zeta,x)|\\
				&\le C|\zeta|^{\eta-n-2\mathfrak s}.
			\end{aligned}
		\end{equation}
		
		\smallskip
		\noindent\textbf{Step 5: comparison with $K_e^0$ and extension.}
		Let $K_0:=K_e^0$. For $x,y\in B_{2/3}$, $x\ne y$, and $t\in[0,1]$, \eqref{eq:qA-small} gives
		\[
		|\Theta_t(x,y)-e\cdot\omega|
		=
		t|q_A(x,y)|=\frac{t|\delta z(x,y)|}{A|x-y|}
		\le
		CA^{-\tau}.
		\]
		Lemma~\ref{lem:power-modulus} then yields
		\begin{equation*}
			\begin{split}
				|K_A(x,y)-K_0(x,y)|
				&\le (p-1)\int_0^{1}\left|\left|e\cdot \omega +tq_A(x,y)\right|^{p-2}-\left|e\cdot \omega\right|^{p-2}\right| \diff t \, |x-y|^{-n-2\mathfrak s}\\
				&\le CA^{-\tau\kappa}|x-y|^{-n-2\mathfrak s}
				=:\varepsilon_A|x-y|^{-n-2\mathfrak s},
			\end{split}
		\end{equation*}
		for all $x,y\in B_{2/3}$, $x\ne y$.
		
		Now set
		\[
		H:=K_A-K_0
		\qquad
		\text{on }
		(B_{2/3}\times B_{2/3})\setminus\{x=y\}.
		\]
		Since
		\[
		K_0(x+h,y+h)=K_0(x,y),
		\qquad
		K_0(x,x+\zeta)=K_0(x,x-\zeta),
		\]
		\eqref{eq:KA-pointwise-translation} and \eqref{eq:KA-pointwise-even} give
		\begin{align*}
			|H(x+h,y+h)-H(x,y)|
			&\le
			C|h|^\eta|x-y|^{-n-2\mathfrak s}
			\qquad
			(x,y,x+h,y+h\in B_{2/3}),
		\end{align*}
		and
		\begin{align*}
			|H(x,x+\zeta)-H(x,x-\zeta)|
			&\le
			C|\zeta|^{\eta-n-2\mathfrak s}
			\qquad
			(x,x+\zeta,x-\zeta\in B_{2/3}).
		\end{align*}
		Let $\varepsilon_*$ be the constant in Lemma~\ref{lem:global-kernel-extension}, and fix $M_0$, depending only on $n,p,s$, so that
		\begin{equation*}
			M_0\ge24^{1+\alpha},
			\qquad
			CM_0^{-\tau\kappa}\le\varepsilon_*.
		\end{equation*}
		By definition, $K_A$ is nonnegative, measurable, and symmetric, and since $A>M_0$,
		\[
		\varepsilon_A
		\le
		CM_0^{-\tau\kappa}
		\le
		\varepsilon_*.
		\]
		Thus the assumptions of Lemma~\ref{lem:global-kernel-extension} are satisfied with
		\[
		K=K_A,
		\qquad
		H=K_A-K_0,
		\qquad
		\varepsilon=\varepsilon_A,
		\qquad
		M=C,
		\]
		and consequently there exists a nonnegative symmetric kernel $\widetilde K$ such that
		\begin{equation}\label{eq:Ktilde-class}
			\widetilde K
			\in
			\mathscr K
			(2\mathfrak s,\eta;\lambda_0/2,\Lambda_*,M_*),
		\end{equation}
		where $\Lambda_*$ and $M_*$ depend only on $n,p,s$, and
		\begin{align*}
			\widetilde K&=K_A
			&&\text{on }B_{5/8}\times B_{5/8},
			\\
			|\widetilde K(x,y)-K_0(x,y)|
			&\le
			\varepsilon_A|x-y|^{-n-2\mathfrak s}
			&&(x\ne y).
		\end{align*}
		In particular, we have
		\begin{equation}\label{eq:Ktilde-pointwise-upper}
			0\le
			\widetilde K(x,y)
			\le
			C|x-y|^{-n-2\mathfrak s},
			\qquad
			x\ne y.
		\end{equation}
		
		\smallskip
		\noindent\textbf{Step 6: cutoff of the remainder.}
		Set $\bar z:=z-z(0)$, choose $\zeta\in C_c^\infty(B_{5/8})$ with $0\le\zeta\le1$ and $\zeta\equiv1$ on $B_{9/16}$, and define
		\begin{equation*}
			W:=\zeta\bar z.
		\end{equation*}
		For each fixed $A$, \eqref{eq:large-slope-C1a} gives
		\[
		z\in C_{\mathrm{loc}}^{1,\alpha}(B_{3/4}),
		\qquad
		\supp W\Subset B_{5/8},
		\qquad
		\|W\|_{L^\infty(\R^n)}\le2,
		\]
		and, for some $C_A<\infty$,
		\[
		|W(x)-W(y)|^2
		\le
		C_A(|x-y|^2\wedge1).
		\]
		Using \eqref{eq:Ktilde-pointwise-upper}, we obtain
		\begin{align*}
			\iint_{\mathcal Q(B_{1/2})}
			|W(x)-W(y)|^2
			\widetilde K(x,y)\,\diff x\diff y
			&\le
			2C_A
			\int_{B_{1/2}}\int_{\R^n}
			\frac{|x-y|^2\wedge1}
			{|x-y|^{n+2\mathfrak s}}
			\,\diff y\diff x
			<\infty,
		\end{align*}
		so that $W\in\mathcal H_{\widetilde K}(B_{1/2})$.
		
		For $\varphi\in C_c^\infty(B_{1/2})$ and $x\in\supp\varphi$,
		\[
		\delta W(x,y)
		=
		\delta\bar z(x,y)
		+
		(1-\zeta(y))\bar z(y).
		\]
		For $0<\varepsilon<1$, symmetry on the set
		\[
		\{(x,y):\varepsilon<|x-y|<\varepsilon^{-1}\}
		\]
		gives
		\begin{align*}
			\mathcal E_{\widetilde K}(W,\varphi)
			=
			2\lim_{\varepsilon\to0}
			\int_{B_{1/2}}\varphi(x)
			\int_{\{\varepsilon<|x-y|<\varepsilon^{-1}\}}
			\delta W(x,y)\widetilde K(x,y)\,\diff y\diff x.
		\end{align*}
	By Lemma~\ref{lem:affine-linearization} and
	$\delta\bar z=\delta z$, we have
	\[
	\mathcal E_{K_A}(\bar z,\varphi)=0,
	\]
	with an absolutely convergent pairing. Hence, by symmetry,
	\[
	0
	=
	2\lim_{\varepsilon\to0}
	\int_{B_{1/2}}\varphi(x)
	\int_{\{\varepsilon<|x-y|<\varepsilon^{-1}\}}
	\delta\bar z(x,y)K_A(x,y)\,\diff y\diff x.
	\]
		Subtracting this identity from the expression for $\mathcal E_{\widetilde K}(W,\varphi)$ gives
		\begin{align*}
			\mathcal E_{\widetilde K}(W,\varphi)
			&=
			2\lim_{\varepsilon\to 0}
			\int_{B_{1/2}}\varphi(x)
			\int_{\{\varepsilon<|x-y|<\varepsilon^{-1}\}}
			\Bigl\{
			\delta\bar z(x,y)
			[\widetilde K(x,y)-K_A(x,y)]\\
			&\hspace{49mm}
			+
			(1-\zeta(y))\bar z(y)\widetilde K(x,y)
			\Bigr\}
			\,\diff y\diff x.
		\end{align*}
		For $0<\varepsilon<1$ and $x\in B_{1/2}$, set
		\[
		\begin{aligned}
			F_\varepsilon(x)
			:=
			2\int_{\{\varepsilon<|x-y|<\varepsilon^{-1}\}}
			\Bigl\{
			\delta\bar z(x,y)
			\bigl[\widetilde K(x,y)-K_A(x,y)\bigr]+
			(1-\zeta(y))\bar z(y)\widetilde K(x,y)
			\Bigr\}\,\diff y.
		\end{aligned}
		\]
		Then the preceding identity becomes
		\begin{equation}\label{eq:truncated-localized-equation}
			\mathcal E_{\widetilde K}(W,\varphi)
			=
			\lim_{\varepsilon\to 0}
			\int_{B_{1/2}}F_\varepsilon(x)\varphi(x)\,\diff x.
		\end{equation}
		Note that if $x\in B_{1/2}$ and $y\in B_{5/8}$, then $\widetilde K(x,y)=K_A(x,y)$. Similarly, if $y\in B_{9/16}$, then $\zeta(y)=1$. Accordingly, for $x\in B_{1/2}$,  we define
		\begin{equation*}
			\begin{aligned}
				F(x)&:=F_1(x)+F_2(x)\\
				&=
				2\int_{\R^n\setminus B_{5/8}}
				\delta\bar z(x,y)
				[\widetilde K(x,y)-K_A(x,y)]
				\,\diff y
				+
				2\int_{\R^n\setminus B_{9/16}}
				(1-\zeta(y))\bar z(y)
				\widetilde K(x,y)\,\diff y.
			\end{aligned}
		\end{equation*}
	
		\smallskip
		\noindent\textbf{Step 7: estimate of \(F\).}
		We show that the two terms in \(F=F_1+F_2\) are uniformly bounded in
		\(B_{1/2}\). The same estimates will also provide the integrable bounds
		needed to remove the truncation in \(F_\varepsilon\).
		
		From \eqref{eq:large-slope-kernel} and
		\[
		|X+Y|^{p-2}
		\le
		C\bigl(|X|^{p-2}+|Y|^{p-2}\bigr),
		\]
		we obtain
		\begin{equation}\label{eq:KA-upper-localization}
			K_A(x,y)
			\le
			C|x-y|^{-n-1-\beta}
			+
			CA^{2-p}
			\frac{|\delta z(x,y)|^{p-2}}
			{|x-y|^{n+sp}}.
		\end{equation}
		Hence
		\begin{equation}\label{eq:localized-kernel-basic}
			K_A(x,y)|\delta z(x,y)|
			\le
			C\frac{|\delta z(x,y)|}
			{|x-y|^{n+1+\beta}}
			+
			CA^{2-p}
			\frac{|\delta z(x,y)|^{p-1}}
			{|x-y|^{n+sp}}.
		\end{equation}
		Moreover, by the construction of \(\widetilde K\),
		\[
		0\le\widetilde K\le K_A+K_0,
		\qquad
		|\widetilde K-K_A|\le K_A+K_0,
		\]
		and
		\[
		K_0(x,y)\le C|x-y|^{-n-1-\beta}.
		\]
	
	We first  prove two tail bounds that will be used below. Since
		\(|z(0)|\le1\), \eqref{eq:large-slope-normalization} gives
		\begin{align}	\label{eq:bar-z-p-tail}
			\int_{|y|>1}
			\frac{|\bar z(y)|^{p-1}}
			{|y|^{n+sp}}\,\diff y
			&\le
			C\Tail_{p-1,sp}(z;0,1)^{p-1}
			+
			C|z(0)|^{p-1}
			\int_{|y|>1}|y|^{-n-sp}\,\diff y
			\nonumber\\
			&\le C.
		\end{align}
		Since \(sp=p-1+\beta\), H\"older's inequality then yields
		\begin{align}	\label{eq:linear-weighted-tail}
			\int_{|y|>1}
			\frac{|\bar z(y)|}
			{|y|^{n+1+\beta}}\,\diff y
			&\le
			\left(
			\int_{|y|>1}
			\frac{|\bar z(y)|^{p-1}}
			{|y|^{n+sp}}\,\diff y
			\right)^{1/(p-1)}
			\left(
			\int_{|y|>1}
			|y|^{-n-\beta}\,\diff y
			\right)^{(p-2)/(p-1)}
			\nonumber\\
			&\le C.
		\end{align}
		
		Fix \(x\in B_{1/2}\). We have
		\[
		|\bar z(x)|\le2,
		\qquad
		|\bar z(y)|\le2
		\quad (y\in B_1),
		\qquad
		A^{2-p}\le1.
		\]
		Moreover,   we have  $
		|x-y|\ge\frac18$  for $y\in B_1\setminus B_{5/8}$,	and
	$	|x-y|\ge\frac1{16}$ for $y\in B_1\setminus B_{9/16}$.
		If \(|y|>1\), then $
		|x-y|
		\ge |y|-|x|
		\ge\frac12|y|$.

		Using \(\delta z=\delta\bar z\),
		\eqref{eq:localized-kernel-basic}, and
		\(|\widetilde K-K_A|\le K_A+K_0\), we obtain from \eqref{eq:bar-z-p-tail} and \eqref{eq:linear-weighted-tail} that
		\begin{align*}
			|F_1(x)|& \leq 2\int_{\R^n\setminus B_{5/8}}
			|\delta\bar z(x,y)|
			\,|\widetilde K(x,y)-K_A(x,y)|\,\diff y\\
			&\le
			C\int_{\R^n\setminus B_{5/8}}
			\frac{|\bar z(x)-\bar z(y)|}
			{|x-y|^{n+1+\beta}}\,\diff y
			+
			CA^{2-p}
			\int_{\R^n\setminus B_{5/8}}
			\frac{|\bar z(x)-\bar z(y)|^{p-1}}
			{|x-y|^{n+sp}}\,\diff y\\
			&\le
			C
			+
			C\int_{|y|>1}
			\frac{|\bar z(y)|}
			{|y|^{n+1+\beta}}\,\diff y
			+
			C\int_{|y|>1}
			\frac{|\bar z(y)|^{p-1}}
			{|y|^{n+sp}}\,\diff y
			\le C.
		\end{align*}
		Similarly, by \eqref{eq:KA-upper-localization} and
		\(0\le\widetilde K\le K_A+K_0\), we have
		\begin{align*}
		|F_2(x)|&\leq 2\int_{\R^n\setminus B_{9/16}}
			(1-\zeta(y))|\bar z(y)|
			\widetilde K(x,y)\,\diff y\\
			&\le
			C\int_{\R^n\setminus B_{9/16}}
			\frac{|\bar z(y)|}
			{|x-y|^{n+1+\beta}}\,\diff y+
			CA^{2-p}
			\int_{\R^n\setminus B_{9/16}}
			\frac{
				|\bar z(y)|
				|\bar z(x)-\bar z(y)|^{p-2}}
			{|x-y|^{n+sp}}\,\diff y.
		\end{align*}
		Since
		\[
		|\bar z(y)|
		|\bar z(x)-\bar z(y)|^{p-2}
		\le
		C\bigl(
		|\bar z(y)|^{p-1}
		+
		|\bar z(x)|^{p-1}
		\bigr),
		\]
it follows from \eqref{eq:bar-z-p-tail} and \eqref{eq:linear-weighted-tail}  that 
		\begin{align*}
			|F_2(x)| \le
			C
			+
			C\int_{|y|>1}
			\frac{|\bar z(y)|}
			{|y|^{n+1+\beta}}\,\diff y
			+
			C\int_{|y|>1}
			\frac{
				|\bar z(y)|^{p-1}
				+
				|\bar z(x)|^{p-1}}
			{|y|^{n+sp}}\,\diff y
			\le C.
		\end{align*}
	
The two estimates above are uniform for \(x\in B_{1/2}\) and, more
		importantly, they bound the absolute values of the two integrands defining
		\(F\). Hence
		\begin{equation}\label{eq:F-uniform-bound}
			|F_\varepsilon(x)|+|F(x)|
			\le C
			\qquad
			(x\in B_{1/2},\ 0<\varepsilon<1).
		\end{equation}
		For each fixed \(x\in B_{1/2}\), the truncated integrand defining
		\(F_\varepsilon(x)\) converges pointwise to the integrand defining \(F(x)\).
		The two estimates above provide an integrable bound in \(y\), independent of
		\(\varepsilon\). Therefore, by the dominated convergence theorem,
		\[
		F_\varepsilon(x)\longrightarrow F(x)
		\qquad
		\text{as }\varepsilon\to0.
		\]
		
		Finally, \eqref{eq:F-uniform-bound} gives
		\[
		|F_\varepsilon(x)\varphi(x)|
		\le C|\varphi(x)|.
		\]
	Applying the dominated convergence theorem in \eqref{eq:truncated-localized-equation}, we obtain
		\begin{equation}\label{eq:localized-equation}
			\mathcal E_{\widetilde K}(W,\varphi)
			=
			\int_{B_{1/2}}F(x)\varphi(x)\,\diff x
			\qquad
			\text{for every }\varphi\in C_c^\infty(B_{1/2}).
		\end{equation}
		\smallskip
		\noindent\textbf{Step 8: the Schauder estimate.}
		We take $
		\alpha_0
		:=
		\frac12\min\{\alpha,\eta,\beta\}$.
		Since $2\mathfrak s=1+\beta\in(1,2)$ and $n\ge2$, we have
		\[
		n>2\mathfrak s,
		\qquad
		0<\alpha_0<2\mathfrak s-1,
		\qquad
		q_{\alpha_0}
		=
		\frac{n}{2\mathfrak s-(1+\alpha_0)}
		=
		\frac{n}{\beta-\alpha_0}.
		\]
		By Lemma~\ref{lem:kernel-exponent-downgrade} and \eqref{eq:Ktilde-class},
		\[
		\widetilde K
		\in
		\mathscr K
		(2\mathfrak s,\alpha_0;
		\lambda_0/2,\Lambda_*,M_*),
		\]
		after replacing $M_*$ by $\max\{M_*,2\Lambda_*\}$. Furthermore,
		\[
		W\in
		\mathcal H_{\widetilde K}(B_{1/2})
		\cap
		C_{\mathrm{loc}}^{1,\alpha_0}(B_{1/2})
		\cap
		L^\infty(\R^n),
		\quad
		\|W\|_{L^\infty(\R^n)}
		\le2,
		\quad
		\|F\|_{L^{q_{\alpha_0}}(B_{1/2})}
		\le C.
		\]
		For each fixed $A$, Step~6 supplies the required a priori regularity.
        The kernel parameters and the two displayed norms are bounded
        independently of $A$. The Schauder estimate therefore gives a bound
        independent of $A$. Since $W=\bar z$ on $B_{9/16}$,
        applying Lemma~\ref{lem:fixed-radius-schauder} to
        \eqref{eq:localized-equation} gives
		\begin{equation}\label{eq:large-slope-remainder-C1a}
			\|\nabla z\|_{L^\infty(B_{1/4})}
			+
			[\nabla z]_{C^{0,\alpha_0}(B_{1/4})}
			\le C.
		\end{equation}
		
		\smallskip
		\noindent\textbf{Step 9: affine decay.}
		Let
		\[
		L(x):=z(0)+\nabla z(0)\cdot x.
		\]
		For $x\in B_{1/4}$, \eqref{eq:large-slope-remainder-C1a} and the fundamental theorem of calculus give
		\begin{align}
			\label{eq:large-slope-taylor}
			|z(x)-L(x)|
			=
			\left|
			\int_0^1
			\bigl[\nabla z(tx)-\nabla z(0)\bigr]
			\cdot x\,\diff t
			\right|\le
			C|x|^{1+\alpha_0}.
		\end{align}
		Moreover, $|z(0)|\le1$ and $|\nabla z(0)|\le C$, so
		\[
		|z(y)-L(y)|^{p-1}
		\le
		C\bigl(|z(y)|^{p-1}+1+|y|^{p-1}\bigr).
		\]
		Since $z$ and $L$ are bounded on $B_1\setminus B_{1/4}$, \eqref{eq:large-slope-normalization} yields
		\begin{align}
			\int_{\R^n\setminus B_{1/4}}
			\frac{|z(y)-L(y)|^{p-1}}
			{|y|^{n+sp}}\,\diff y
			\le
			C.
			\label{eq:large-slope-far-tail}
		\end{align}
		The bounds \eqref{eq:large-slope-taylor} and
        \eqref{eq:large-slope-far-tail} verify the hypotheses of
        Lemma~\ref{lem:taylor-affine-tail} for $g=z$, $H=C$, and
        $\alpha=\alpha_0$. Take
		\[
		\sigma_\ell
		:=
		\frac12
		\min\left\{
		\alpha_0,\frac{\beta}{p-1}
		\right\}>0.
		\]
		Then
		\[
		\frac1\rho
		\|z-L\|_{L^\infty(B_\rho)}
		+
		\frac1\rho
		\Tail_{p-1,sp}(z-L;0,\rho)
		\le
		C\rho^{\sigma_\ell},
		\qquad
		0<\rho\le\frac1{16}.
		\]
		Since
		\[
		v-\bigl(a\cdot x+L(x)\bigr)=z-L
		\]
		and $a\cdot x+L(x)$ is affine, the definition of $\Psi_\infty$ gives
		\[
		\Psi_\infty(v;0,\rho)
		\le
		C\rho^{\sigma_\ell},
		\qquad
		0<\rho\le\frac1{16}.
		\]
		Let $\sigma_c$ be the exponent in Lemma~\ref{lem:controlled-affine-decay} with $M=M_0$, and set
		\[
		\sigma:=\min\{\sigma_c,\sigma_\ell\}.
		\]
		The controlled slope and large slope estimates therefore combine to yield
		\[
		\Psi_\infty(v;0,\rho)
		\le
		C\rho^\sigma,
		\qquad
		0<\rho\le\frac1{16}.
		\]
		Finally, \eqref{eq:affine-excess-scaling} gives
		\[
		\Psi_\infty(v;x_0,\rho)
		\le
		C\left(\frac{\rho}{r}\right)^\sigma
		\Psi_\infty(v;x_0,r),
		\qquad
		0<\rho\le\frac{r}{32}.
		\]
	\end{proof}
	
	\subsection{Affine decay in \texorpdfstring{$L^{p-1}$}{L(p-1)} average}
	
	We next pass from the uniform affine decay in
    Lemma~\ref{lem:full-affine-decay-high} to decay for the averaged excess
    \eqref{eq:mean-affine-excess}. The nonlocal tail remains unchanged.
    The task is to replace the local supremum by an integral  average.
	
	\begin{lemma}\label{lem:mean-homogeneous-decay}
		Let
		$v\in W^{s,p}_{\mathrm{loc}}(B_{2r}(x_0))\cap L_{sp}^m(\R^n)
		\cap L^\infty(B_{2r}(x_0))$
		be a local weak solution of $(-\Delta_p)^sv=0$ in $B_{2r}(x_0)$.
		There exist $\sigma\in(0,1)$ and $C\ge1$, depending only on $n,p,s$,
		such that
		\begin{equation}\label{eq:mean-homogeneous-decay}
			\Psi_m(v;x_0,\rho)
			\le C\left(\frac\rho r\right)^\sigma\Psi_m(v;x_0,r),
			\qquad 0<\rho\le r/64.
		\end{equation}
	\end{lemma}
	\begin{proof}
Take
	\[
	V(y):=\frac{v(x_0+ry)}{r}.
	\]
	Then $V$ is fractional $p$-harmonic in $B_2$ and satisfies the same
	regularity assumptions as $v$. For every $t\in(0,1]$, a change of
	variables in the definition of $\Psi_m$ gives
	\[
	\Psi_m(V;0,t)=\Psi_m(v;x_0,rt).
	\]
	Thus it is enough to prove
	\begin{equation}\label{eq:mean-homogeneous-decay-normalized}
	\Psi_m(V;0,t)\le Ct^\sigma\Psi_m(V;0,1),
	\qquad 0<t\le\frac1{64},
	\end{equation}
	since \eqref{eq:mean-homogeneous-decay} follows by taking
	$t=\rho/r$. Relabeling $V$ as $v$, we may therefore assume
		\[
		v\in W_{\mathrm{loc}}^{s,p}(B_2)
		\cap L_{sp}^m(\R^n)
		\cap L^\infty(B_2),
		\qquad
		(-\Delta_p)^sv=0
		\quad\text{in }B_2.
		\]
		
		We first estimate the gradient of the remainder after subtracting an arbitrary affine function. Fix $\xi\in B_2$ and $d>0$ such that	$B_{2d}(\xi)\Subset B_2$.
		Lemma~\ref{lem:homogeneous-C1a} applies in $B_{2d}(\xi)$ and gives
        a representative of $v$ in $C^{1,\alpha}(B_d(\xi))$.
        We use this representative below.
		
		Let $\ell$ be an arbitrary affine function and  $
		z:=v-\ell$.
		Choose $\tau\in(0,1/32)$ so small that the constants in
		Lemma~\ref{lem:full-affine-decay-high} satisfy  $
		C\tau^\sigma\le\frac12$.

		We take
		\[
		d_j:=\tau^jd,
		\qquad
		E_j:=\Psi_\infty(v;\xi,d_j),
		\qquad j=0,1,2,\ldots,
		\]
		and let $L_j$ be a minimizing affine function for
		$\Psi_\infty(v;\xi,d_j)$.	Since $
		d_{j+1}=\tau d_j\le\frac{d_j}{32}$, Lemma~\ref{lem:full-affine-decay-high} gives
		\[
		E_{j+1}
		\le
		C\tau^\sigma E_j
		\le
		\frac12E_j.
		\]
		Consequently, we obtain
		\begin{equation}\label{eq:mean-uniform-geometric-decay}
			E_j\le2^{-j}E_0,
			\qquad
			\sum_{j=0}^{\infty}E_j\le2E_0.
		\end{equation}
		
		By the definition of $E_j$, we have
		\begin{equation}\label{eq:mean-local-Lj-approximation}
			\|v-L_j\|_{L^\infty(B_{d_j}(\xi))}
			\le d_jE_j.
		\end{equation}
		Since
		$B_{d_{j+1}}(\xi)\subset B_{d_j}(\xi)$,  we deduce that 
		\[
		\begin{aligned}
			\|L_{j+1}-L_j\|_{L^\infty(B_{d_{j+1}}(\xi))}
			&\le
			\|L_{j+1}-v\|_{L^\infty(B_{d_{j+1}}(\xi))}
			+
			\|v-L_j\|_{L^\infty(B_{d_{j+1}}(\xi))}\\
			&\le
			d_{j+1}E_{j+1}+d_jE_j.
		\end{aligned}
		\]
		
		We next estimate the difference of the slopes.
		Fix $\nu\in\mathbb S^{n-1}$ and $0<h<d_{j+1}$.
		Since $L_{j+1}-L_j$ is affine,
		\[
		2h(\nabla L_{j+1}-\nabla L_j)\cdot\nu
		=
		(L_{j+1}-L_j)(\xi+h\nu)
		-
		(L_{j+1}-L_j)(\xi-h\nu).
		\]
		Both points belong to $B_{d_{j+1}}(\xi)$, and therefore
		\[
		\begin{aligned}
			|(\nabla L_{j+1}-\nabla L_j)\cdot\nu|
			&\le
			\frac{
				|(L_{j+1}-L_j)(\xi+h\nu)|
				+
				|(L_{j+1}-L_j)(\xi-h\nu)|
			}{2h}\\
			&\le
			\frac1h
			\|L_{j+1}-L_j\|_{L^\infty(B_{d_{j+1}}(\xi))}\le
			\frac{d_jE_j+d_{j+1}E_{j+1}}{h}.
		\end{aligned}
		\]
		Letting $h\to d_{j+1}$ and then taking the supremum over
		$\nu\in\mathbb S^{n-1}$, we obtain
		\begin{equation}\label{eq:mean-successive-slopes}
			\begin{aligned}
				|\nabla L_{j+1}-\nabla L_j|
				\le
				\frac{d_jE_j+d_{j+1}E_{j+1}}{d_{j+1}}=
				\tau^{-1}E_j+E_{j+1}.
			\end{aligned}
		\end{equation}
		Using \eqref{eq:mean-uniform-geometric-decay}, we have
		\[
		\sum_{j=0}^{\infty}|\nabla L_{j+1}-\nabla L_j|
		\le
		(\tau^{-1}+1)\sum_{j=0}^{\infty}E_j
		\le
		2(\tau^{-1}+1)E_0.
		\]
		Hence the sequence $\{\nabla L_j\}$ converges.
		
		We now identify its limit with $\nabla v(\xi)$.
		Fix $\nu\in\mathbb S^{n-1}$ and set $h_j:=\frac{d_j}{2}$.  Since
		$v\in C^1(B_d(\xi))$, the fundamental theorem of calculus gives
		\[
		v(\xi+h_j\nu)-v(\xi-h_j\nu)
		=
		\int_{-h_j}^{h_j}
		\nabla v(\xi+t\nu)\cdot\nu\,\diff t.
		\]
	Therefore, 	using \eqref{eq:mean-local-Lj-approximation} and
		$2h_j=d_j$, we obtain
		\[
		\begin{aligned}
			&|(\nabla L_j-\nabla v(\xi))\cdot\nu|\\
				&\quad \le
			\frac{
				|(L_j-v)(\xi+h_j\nu)|
				+
				|(L_j-v)(\xi-h_j\nu)|
			}{2h_j}+
			\frac1{2h_j}
			\int_{-h_j}^{h_j}
			|\nabla v(\xi+t\nu)-\nabla v(\xi)|\,\diff t\\
			&\quad \le
			\frac{2d_jE_j}{2h_j}
			+
			\sup_{y\in B_{h_j}(\xi)}
			|\nabla v(y)-\nabla v(\xi)|=
			2E_j
			+
			\sup_{y\in B_{h_j}(\xi)}
			|\nabla v(y)-\nabla v(\xi)|.
		\end{aligned}
		\]
		The right hand side is independent of $\nu$. Hence
		\[
		|\nabla L_j-\nabla v(\xi)|
		\le
		2E_j
		+
		\sup_{y\in B_{h_j}(\xi)}
		|\nabla v(y)-\nabla v(\xi)|.
		\]
		Since $E_j\to0$, $h_j\to0$, and $\nabla v$ is continuous at $\xi$,
		\begin{equation}\label{eq:mean-slope-gradient-limit}
			\nabla L_j\longrightarrow \nabla v(\xi).
		\end{equation}
		
	We next compare \(\nabla L_0\) with the slope of \(\ell\). Since
	\[
	L_0-\ell=(L_0-v)+(v-\ell),
	\]
	we have
	\[
	\|L_0-\ell\|_{L^\infty(B_d(\xi))}
	\le
	dE_0+\|z\|_{L^\infty(B_d(\xi))}.
	\]
	Since \(L_0-\ell\) is affine, we deduce that 
	\[
	|\nabla L_0-\nabla\ell|
	\le
	\frac1d\|L_0-\ell\|_{L^\infty(B_d(\xi))}
	\le
	E_0+\frac1d\|z\|_{L^\infty(B_d(\xi))}.
	\]
	Thus, by \eqref{eq:mean-successive-slopes},
		\eqref{eq:mean-uniform-geometric-decay}, and
		\eqref{eq:mean-slope-gradient-limit}, we obtain
		\[
		\begin{aligned}
				|\nabla z(\xi)|
			=
			|\nabla v(\xi)-\nabla \ell| &\le
			|\nabla v(\xi)-\nabla L_0| +|\nabla L_0-\nabla \ell|\\
			&\le
			\sum_{j=0}^{\infty}|\nabla L_{j+1}-\nabla L_j|+|\nabla L_0-\nabla \ell| \le CE_0+\frac{C}{d}
			\|z\|_{L^\infty(B_d(\xi))}.
		\end{aligned}
		\]
	Since $\ell$ is an admissible affine function in the definition of
		$E_0=\Psi_\infty(v;\xi,d)$,
		\[
		E_0
		\le
		\frac1d
		\left[
		\|z\|_{L^\infty(B_d(\xi))}
		+
		\Tail_{m,sp}(z;\xi,d)
		\right].
		\]
	Therefore,	we conclude that
		\begin{equation}\label{eq:mean-remainder-gradient}
			|\nabla z(\xi)|
			\le
			\frac Cd
			\left[
			\|z\|_{L^\infty(B_d(\xi))}
			+
			\Tail_{m,sp}(z;\xi,d)
			\right].
		\end{equation}
		
		We now use \eqref{eq:mean-remainder-gradient} to control the local
		supremum of $z$ by its $L^m$ norm and its tail. Take
		\[
		H
		:=
		\left(
		\fint_{B_1}|z|^m\,\diff x
		\right)^{1/m}
		+
		\Tail_{m,sp}(z;0,1),
		\qquad
		M(t):=\|z\|_{L^\infty(B_t)}.
		\]
		Fix
		\[
		\frac12\le t<T\le\frac34,
		\qquad
		d:=\frac{T-t}{4}.
		\]
		If $\xi\in B_{t+2d}$, then
		\[
		B_d(\xi)\subset B_T,
		\qquad
		B_{2d}(\xi)\subset B_T\subset B_2.
		\]
		Splitting the tail at radius $1$ gives 
		\[
		\begin{aligned}
			\Tail_{m,sp}(z;\xi,d)^m
			={}&
			d^{sp}
			\int_{B_1\setminus B_d(\xi)}
			\frac{|z(y)|^m}{|y-\xi|^{n+sp}}
			\,\diff y+
			d^{sp}
			\int_{\R^n\setminus B_1}
			\frac{|z(y)|^m}{|y-\xi|^{n+sp}}
			\,\diff y.
		\end{aligned}
		\]
		On $B_1\setminus B_d(\xi)$,
		$|y-\xi|\ge d$, and hence
		\[
		d^{sp}
		\int_{B_1\setminus B_d(\xi)}
		\frac{|z(y)|^m}{|y-\xi|^{n+sp}}
		\,\diff y
		\le
		d^{-n}\int_{B_1}|z(y)|^m\,\diff y.
		\]
		Moreover, since $\xi\in B_{t+2d}\subset B_{3/4}$,
		\[
		|y-\xi|\ge\frac{|y|}{4}
		\qquad\text{for }y\in\R^n\setminus B_1.
		\]
		Thus
		\[
		d^{sp}
		\int_{\R^n\setminus B_1}
		\frac{|z(y)|^m}{|y-\xi|^{n+sp}}
		\,\diff y
		\le
		Cd^{sp}
		\int_{\R^n\setminus B_1}
		\frac{|z(y)|^m}{|y|^{n+sp}}
		\,\diff y.
		\]
		Since $d<1$, we obtain
		\begin{equation}\label{eq:mean-tail-change-centre-proof}
			\Tail_{m,sp}(z;\xi,d)^m
			\le
			Cd^{-n}H^m.
		\end{equation}
		
		Applying \eqref{eq:mean-remainder-gradient} at every
		$\xi\in B_{t+2d}$ and using
		\eqref{eq:mean-tail-change-centre-proof}, we get
		\begin{equation}\label{eq:mean-local-gradient-interpolation}
			\|\nabla z\|_{L^\infty(B_{t+2d})}
			\le
			Cd^{-1}M(T)
			+
			Cd^{-1-n/m}H.
		\end{equation}
		
		Fix $x\in B_t$ and $0<\varepsilon<d$.
		For every $y\in B_\varepsilon(x)$, the line segment joining $x$ and
		$y$ is contained in $B_{t+2d}$. Hence
		\[
		|z(x)|
		\le
		|z(y)|
		+
		\varepsilon
		\|\nabla z\|_{L^\infty(B_{t+2d})}.
		\]
		Averaging over $B_\varepsilon(x)$ and using H\"older's inequality,
		\[
		\begin{aligned}
			|z(x)|
			\le
			\fint_{B_\varepsilon(x)}|z(y)|\,\diff y
			+
			\varepsilon
			\|\nabla z\|_{L^\infty(B_{t+2d})}\le
			C\varepsilon^{-n/m}H
			+
			\varepsilon
			\|\nabla z\|_{L^\infty(B_{t+2d})}.
		\end{aligned}
		\]
		Using \eqref{eq:mean-local-gradient-interpolation} and taking the
		supremum over $x\in B_t$, we obtain
		\[
		M(t)
		\le
		C\varepsilon^{-n/m}H
		+
		C\varepsilon d^{-1}M(T)
		+
		C\varepsilon d^{-1-n/m}H.
		\]
		
		Let $\nu\in(0,1)$. Choose
		\[
		\varepsilon=c_\nu d,
		\]
		where $c_\nu\in(0,1)$ is sufficiently small that
		$Cc_\nu\le\nu$. Since $d=(T-t)/4$, it follows that
		\begin{equation}\label{eq:mean-radius-iteration}
			M(t)
			\le
			\nu M(T)
			+
			C_\nu(T-t)^{-n/m}H.
		\end{equation}
		
		Choose $\nu<2^{-n/m}$ and set
		\[
		t_j
		:=
		\frac12+\frac14(1-2^{-j}),
		\qquad j=0,1,2,\ldots.
		\]
		Then
		\[
		t_0=\frac12,
		\qquad
		t_j\to\frac34,
		\qquad
		t_{j+1}-t_j=2^{-j-3}.
		\]
		Applying \eqref{eq:mean-radius-iteration} successively with
		$t=t_j$ and $T=t_{j+1}$ gives
		\[
		M(1/2)
		\le
		\nu^NM(t_N)
		+
		CH
		\sum_{j=0}^{N-1}
		\bigl(\nu2^{n/m}\bigr)^j.
		\]
		Since $M(t_N)\le M(3/4)<\infty$ for the fixed affine function
		$\ell$, letting $N\to\infty$ yields
		\begin{equation}\label{eq:mean-local-supremum}
			M(1/2)\le CH.
		\end{equation}
		The constant in \eqref{eq:mean-local-supremum} is independent of
		the slope of $\ell$.
		
		We also have
		\[
		\begin{aligned}
			\Tail_{m,sp}(z;0,1/2)^m
			={}&
			2^{-sp}
			\int_{B_1\setminus B_{1/2}}
			\frac{|z(y)|^m}{|y|^{n+sp}}
			\,\diff y+
			2^{-sp}
			\int_{\R^n\setminus B_1}
			\frac{|z(y)|^m}{|y|^{n+sp}}
			\,\diff y\le CH^m.
		\end{aligned}
		\]
		Consequently,
		\[
		\|z\|_{L^\infty(B_{1/2})}
		+
		\Tail_{m,sp}(z;0,1/2)
		\le CH.
		\]
		
		Now choose $\ell$ to be a minimizing affine function for
		$\Psi_m(v;0,1)$. Then
		\[
		H=\Psi_m(v;0,1),
		\]
		and therefore
		\begin{equation}\label{eq:mean-to-uniform-affine}
			\Psi_\infty(v;0,1/2)
			\le
			C\Psi_m(v;0,1).
		\end{equation}
		
		Finally, since
		\[
		\Psi_m(v;0,\rho)\le\Psi_\infty(v;0,\rho),
		\]
		Lemma~\ref{lem:full-affine-decay-high}, applied with initial radius
		$1/2$, gives for $0<\rho\le1/64$,
		\[
		\begin{aligned}
			\Psi_m(v;0,\rho)
			\le
			\Psi_\infty(v;0,\rho)\le
			C(2\rho)^\sigma
			\Psi_\infty(v;0,1/2)\le
			C\rho^\sigma
			\Psi_m(v;0,1).
		\end{aligned}
		\]
		This proves \eqref{eq:mean-homogeneous-decay-normalized} and completes the proof.
	\end{proof}

\section{Weak gradients and pointwise potential estimates}
\label{sec:mean-gradient}
In this section, we prove Theorem~\ref{thm:mean-gradient}.
We first combine the homogeneous affine decay from
Section~\ref{sec:affine-excess} with the comparison estimate in
Lemma~\ref{lem:mean-measure-comparison} to derive an affine
recurrence for SOLA. We then use this recurrence to control
the gradients of mollifications and establish
$u\in W^{1,m}_{\mathrm{loc}}(\Omega)$.
Finally, we identify the limiting affine slope at Lebesgue
points of the weak gradient and obtain the pointwise estimate.

\subsection{The affine recurrence for SOLA}
For $g\in L_{sp}^m(\R^n)$ and $0<\rho\le r$, restriction of the local
integral and splitting the tail at $r$ give
\begin{align}
&\frac1\rho\left(\fint_{B_\rho(x_0)}|g|^m\,\diff x\right)^{1/m}
 +\frac1\rho\Tail_{m,sp}(g;x_0,\rho)\notag\\
&\qquad\le C\left(\frac r\rho\right)^{1+n/m}
 \left[\frac1r\left(\fint_{B_r(x_0)}|g|^m\,\diff x\right)^{1/m}
 +\frac1r\Tail_{m,sp}(g;x_0,r)\right].
\label{eq:mean-radius-restriction}
\end{align}
Indeed, on $B_r(x_0)\setminus B_\rho(x_0)$ the kernel is bounded by
$\rho^{-n-sp}$. The outer part of the normalized tail has the factor
$(\rho/r)^{sp/m-1}\le1$.

\begin{lemma}[SOLA convergence at a fixed scale]
\label{lem:sola-fixed-scale}
Let $u_j\to u$ be as in Definition~\ref{def:sola}. For every fixed
$B_r(x_0)\Subset\Omega$,
\begin{equation}\label{eq:sola-Lm-tail-convergence}
\|u_j-u\|_{L^m(B_r(x_0))}
+\Tail_{m,sp}(u_j-u;x_0,r)\longrightarrow0.
\end{equation}
Consequently,
\begin{equation}\label{eq:sola-psi-convergence}
\Psi_m(u_j;x_0,r)\longrightarrow\Psi_m(u;x_0,r).
\end{equation}
\end{lemma}
\begin{proof}
By \eqref{eq:sola-tail-class}, $u_j\to u$ in $L^m(\R^n)$, and
\[
\Tail_{m,sp}(u_j-u;x_0,r)^m
\le r^{-n}\|u_j-u\|_{L^m(\R^n)}^m\longrightarrow0.
\]
For an arbitrary affine function $\ell$, Minkowski's inequality gives
\[
\begin{aligned}
	&\frac1r\left(\fint_{B_r(x_0)}|u_j-\ell|^m\,\diff x\right)^{1/m}
	+\frac1r\Tail_{m,sp}(u_j-\ell;x_0,r)\\
	&\qquad\le
	\frac1r\left(\fint_{B_r(x_0)}|u-\ell|^m\,\diff x\right)^{1/m}
	+\frac1r\Tail_{m,sp}(u-\ell;x_0,r)\\
	&\qquad\quad+
	\frac1r\left(\fint_{B_r(x_0)}|u_j-u|^m\,\diff x\right)^{1/m}
	+\frac1r\Tail_{m,sp}(u_j-u;x_0,r).
\end{aligned}
\]
Taking the infimum over all affine functions $\ell$ yields
\[
\Psi_m(u_j;x_0,r)-\Psi_m(u;x_0,r)
\le A_j,
\]
where
\[
A_j:=
\frac1r\left(\fint_{B_r(x_0)}|u_j-u|^m\,\diff x\right)^{1/m}
+\frac1r\Tail_{m,sp}(u_j-u;x_0,r).
\]
Interchanging $u_j$ and $u$ gives the reverse inequality. Hence
\[
|\Psi_m(u_j;x_0,r)-\Psi_m(u;x_0,r)|
\le A_j.
\]
Since $A_j\to0$, we obtain
\[
\Psi_m(u_j;x_0,r)\to\Psi_m(u;x_0,r).
\]
\end{proof}

\begin{lemma}\label{lem:mean-one-step}
There exist $\theta\in(0,1/512)$ and $C\ge1$, depending only on $n,p,s$,
such that every SOLA $u$ satisfies
\begin{equation}\label{eq:mean-one-step}
\Psi_m(u;x_0,\theta r)
\le\frac12\Psi_m(u;x_0,r)+CD_\mu(x_0,2r)
\end{equation}
whenever $B_{2r}(x_0)\Subset\Omega$.
\end{lemma}
\begin{proof}
	We first prove the estimate for an energy solution. 	Let
	$U\in W^{s,p}(\R^n)$
	solve
	\[
	(-\Delta_p)^sU=\nu
	\qquad\text{weakly in }B_{2r}(x_0),
	\]
	where $\nu$ is a finite signed Radon measure. Let $v$ be the
	fractional $p$-harmonic replacement of $U$ in $B_r(x_0)$, and set
	\[
	E:=\Psi_m(U;x_0,r),
	\qquad
	D:=D_\nu(x_0,r).
	\]
	Let
	\[
	\ell(x)=b+a\cdot(x-x_0)
	\]
	be a minimizing affine function for $\Psi_m(U;x_0,r)$. Thus
	\begin{equation}\label{eq:one-step-energy-minimizer}
		\frac1r
		\left[
		\left(\fint_{B_r(x_0)}|U-\ell|^m\,\diff x\right)^{1/m}
		+
		\Tail_{m,sp}(U-\ell;x_0,r)
		\right]
		=E.
	\end{equation}
	
	We first estimate the affine excess of $v$ at radius $r/4$.
	Since
	\[
	v-\ell=(U-\ell)+(v-U),
	\]
	Minkowski's inequality gives
	\begin{align*}
		\Psi_m(v;x_0,r/4)
		&\le
		\frac4r
		\left[
		\left(
		\fint_{B_{r/4}(x_0)}
		|U-\ell|^m\,\diff x
		\right)^{1/m}
		+
		\Tail_{m,sp}(U-\ell;x_0,r/4)
		\right]\\
		&\quad+
		\frac4r
		\left[
		\left(
		\fint_{B_{r/4}(x_0)}
		|U-v|^m\,\diff x
		\right)^{1/m}
		+
		\Tail_{m,sp}(U-v;x_0,r/4)
		\right].
	\end{align*}
	By \eqref{eq:mean-radius-restriction} and
	\eqref{eq:one-step-energy-minimizer}, the first bracket, multiplied by $4/r$, is bounded by $CE$. By
	Lemma~\ref{lem:mean-measure-comparison}, applied with
	$\rho=r/4$, the second one is bounded by $CD$. Therefore
	\begin{equation}\label{eq:replacement-mean-excess-r4}
		\Psi_m(v;x_0,r/4)
		\le C(E+D).
	\end{equation}
	
	The function $v$ belongs to $W^{s,p}(\R^n)$, is fractional
	$p$-harmonic in $B_r(x_0)$, and is bounded in $B_{r/2}(x_0)$ by
	Lemma~\ref{lem:local-boundedness}. Hence
	Lemma~\ref{lem:mean-homogeneous-decay} can be applied with initial
	radius $r/4$. If $ 0<\theta<\frac1{512}$,
	then $\theta r<(r/4)/64$, and consequently
	\begin{align}	\label{eq:replacement-decay-small-scale}
		\begin{split}
		\Psi_m(v;x_0,\theta r)
		&\le
		C
		\left(
		\frac{\theta r}{r/4}
		\right)^\sigma
		\Psi_m(v;x_0,r/4)=
		C(4\theta)^\sigma
		\Psi_m(v;x_0,r/4)\\
		&\le
		C(4\theta)^\sigma(E+D).
	\end{split}\end{align}
	
	We now return from $v$ to $U$. By Minkowski's inequality,
	\[
	\Psi_m(U;x_0,\theta r)
	\le
	\Psi_m(v;x_0,\theta r)
	+
	\frac1{\theta r}
	\left[
	\left(
	\fint_{B_{\theta r}(x_0)}
	|U-v|^m\,\diff x
	\right)^{1/m}
	+
	\Tail_{m,sp}(U-v;x_0,\theta r)
	\right].
	\]
	Using \eqref{eq:mean-measure-comparison-small} with
	$\rho=\theta r$, we have
	\[
	\frac1{\theta r}
	\left[
	\left(
	\fint_{B_{\theta r}(x_0)}
	|U-v|^m\,\diff x
	\right)^{1/m}
	+
	\Tail_{m,sp}(U-v;x_0,\theta r)
	\right]
	\le
	C\theta^{-1-n/m}D.
	\]
	Together with \eqref{eq:replacement-decay-small-scale}, this yields
	\[
	\Psi_m(U;x_0,\theta r)
	\le
	C(4\theta)^\sigma E
	+
	C\left(
	(4\theta)^\sigma+\theta^{-1-n/m}
	\right)D.
	\]
	Choose $\theta=\theta(n,p,s)\in(0,1/512)$ small enough that
	\[
	C(4\theta)^\sigma\le\frac12.
	\]
	Since $\theta$ is now fixed, the remaining coefficient of $D$ is
	absorbed into a constant depending only on $n,p,s$. We obtain
	\begin{equation}\label{eq:mean-one-step-energy}
		\Psi_m(U;x_0,\theta r)
		\le
		\frac12\Psi_m(U;x_0,r)
		+
		CD_\nu(x_0,r).
	\end{equation}
	
	We now pass to the SOLA limit. Let $u_j,\mu_j$ be the approximations
    in Definition~\ref{def:sola}. We have
	\[
	u_j\in W^{s,p}(\R^n)
	\]
	by the definition of SOLA. Moreover, $u_j$ solves
	\[
	(-\Delta_p)^su_j=\mu_j
	\qquad\text{weakly in }\Omega,
	\]
	and therefore also in $B_{2r}(x_0)$. Hence
	\eqref{eq:mean-one-step-energy} applies to $U=u_j$ and
	$\nu=\mu_j$:
	\begin{equation}\label{eq:one-step-approximation}
		\Psi_m(u_j;x_0,\theta r)
		\le
		\frac12\Psi_m(u_j;x_0,r)
		+
		CD_{\mu_j}(x_0,r).
	\end{equation}
	
	By Lemma~\ref{lem:sola-fixed-scale},
	\[
	\Psi_m(u_j;x_0,r)
	\longrightarrow
	\Psi_m(u;x_0,r),
	\qquad
	\Psi_m(u_j;x_0,\theta r)
	\longrightarrow
	\Psi_m(u;x_0,\theta r).
	\]
	On the other hand, the measure approximation condition gives
	\[
	\limsup_{j\to\infty}
	|\mu_j|(B_r(x_0))
	\le
	|\mu|(\overline{B_r(x_0)}).
	\]
	Therefore, 
	\begin{align*}
		\limsup_{j\to\infty}D_{\mu_j}(x_0,r)
		\le
		\left(
		\frac{|\mu|(\overline{B_r(x_0)})}
		{r^{n-\beta}}
		\right)^{1/m}\le
		\left(
		\frac{|\mu|(B_{2r}(x_0))}
		{r^{n-\beta}}
		\right)^{1/m}=
		2^{(n-\beta)/m}
		D_\mu(x_0,2r).
	\end{align*}
	Taking the upper limit in \eqref{eq:one-step-approximation} gives
	\[
	\Psi_m(u;x_0,\theta r)
	\le
	\frac12\Psi_m(u;x_0,r)
	+
	CD_\mu(x_0,2r).
	\]
	This proves \eqref{eq:mean-one-step}.
\end{proof}

We next construct the averaged affine expansion, without assuming the existence of a weak gradient.
\begin{lemma}\label{lem:mean-affine-expansion}
	Let $u$ be a SOLA, let $B_{2R}(x_0)\Subset\Omega$, and assume that
	$\Wpot_{\gamma,p}^{|\mu|}(x_0,2R)<\infty$.
	With $\theta$ as in Lemma~\ref{lem:mean-one-step}, take
	\[
	r_k:=\frac R2\theta^k,\qquad
	E_k:=\Psi_m(u;x_0,r_k),\qquad D_k:=D_\mu(x_0,2r_k),
	\]
	and let $\ell_k(x)=b_k+a_k\cdot(x-x_0)$ be a minimizing affine
	function for $E_k$.
	Then there is a unique pair $(b_*,a_*)\in\R\times\R^n$ such that
	\begin{equation}\label{eq:mean-first-order-expansion}
		\lim_{r\to0}\frac1r
		\left(\fint_{B_r(x_0)}
		|u-b_*-a_*\cdot(x-x_0)|^m\,\diff x\right)^{1/m}=0.
	\end{equation}
	Moreover,
	\begin{equation}\label{eq:mean-affine-limits}
		E_k+D_k+|a_k-a_*|+\frac{|b_k-b_*|}{r_k}\longrightarrow0,
	\end{equation}
	and for $C=C(n,p,s)$, we have
	\begin{equation}\label{eq:mean-all-scale-bound}
		|a_*|+\sup_{k\ge0}|a_k|+\sum_{k=0}^\infty E_k
		\le C\left[\mathcal A_m(u;x_0,2R)
		+\Wpot_{\gamma,p}^{|\mu|}(x_0,2R)\right].
	\end{equation}
\end{lemma}

\begin{proof}
	By \eqref{eq:sola-tail-class},   we have $u\in L^m(\R^n)\cap L_{sp}^m(\R^n)$.
	Hence all $E_k$ are finite, and for  each $E_k$,
	Lemma~\ref{lem:affine-minimizer} provides a minimizing affine
	function
	\[
	\ell_k(x)=b_k+a_k\cdot(x-x_0).	\]

	Take $c:=(u)_{B_{2R}(x_0)}$. Since $c$ is an admissible affine competitor for
	$\Psi_m(u;x_0,r_0)$, where $r_0=R/2$, and
	\eqref{eq:mean-radius-restriction} applies between the radii
	$r_0$ and $2R$, we have
	\begin{equation}\label{eq:mean-initial-competitor}
		E_0
		+
		\frac1{r_0}
		\left[
		\left(
		\fint_{B_{r_0}(x_0)}
		|u-c|^m\,\diff x
		\right)^{1/m}
		+
		\Tail_{m,sp}(u-c;x_0,r_0)
		\right]
		\le
		C\mathcal A_m(u;x_0,2R).
	\end{equation}
To control the initial slope, apply
\eqref{eq:mean-affine-norm} to $\ell_0-c$. Since
\[
\ell_0-c=(\ell_0-u)+(u-c),
\]
the minimizing property of $\ell_0$ and
\eqref{eq:mean-initial-competitor} give
\[
\begin{aligned}
	r_0|a_0|
	&\le
	C\left(
	\fint_{B_{r_0}(x_0)}
	|\ell_0-c|^m\,\diff x
	\right)^{1/m}\\
	&\le
	Cr_0E_0+
	C\left(
	\fint_{B_{r_0}(x_0)}
	|u-c|^m\,\diff x
	\right)^{1/m}\le
	Cr_0\mathcal A_m(u;x_0,2R).
\end{aligned}
\]
Hence
\begin{equation}\label{eq:mean-initial-slope}
	|a_0|+E_0
	\le C\mathcal A_m(u;x_0,2R).
\end{equation}
	Since  $
	r_{k+1}=\theta r_k$
	and $B_{2r_k}(x_0)\Subset\Omega$, Lemma~\ref{lem:mean-one-step}
	gives
	\begin{equation}\label{eq:mean-scale-recurrence}
		E_{k+1}
		\le
		\frac12E_k+CD_k,
		\qquad k\ge0.
	\end{equation}
	Summing \eqref{eq:mean-scale-recurrence} for
	$k=0,\ldots,N-1$, we obtain
	\[
	\sum_{k=1}^{N}E_k
	\le
	\frac12\sum_{k=0}^{N-1}E_k
	+
	C\sum_{k=0}^{N-1}D_k.
	\]
	Absorbing the finite sum of excesses gives
	\begin{equation}\label{eq:mean-finite-excess-sum}
		\sum_{k=0}^{N}E_k
		\le
		2E_0
		+
		C\sum_{k=0}^{N-1}D_k.
	\end{equation}
	
	Because $
	2r_k=R\theta^k$,
	Lemma~\ref{lem:dyadic-potential} yields
	\begin{equation}\label{eq:mean-data-sum}
		\sum_{k=0}^{\infty}D_k
		\le
		C\Wpot_{\gamma,p}^{|\mu|}(x_0,2R)
		<\infty.
	\end{equation}
	Letting $N\to\infty$ in
	\eqref{eq:mean-finite-excess-sum} and using
	\eqref{eq:mean-initial-competitor}, we obtain
	\begin{equation}\label{eq:mean-excess-sum}
		\sum_{k=0}^{\infty}E_k
		\le
		C\left[
		\mathcal A_m(u;x_0,2R)
		+
		\Wpot_{\gamma,p}^{|\mu|}(x_0,2R)
		\right].
	\end{equation}
	In particular,
	\[
	E_k\longrightarrow0,
	\qquad
	D_k\longrightarrow0.
	\]
	
	We next compare the minimizing affine functions at two successive scales.	Applying \eqref{eq:mean-affine-norm} to
	$\ell_{k+1}-\ell_k$ on $B_{r_{k+1}}(x_0)$ gives
	\[
	\begin{aligned}
		|b_{k+1}-b_k|
		+
		r_{k+1}|a_{k+1}-a_k|\le
		C
		\left(
		\fint_{B_{r_{k+1}}(x_0)}
		|\ell_{k+1}-\ell_k|^m\,\diff x
		\right)^{1/m}.
	\end{aligned}
	\]
	Since
	\[
	\ell_{k+1}-\ell_k
	=
	(\ell_{k+1}-u)+(u-\ell_k),
	\]
	Minkowski's inequality and the minimizing properties give
	\[
	\begin{aligned}
		\left(
		\fint_{B_{r_{k+1}}(x_0)}
		|\ell_{k+1}-\ell_k|^m\,\diff x
		\right)^{1/m}&\le
		r_{k+1}E_{k+1}
		+
		\left(
		\frac{r_k}{r_{k+1}}
		\right)^{n/m}
		r_kE_k\\
		&=
		r_{k+1}E_{k+1}
		+
		\theta^{-1-n/m}r_{k+1}E_k.
	\end{aligned}
	\]
	Thus
\begin{equation}\label{eq:mean-affine-increments}
		\begin{aligned}
			|a_{k+1}-a_k|
			\le C_\theta(E_k+E_{k+1})\quad \text{and} \quad 
			|b_{k+1}-b_k|\le
			C_\theta r_{k+1}(E_k+E_{k+1}).
		\end{aligned}
	\end{equation}
	By \eqref{eq:mean-affine-increments} and
	\eqref{eq:mean-excess-sum}, the sequence $\{a_k\}$ is Cauchy.
	Consequently, there exists $a_*\in\R^n$ such that
	\[
	|a_*-a_k|
	\le C_\theta\sum_{j=k}^{\infty}E_j
	\longrightarrow0.
	\]
	The second estimate in
	\eqref{eq:mean-affine-increments} similarly shows that
	$b_k\to b_*$ for some $b_*\in\R$. Since
	$r_{j+1}/r_k=\theta^{j+1-k}$,
	\[
	\begin{aligned}
		\frac{|b_*-b_k|}{r_k}
		\le
		C_\theta
		\sum_{j=k}^{\infty}
		\frac{r_{j+1}}{r_k}
		(E_j+E_{j+1})\le
		C_\theta
		\sum_{j=k}^{\infty}
		\theta^{j-k}(E_j+E_{j+1})
		\longrightarrow0.
	\end{aligned}
	\]
	The last limit follows from $E_j\to0$ and the summability of the
	geometric weights. We have therefore proved
	\[
	E_k+D_k+|a_k-a_*|
	+
	\frac{|b_k-b_*|}{r_k}
	\longrightarrow0.
	\]
	
Take
	\[
	\ell_*(x):=b_*+a_*\cdot(x-x_0).
	\]
	Using
	\[
	u-\ell_*=(u-\ell_k)+(\ell_k-\ell_*),
	\]
	we obtain
	\[
	\begin{aligned}
		\frac1{r_k}
		\left(
		\fint_{B_{r_k}(x_0)}
		|u-\ell_*|^m\,\diff x
		\right)^{1/m}\le
		E_k
		+
		C|a_k-a_*|
		+
		C\frac{|b_k-b_*|}{r_k}
		\longrightarrow0.
	\end{aligned}
	\]
	If $r_{k+1}<r\le r_k$, then
	\[
	\begin{aligned}
		\frac1r
		\left(
		\fint_{B_r(x_0)}
		|u-\ell_*|^m\,\diff x
		\right)^{1/m}&\le
		\left(\frac{r_k}{r}\right)^{1+n/m}
		\frac1{r_k}
		\left(
		\fint_{B_{r_k}(x_0)}
		|u-\ell_*|^m\,\diff x
		\right)^{1/m}\\
		&\le
		\theta^{-1-n/m}
		\frac1{r_k}
		\left(
		\fint_{B_{r_k}(x_0)}
		|u-\ell_*|^m\,\diff x
		\right)^{1/m}.
	\end{aligned}
	\]
	This proves \eqref{eq:mean-first-order-expansion} for all radii
    tending to zero.
	
	If another pair $(\widetilde b,\widetilde a)$ satisfies
	\eqref{eq:mean-first-order-expansion}, then the affine function
	\[
	P(x)
	=
	(b_*-\widetilde b)
	+
	(a_*-\widetilde a)\cdot(x-x_0)
	\]
	satisfies
	\[
	\frac1r
	\left(
	\fint_{B_r(x_0)}|P|^m\,\diff x
	\right)^{1/m}
	\longrightarrow0.
	\]
	By \eqref{eq:mean-affine-norm},
	\[
	\frac{|b_*-\widetilde b|}{r}
	+
	|a_*-\widetilde a|
	\le
	\frac C r
	\left(
	\fint_{B_r(x_0)}|P|^m\,\diff x
	\right)^{1/m},
	\]
	and letting $r\to0$ gives  $
	b_*=\widetilde b$ and  $
	a_*=\widetilde a$.	
	Finally, summing the first estimate in
	\eqref{eq:mean-affine-increments} gives
	\[
	|a_k|
	\le
	|a_0|
	+
	C_\theta\sum_{j=0}^{\infty}E_j.
	\]
	Hence \eqref{eq:mean-initial-slope} and
	\eqref{eq:mean-excess-sum} imply
	\[
	|a_*|+\sup_{k\ge0}|a_k|
	+\sum_{k=0}^{\infty}E_k
	\le
	C\left[
	\mathcal A_m(u;x_0,2R)
	+
	\Wpot_{\gamma,p}^{|\mu|}(x_0,2R)
	\right],
	\]
	which proves \eqref{eq:mean-all-scale-bound}.
\end{proof}
\subsection{Existence of the weak gradient}
\begin{lemma}[Local integrability of the Wolff potential]
\label{lem:wolff-local-Lm}
Let $K\Subset\Omega$ be compact and let $T>0$ satisfy
$K_T:=\{y:\dist(y,K)\le T\}\Subset\Omega$. Then
\begin{equation}\label{eq:wolff-local-Lm}
\|\Wpot_{\gamma,p}^{|\mu|}(\,\cdot\,,T)\|_{L^m(K)}
\le CT^{\beta/m}|\mu|(K_T)^{1/m},
\end{equation}
where $C=C(n,p,s)$. In particular, this potential is finite almost
everywhere on $K$.
\end{lemma}

\begin{proof}
	For $x\in K$ and $0<t\le T$, we have
	\[
	|\mu|(B_t(x))
	=
	\int_{K_T}
	\mathbf 1_{\{|x-y|<t\}}\,\diff|\mu|(y).
	\]
Notice that  the map $(x,t)\mapsto D_\mu(x,t)$ is measurable.
	Tonelli's theorem gives
	\[
	\int_K|\mu|(B_t(x))\,\diff x
	=
	\int_{K_T}|K\cap B_t(y)|\,\diff|\mu|(y)
	\le |B_1|t^n|\mu|(K_T).
	\]
	Consequently, we have
	\[
	\begin{aligned}
		\|D_\mu(\cdot,t)\|_{L^m(K)}
		=
		\left(
		t^{\beta-n}
		\int_K|\mu|(B_t(x))\,\diff x
		\right)^{1/m}\le
		|B_1|^{1/m}t^{\beta/m}|\mu|(K_T)^{1/m}.
	\end{aligned}
	\]
	Since $\beta>0$,
	\[
	\int_0^T t^{\beta/m}\frac{\diff t}{t}
	=
	\frac{m}{\beta}T^{\beta/m}<\infty.
	\]
	Thus $t\mapsto\|D_\mu(\cdot,t)\|_{L^m(K)}$ is integrable
	with respect to $\diff t/t$ on $(0,T)$.
	As $m=p-1>1$, Minkowski's integral inequality yields
	\[
	\begin{aligned}
		\|\Wpot_{\gamma,p}^{|\mu|}(\cdot,T)\|_{L^m(K)}
		=
		\left\|
		\int_0^T D_\mu(\cdot,t)\frac{\diff t}{t}
		\right\|_{L^m(K)}\le
		\int_0^T
		\|D_\mu(\cdot,t)\|_{L^m(K)}
		\frac{\diff t}{t}\le
		CT^{\beta/m}|\mu|(K_T)^{1/m}.
	\end{aligned}
	\]
	This proves \eqref{eq:wolff-local-Lm}. In particular,
	$\Wpot_{\gamma,p}^{|\mu|}(\cdot,T)$ is finite almost
	everywhere on $K$.
\end{proof}

\begin{lemma}[Weak differentiability of SOLA]
	\label{lem:sola-weak-gradient}
	Let $u$ be a SOLA of \eqref{eq:main}.
	Fix $\eta\in C_c^\infty(B_1)$ with $\eta\ge0$ and
	$\int_{\R^n}\eta\,\diff x=1$, and define
	\[
	\eta_\varepsilon(x):=\varepsilon^{-n}\eta(x/\varepsilon),
	\qquad
	u_\varepsilon:=u*\eta_\varepsilon.
	\]
	For every $x_0\in\Omega$ and $R>0$ such that
	$B_{2R}(x_0)\Subset\Omega$ and
	$\Wpot_{\gamma,p}^{|\mu|}(x_0,2R)<\infty$, we have
	\begin{equation}\label{eq:mollified-gradient-pointwise-limit}
		\nabla u_\varepsilon(x_0)\longrightarrow a_*(x_0)
		\qquad\text{as }\varepsilon\to 0,
	\end{equation}
	where $a_*(x_0)$ is the slope of the affine expansion
	given by Lemma~\ref{lem:mean-affine-expansion}.	Moreover, every SOLA $u$ of \eqref{eq:main} satisfies
	\[
	u\in W^{1,m}_{\mathrm{loc}}(\Omega).
	\]
\end{lemma}

\begin{proof}
	By Definition~\ref{def:sola}, $u\in L^m(\R^n)$. Fix a compact set $K\Subset\Omega$ and choose $R>0$ such that
	\[
	K_{2R}:=\{y\in\R^n:\dist(y,K)\le2R\}\Subset\Omega.
	\]
	For $x\in K$,  we define
	\[
	c_x:=(u)_{B_{2R}(x)},
	\qquad
	H_R(x):=
	\mathcal A_m(u;x,2R)
	+
	\Wpot_{\gamma,p}^{|\mu|}(x,2R).
	\]
	H\"older's and Minkowski's inequalities give
	\[
	\begin{aligned}
		|c_x|+	\left(
		\fint_{B_{2R}(x)}|u-c_x|^m\,\diff y
		\right)^{1/m}\le
		CR^{-n/m}\|u\|_{L^m(\R^n)}.
	\end{aligned}
	\]
	For the tail, we have
	\[
	\begin{aligned}
		\Tail_{m,sp}(u-c_x;x,2R)^m
		\le
		CR^{-n}\|u\|_{L^m(\R^n)}^m
		+C|c_x|^m\le
		CR^{-n}\|u\|_{L^m(\R^n)}^m.
	\end{aligned}
	\]
Combining these bounds gives
\begin{equation}\label{eq:background-fixed-radius-bound}
		\sup_{x\in K}\mathcal A_m(u;x,2R)
		\le
		CR^{-1-n/m}\|u\|_{L^m(\R^n)}.
	\end{equation}
Moreover,  Lemma~\ref{lem:wolff-local-Lm} gives
	\begin{equation}\label{eq:gradient-dominating-function}
		H_R\in L^m(K).
	\end{equation}
	In particular, $H_R(x)<\infty$ for almost every $x\in K$.
	
	Fix $x\in K$ such that $H_R(x)<\infty$.
	Lemma~\ref{lem:mean-affine-expansion}, applied with centre $x$,
	provides the radii and minimizing affine functions
	\[
	r_k:=\frac R2\theta^k,
	\qquad
	\ell_k^x(y):=b_k(x)+a_k(x)\cdot(y-x),
	\qquad
	E_k(x):=\Psi_m(u;x,r_k).
	\]
	They satisfy
	\[
	E_k(x)\longrightarrow0,
	\qquad
	a_k(x)\longrightarrow a_*(x),
	\]
	and \eqref{eq:mean-all-scale-bound} gives
	\[
	\sup_{k\ge0}|a_k(x)|
	+
	\sum_{k=0}^{\infty}E_k(x)
	\le CH_R(x).
	\]
	
	Recall that $u_\varepsilon=u*\eta_\varepsilon$.
	Since $u$ is locally integrable and $\eta_\varepsilon$ is
	smooth with compact support, differentiation under the
	integral gives
	\[
	\nabla u_\varepsilon(z)
	=
	\int_{\R^n}u(y)\nabla\eta_\varepsilon(z-y)\,\diff y.
	\]
	For fixed $x$ and $k$, a change of variables gives
	\[
	\begin{aligned}
		(\ell_k^x*\eta_\varepsilon)(z)
		=
		\int_{\R^n}\ell_k^x(z-h)\eta_\varepsilon(h)\,\diff h=
		\ell_k^x(z)
		-
		a_k(x)\cdot
		\int_{\R^n}h\eta_\varepsilon(h)\,\diff h,
	\end{aligned}
	\]
	where we used $\int_{\R^n}\eta_\varepsilon=1$.
	The last term is independent of $z$, and hence
	\[
	\nabla_z(\ell_k^x*\eta_\varepsilon)(z)=a_k(x).
	\]
	Thus, with the centre $x$ fixed,
	\[
	\begin{aligned}
		\nabla u_\varepsilon(x)-a_k(x)
		=
		\int_{\R^n}
		[u(y)-\ell_k^x(y)]
		\nabla\eta_\varepsilon(x-y)\,\diff y=
		\int_{B_\varepsilon(x)}
		[u(y)-\ell_k^x(y)]
		\nabla\eta_\varepsilon(x-y)\,\diff y.
	\end{aligned}
	\]
	
	For $0<\varepsilon<R/2$, choose $k$ such that $
	r_{k+1}<\varepsilon\le r_k$.
	Then $B_\varepsilon(x)\subset B_{r_k}(x)$ and
	$r_k/\varepsilon<\theta^{-1}$.
	Since $
	|\nabla\eta_\varepsilon|
	\le C_\eta\varepsilon^{-n-1}$,
	H\"older's inequality gives
	\[
	\begin{aligned}
		|\nabla u_\varepsilon(x)-a_k(x)|\le
		C_\eta
		\left(\frac{r_k}{\varepsilon}\right)^{1+n/m}
		E_k(x)\le C_{\theta,\eta}E_k(x).
	\end{aligned}
	\]
	The constants in these estimates may depend on the fixed
	function $\eta$, but not on $x$, $k$, or $\varepsilon$.
	
	As $\varepsilon\to0$, the corresponding index $k$ tends to
	infinity. Hence
	\[
	\begin{aligned}
		|\nabla u_\varepsilon(x)-a_*(x)|
		\le
		|\nabla u_\varepsilon(x)-a_k(x)|
		+
		|a_k(x)-a_*(x)|
		\longrightarrow0.
	\end{aligned}
	\]
	This argument applies at every point satisfying the hypotheses
	of Lemma~\ref{lem:mean-affine-expansion}, and proves
	\eqref{eq:mollified-gradient-pointwise-limit}.

	Combining the above estimate with \eqref{eq:mean-all-scale-bound} gives, for almost every $x\in K$ and $0<\varepsilon<R/2$,
	\begin{equation}\label{eq:mollified-gradient-domination}
	|\nabla u_\varepsilon(x)|
	\le |a_k(x)|+|\nabla u_\varepsilon(x)-a_k(x)|
	\le C_{\theta,\eta}H_R(x).
	\end{equation}
	We define
	\[
	K_0:=\{x\in K:H_R(x)<\infty\}.
	\]
	By \eqref{eq:gradient-dominating-function}, we have  $
	|K\setminus K_0|=0$. For every $x\in K_0$, 	\eqref{eq:mollified-gradient-pointwise-limit} shows that
	\[
	\nabla u_\varepsilon(x)\longrightarrow a_*(x)
	\qquad\text{as }\varepsilon\to0.
	\]

	For every $x\in K_0$, \eqref{eq:mollified-gradient-pointwise-limit}
	gives
	\[
	\nabla u_{2^{-j}}(x)\longrightarrow a_*(x)
	\qquad\text{as }j\to\infty.
	\]
	Since $|K\setminus K_0|=0$ and the compact set
	$K\Subset\Omega$ was arbitrary, the sequence
	$\{\nabla u_{2^{-j}}\}$ converges almost everywhere in $\Omega$.
	Define
	\[
	G(x):=
	\begin{cases}
		\displaystyle\lim_{j\to\infty}\nabla u_{2^{-j}}(x),
		&\text{if this limit exists in }\R^n,\\
		0,&\text{otherwise}.
	\end{cases}
	\]
	Each $u_{2^{-j}}$ is smooth, so its gradient is measurable.
	Thus $G$ is measurable as an almost everywhere limit of
	measurable functions.
	
	For every $x\in K_0$, we have $G(x)=a_*(x)$.
	Returning to \eqref{eq:mollified-gradient-pointwise-limit},
	we obtain
	\[
	\nabla u_\varepsilon(x)\longrightarrow G(x)
	\qquad\text{as }\varepsilon\to0
	\quad\text{for every }x\in K_0.
	\]
	Moreover, \eqref{eq:mollified-gradient-domination} gives
	\[
	|\nabla u_\varepsilon(x)|
	\le C_{\theta,\eta}H_R(x)
	\]
	for $x\in K_0$ and all sufficiently small $\varepsilon$.
	Passing to the pointwise limit yields
	\[
	|G(x)|\le C_{\theta,\eta}H_R(x)
	\qquad\text{for }x\in K_0.
	\]
	Since $|K\setminus K_0|=0$ and
	$H_R^m\in L^1(K)$ by
	\eqref{eq:gradient-dominating-function}, it follows that
	$G\in L^m(K;\R^n)$.
	
	For almost every $x\in K$, we also have
	\[
	\begin{aligned}
		|\nabla u_\varepsilon(x)-G(x)|^m
		\le
		\bigl(|\nabla u_\varepsilon(x)|+|G(x)|\bigr)^m\le
		(2C_{\theta,\eta})^m H_R(x)^m.
	\end{aligned}
	\]
	The right-hand side is integrable on $K$ and independent
	of $\varepsilon$. Hence the dominated convergence theorem gives
	\[
	\int_K|\nabla u_\varepsilon-G|^m\,\diff x
	\longrightarrow0.
	\]
	Therefore,
	\[
	\nabla u_\varepsilon\longrightarrow G
	\qquad\text{strongly in }L^m(K;\R^n).
	\]
	Since $u\in L^m_{\mathrm{loc}}(\Omega)$, the approximation
	property of mollification also gives
	\[
	u_\varepsilon\longrightarrow u
	\qquad\text{strongly in }L^m(K).
	\]
	
	We now verify that $G$ is the weak gradient of $u$.
	Write $G=(G_1,\ldots,G_n)$ and set $m'=m/(m-1)$.
	Let $\varphi\in C_c^\infty(\Omega)$, and choose
	$K\Subset\Omega$ whose interior contains
	$\operatorname{supp}\varphi$.
	For every $i=1,\ldots,n$, integration by parts gives
	\[
	\int_\Omega \partial_i u_\varepsilon\,\varphi\,\diff x
	+
	\int_\Omega u_\varepsilon\,\partial_i\varphi\,\diff x
	=0.
	\]
	Consequently, by H\"older's inequality,
	\[
	\begin{aligned}
		&\left|
		\int_\Omega G_i\varphi\,\diff x
		+
		\int_\Omega u\,\partial_i\varphi\,\diff x
		\right|\\
		&\quad=
		\left|
		\int_\Omega
		(G_i-\partial_i u_\varepsilon)\varphi\,\diff x
		+
		\int_\Omega
		(u-u_\varepsilon)\partial_i\varphi\,\diff x
		\right|\\
		&\quad\le
		\|\nabla u_\varepsilon-G\|_{L^m(K)}
		\|\varphi\|_{L^{m'}(K)}
		+
		\|u_\varepsilon-u\|_{L^m(K)}
		\|\partial_i\varphi\|_{L^{m'}(K)}
		\longrightarrow0.
	\end{aligned}
	\]
	The expression on the left does not depend on $\varepsilon$. Thus, we obtain
	\[
	\int_\Omega G_i\varphi\,\diff x
	=
	-\int_\Omega u\,\partial_i\varphi\,\diff x
	\qquad
	\text{for every }\varphi\in C_c^\infty(\Omega).
	\]
	This proves that $G_i$ is the weak derivative $\partial_i u$.
	Since $u,G\in L^m_{\mathrm{loc}}(\Omega)$, we conclude that
	\[
	u\in W^{1,m}_{\mathrm{loc}}(\Omega),
	\qquad
	G=\nabla u
	\quad\text{almost everywhere in }\Omega.
	\]
\end{proof}
\subsection{The estimate at Lebesgue points}
\begin{proof}[Proof of Theorem~\ref{thm:mean-gradient}]
Weak differentiability follows from
Lemma~\ref{lem:sola-weak-gradient}. Fix a Lebesgue point $x_0$ of
$\nabla u$ and $B_{2R}(x_0)\Subset\Omega$.
If $\Wpot_{\gamma,p}^{|\mu|}(x_0,2R)=\infty$, the estimate is immediate.
Otherwise, Lemma~\ref{lem:mean-affine-expansion} provides the slope $a_*$
and the bound \eqref{eq:mean-all-scale-bound}.
For $0<\varepsilon<R$, the convolution identity for weak derivatives and
the Lebesgue point property give
\[
\begin{aligned}
|\nabla u_\varepsilon(x_0)-\nabla u(x_0)|
&\le \int_{B_\varepsilon(x_0)}\eta_\varepsilon(x_0-y)
|\nabla u(y)-\nabla u(x_0)|\,\diff y\\
&\le C\fint_{B_\varepsilon(x_0)}
|\nabla u(y)-\nabla u(x_0)|\,\diff y\longrightarrow0.
\end{aligned}
\]
Together with \eqref{eq:mollified-gradient-pointwise-limit}, this shows
that $a_*=\nabla u(x_0)$. The bound
\eqref{eq:mean-all-scale-bound} is therefore
\eqref{eq:mean-gradient-bound}. By the Lebesgue differentiation theorem, this estimate holds for
almost every $x_0\in\Omega$.
\end{proof}

\section{Classical differentiability under \texorpdfstring{$sp>n$}{sp>n}}
\label{sec:classical-gradient}
Throughout this section we assume \eqref{eq:gradient-parameter-range}
and $sp>n$. We first establish comparison estimates for local weak
solutions and then prove the local energy regularity needed to apply
them to SOLA. The affine functions from
Lemma~\ref{lem:mean-affine-expansion} will then yield the classical
derivative. The argument requires no global energy assumption on the SOLA.

\subsection{Measure testing and comparison}
\begin{lemma}[Fractional Morrey estimate]\label{lem:fractional-morrey}
Let $\tau=s-n/p>0$. Every $w\in W^{s,p}(\R^n)$ has a unique bounded
continuous representative satisfying
\begin{equation}\label{eq:fractional-morrey}
[w]_{C^{0,\tau}(\R^n)}\le C[w]_{W^{s,p}(\R^n)}.
\end{equation}
For $w\in W_0^{s,p}(B_r(x_0))$, this representative vanishes on
$\R^n\setminus B_r(x_0)$ and
\begin{equation}\label{eq:fractional-morrey-zero}
\|w\|_{L^\infty(B_r(x_0))}
\le Cr^{s-n/p}[w]_{W^{s,p}(\R^n)}.
\end{equation}
Consequently, every finite signed measure $\nu$ on $B_r(x_0)$ belongs
to $(W_0^{s,p}(B_r(x_0)))'$, with
\begin{equation}\label{eq:measure-dual-local}
\left|\int_{B_r(x_0)}\varphi\,\diff\nu\right|
\le Cr^{s-n/p}|\nu|(B_r(x_0))
[\varphi]_{W^{s,p}(\R^n)}.
\end{equation}
The measure integral is taken with the continuous representative of
$\varphi$. Here and below the constants in the estimates depend only
on $n,p,s$.
\end{lemma}
\begin{proof}
By~\cite[Proposition~2.2]{Nowak2023}, for every $k\in\mathbb N$ there is a representative $w_k\in C^{0,\tau}(B_k(0))$ such that
\[
[w_k]_{C^{0,\tau}(B_k(0))}
\le
C[w]_{W^{s,p}(B_k(0))}
\le
C[w]_{W^{s,p}(\R^n)}.
\]
If $j<k$, then $w_j=w_k$ on $B_j(0)$, since both are continuous representatives of the same function. Hence $\{w_k\}$ defines a continuous representative $w$ on $\R^n$.

For arbitrary $x,y\in\R^n$, choosing $k$ with $x,y\in B_k(0)$ gives
\[
|w(x)-w(y)|
\le
C[w]_{W^{s,p}(\R^n)}|x-y|^\tau,
\]
which is \eqref{eq:fractional-morrey}. Moreover,
\[
|w(x)|
\le
\fint_{B_1(x)}|w(y)|\,\dd y
+
\fint_{B_1(x)}|w(x)-w(y)|\,\dd y
\le
C\|w\|_{L^p(\R^n)}
+
C[w]_{C^{0,\tau}(\R^n)},
\]
so the representative is bounded. Uniqueness follows from continuity.

Now let $w\in W^{s,p}_0(B_r(x_0))$. Since $w=0$ a.e. outside $B_r(x_0)$ and the representative is continuous, $w=0$ on $\R^n\setminus B_r(x_0)$. For $x\in B_r(x_0)$, choose $y_x\in\partial B_r(x_0)$ with
\[
|x-y_x|
=
\dist(x,\partial B_r(x_0))
\le r.
\]
Then $w(y_x)=0$, and \eqref{eq:fractional-morrey} gives
\[
|w(x)|
\le
[w]_{C^{0,\tau}(\R^n)}|x-y_x|^\tau
\le
Cr^{s-n/p}[w]_{W^{s,p}(\R^n)}.
\]
Taking the supremum proves \eqref{eq:fractional-morrey-zero}.
Finally,
\[
\left|\int_{B_r(x_0)}\varphi\,\diff\nu\right|
\le |\nu|(B_r(x_0))\|\varphi\|_{L^\infty(B_r(x_0))},
\]
and \eqref{eq:fractional-morrey-zero} gives
\eqref{eq:measure-dual-local}.
\end{proof}

We first justify the weak formulation under local energy regularity.  Let $B=B_r(x_0)$ and
$U\in W^{s,p}(B_{2r}(x_0))\cap L_{sp}^m(\R^n)$.
For $\varphi\in W_0^{s,p}(B)$, write
\[
\mathcal B_U(\varphi):=
\iint_{\R^n\times\R^n}
\frac{J_p(\delta U(x,y))\delta\varphi(x,y)}{|x-y|^{n+sp}}
\,\diff x\diff y.
\]
This integral is absolutely convergent. Indeed, local Morrey embedding
gives $U\in L^\infty(B)$, and splitting the  integration domain into
$B_{2r}(x_0)\times B_{2r}(x_0)$ and the complement yields
\begin{align}\label{eq:local-weak-pairing-continuity}
	\begin{split}
&\iint_{\R^n\times\R^n}
\frac{|\delta U|^m|\delta\varphi|}{|x-y|^{n+sp}}
\,\diff x\diff y\\
&\quad\le
[U]_{W^{s,p}(B_{2r}(x_0))}^{m}
[\varphi]_{W^{s,p}(\R^n)}+
Cr^{-sp}\bigl[\|U\|_{L^\infty(B)}^m
+\Tail_{m,sp}(U;x_0,2r)^m\bigr]
\|\varphi\|_{L^1(B)}.
\end{split}
\end{align}
Together with \eqref{eq:measure-dual-local}, this allows a weak
identity valid for smooth test functions to extend by density
to every $\varphi\in W_0^{s,p}(B)$.
The same argument remains valid if $B_{2r}(x_0)$ is replaced by
a larger concentric ball compactly containing $B$.

We shall also use the following estimate obtained in the proof
of Lemma~\ref{lem:mean-homogeneous-decay}. If $v$ satisfies the
assumptions of that lemma in $B_{2a}(x_0)$, then for every
affine function $\ell$,
\begin{equation}\label{eq:homogeneous-fixed-affine-supremum}
	\|v-\ell\|_{L^\infty(B_{a/2}(x_0))}
	\le
	C\left[
	\left(
	\fint_{B_a(x_0)}
	|v-\ell|^m\,\diff x
	\right)^{1/m}
	+
	\Tail_{m,sp}(v-\ell;x_0,a)
	\right].
\end{equation}
This follows from \eqref{eq:mean-local-supremum} after changing
the centre and radius. Since that estimate holds for every
affine function before the excess is minimized, the constant
is independent of the slope of $\ell$.

\begin{lemma}[Local comparison and affine approximation]
\label{lem:uniform-replacement-comparison}
Let
\[
U\in W^{s,p}(B_{2r}(x_0))\cap L_{sp}^m(\R^n)
\]
be a local weak solution of $(-\Delta_p)^sU=\nu$ in $B_{2r}(x_0)$,
where $\nu$ is a finite signed measure. There is a unique function
\[
v\in U+W_0^{s,p}(B_r(x_0))
\]
which is fractional $p$-harmonic in $B_r(x_0)$. It satisfies
\begin{equation}\label{eq:measure-uniform-comparison}
\frac1r\|U-v\|_{L^\infty(B_r(x_0))}\le CD_\nu(x_0,r)
\end{equation}
and, for $0<\rho\le r/2$,
\begin{equation}\label{eq:measure-comparison-tail}
\frac1\rho\Tail_{m,sp}(U-v;x_0,\rho)
\le C\frac r\rho D_\nu(x_0,r).
\end{equation}
For every affine function $\ell$, one also has
\begin{align}\label{eq:energy-mean-to-uniform}
	\begin{split}
\frac1r\|U-\ell\|_{L^\infty(B_{r/8}(x_0))}
&\le\frac Cr\left[
\left(\fint_{B_r(x_0)}|U-\ell|^m\,\diff x\right)^{1/m}
+\Tail_{m,sp}(U-\ell;x_0,r)\right]\\
&\quad+CD_\nu(x_0,r).
\end{split}
\end{align}
\end{lemma}

\begin{proof}
We  denote  $B=B_r(x_0)$ and $V=W_0^{s,p}(B)$.
For $w\in V$, define
\begin{equation}\label{eq:local-comparison-functional}
	\mathscr F_U(w):=
	\frac1p\iint_{\mathcal Q(B)}
	\frac{|\delta U+\delta w|^p-|\delta U|^p}{|x-y|^{n+sp}}
	\,\diff x\diff y.
\end{equation}
The inequality
\[
\bigl||a+b|^p-|a|^p\bigr|
\le C\bigl(|a|^{p-1}|b|+|b|^p\bigr)
\]
and \eqref{eq:local-weak-pairing-continuity} show that
$\mathscr F_U(w)$ is finite. They also give
\[
|\mathscr F_U(w)|\le C_U[w]_{W^{s,p}(\R^n)}
+C[w]_{W^{s,p}(\R^n)}^p,
\]
where $C_U<\infty$ is used only for this construction.	
Since $p>2$, the uniform convexity of $t\mapsto |t|^p$ gives
\begin{equation}\label{eq:p-uniform-convexity}
	|a+b|^p-|a|^p-pJ_p(a)b\ge c_p|b|^p
	\qquad (a,b\in\R).
\end{equation}
If $b=0$, the inequality is immediate. Suppose that $b\ne0$. Then
\[
\begin{aligned}
	|a+b|^p-|a|^p-pJ_p(a)b
	&=
	p(p-1)b^2
	\int_0^1(1-t)|a+tb|^{p-2}\,\diff t  \\
	&=
	p(p-1)|b|^p
	\int_0^1(1-t)
	\left|\frac ab+t\right|^{p-2}\,\diff t
	\ge c_p|b|^p ,
\end{aligned}
\]
where the last constant is positive since
\[
\inf_{q\in\R}
\int_0^1(1-t)|q+t|^{p-2}\,\diff t>0.
\]
Consequently,
	\[
	\begin{aligned}
		\mathscr F_U(w)-\mathcal B_U(w)
		&=
		\iint_{\mathcal Q(B)}
		\frac{
			|\delta U+\delta w|^p-|\delta U|^p
			-pJ_p(\delta U)\delta w
		}
		{p|x-y|^{n+sp}}\,\diff x\diff y\\
		&\ge c[w]_{W^{s,p}(\R^n)}^p,
	\end{aligned}
	\]
	and hence
	\[
	\mathscr F_U(w)
	\ge c[w]_{W^{s,p}(\R^n)}^p
	-C_U[w]_{W^{s,p}(\R^n)}.
	\]
	
	Let $(w_k)\subset V$ be a minimizing sequence.
	It is bounded in $V$. By weak compactness and the compact
	embedding into $L^p(B)$, after passing to a subsequence,
	\[
	w_k\rightharpoonup w_0\quad\text{in }V,
	\qquad
	w_k\to w_0\quad\text{a.e. in }\R^n.
	\]
	The integrand in $\mathscr F_U(w)-\mathcal B_U(w)$
	is nonnegative. Fatou's lemma therefore gives
	\[
	\mathscr F_U(w_0)-\mathcal B_U(w_0)
	\le\liminf_{k\to\infty}
	\bigl(
	\mathscr F_U(w_k)-\mathcal B_U(w_k)
	\bigr).
	\]
	Since $\mathcal B_U$ is a bounded linear functional on $V$,
	$\mathcal B_U(w_k)\to\mathcal B_U(w_0)$. Thus
	\[
	\mathscr F_U(w_0)
	\le\liminf_{k\to\infty}\mathscr F_U(w_k)
	=\inf_{w\in V}\mathscr F_U(w).
	\]
	
	Set $v=U+w_0$. For $\varphi\in V$ and $|t|\le1$,
	\[
	\bigl|
	J_p(\delta U+\delta w_0+t\delta\varphi)
	\delta\varphi
	\bigr|
	\le C\bigl(
	|\delta U|^m+|\delta w_0|^m+|\delta\varphi|^m
	\bigr)|\delta\varphi|.
	\]
	The right-hand side is integrable against
	$|x-y|^{-n-sp}\,\diff x\diff y$ on $\mathcal Q(B)$.
	Differentiating at the minimum gives
	\[
	0=
	\left.
	\frac{\diff}{\diff t}
	\mathscr F_U(w_0+t\varphi)
	\right|_{t=0}
	=\mathcal B_v(\varphi).
	\]
	Also,
	\[
	v\in W^{s,p}(B_{2r}(x_0))\cap L_{sp}^m(\R^n),
	\qquad
	v=U\quad\text{a.e. in }\R^n\setminus B.
	\]
	If $v_1,v_2\in U+V$ satisfy the same homogeneous equation,
	then \eqref{eq:pcoercive} yields
	\[
	c[v_1-v_2]_{W^{s,p}(\R^n)}^p
	\le
	\mathcal B_{v_1}(v_1-v_2)
	-\mathcal B_{v_2}(v_1-v_2)
	=0.
	\]
	Since $v_1-v_2=0$ outside $B$, this proves uniqueness.
	
	Set $w=U-v\in V$.
	By \eqref{eq:local-weak-pairing-continuity} and
	\eqref{eq:measure-dual-local}, the weak identity for $U$
	extends to all test functions in $V$.
	Subtracting the equations and testing with $w$, we obtain
	\[
	\begin{aligned}
		c[w]_{W^{s,p}(\R^n)}^p
		\le
		\iint_{\mathcal Q(B)}
		\frac{
			[J_p(\delta U)-J_p(\delta v)]\delta w
		}
		{|x-y|^{n+sp}}\,\diff x\diff y=\int_B w\,\diff\nu
		\le |\nu|(B)\|w\|_{L^\infty(B)}.
	\end{aligned}
	\]
	If $w\not\equiv0$,  then using \eqref{eq:fractional-morrey-zero}, we find
	\[
	\|w\|_{L^\infty(B)}^m
	\le Cr^{sp-n}|\nu|(B).
	\]
	The same estimate is immediate when $w=0$.
	Since $sp=m+\beta$,
	\[
	\frac1r\|w\|_{L^\infty(B)}
	\le C\left(
	\frac{|\nu|(B)}{r^{n-\beta}}
	\right)^{1/m}
	=C D_\nu(x_0,r),
	\]
	which proves \eqref{eq:measure-uniform-comparison}.
	
	Since $w=0$ outside $B$, for $0<\rho\le r/2$,
	\[
	\begin{aligned}
		\Tail_{m,sp}(w;x_0,\rho)^m
		&=\rho^{sp}
		\int_{B\setminus B_\rho(x_0)}
		\frac{|w(y)|^m}{|y-x_0|^{n+sp}}\,\diff y\\
		&\le C\|w\|_{L^\infty(B)}^m
		\rho^{sp}\int_\rho^r t^{-1-sp}\,\diff t\le C\|w\|_{L^\infty(B)}^m.
	\end{aligned}
	\]
	Together with \eqref{eq:measure-uniform-comparison},
	this proves \eqref{eq:measure-comparison-tail}.
	
	Finally, fix an affine function $\ell$ and take $a=r/4$.
	Since $sp>n$, local Morrey embedding gives
	$v\in L^\infty(B_{r/2}(x_0))$.  Since $v-\ell=U-\ell-w$, 
	applying \eqref{eq:homogeneous-fixed-affine-supremum} to $v$, Minkowski's inequality and the tail bound
	for $w$ yield
	\begin{equation*}\begin{split}
	\|v-\ell\|_{L^\infty(B_{r/8}(x_0))}
&\le C\left[
	\left(
	\fint_{B_a(x_0)}|v-\ell|^m\,\diff x
	\right)^{1/m}
	+\Tail_{m,sp}(v-\ell;x_0,a)
	\right]\\
	&\le C\left[
	\left(
	\fint_B|U-\ell|^m\,\diff x
	\right)^{1/m}
	+\Tail_{m,sp}(U-\ell;x_0,r)
	+\|w\|_{L^\infty(B)}\right].
	\end{split}\end{equation*}
	Now use
	\[
	\|U-\ell\|_{L^\infty(B_{r/8}(x_0))}
	\le
	\|v-\ell\|_{L^\infty(B_{r/8}(x_0))}
	+\|w\|_{L^\infty(B)},
	\]
	apply \eqref{eq:measure-uniform-comparison}, and divide by $r$.
	This proves \eqref{eq:energy-mean-to-uniform}.
\end{proof}

\subsection{Local energy regularity of SOLA}

\begin{lemma}[Local energy regularity of SOLA]
	\label{lem:sola-local-weak}
	Assume $sp>n$, and let $u$ be a SOLA of \eqref{eq:main}.
	Then $u$ has a representative, still denoted by $u$, such that
	\[
	u\in W^{s,p}_{\mathrm{loc}}(\Omega)
	\cap C^{0,\tau}_{\mathrm{loc}}(\Omega),
	\qquad
	\tau=s-\frac np.
	\]
	Moreover, for every ball $B\Subset\Omega$,
	\begin{equation}\label{eq:sola-local-weak-form}
		\mathcal B_u(\varphi)
		=
		\int_B\varphi\,\diff\mu
		\qquad
		\text{for every }\varphi\in W_0^{s,p}(B).
	\end{equation}
\end{lemma}

\begin{proof}
	Let $(u_j,\mu_j)$ be the approximation in
	Definition~\ref{def:sola}.
	By \eqref{eq:sola-tail-class} and
	\eqref{eq:sola-measure-control}, for each
	$B_{2r}(x)\Subset\Omega$,
	\[
	\sup_j\left[
	\left(
	\fint_{B_r(x)}|u_j|^m\,\diff y
	\right)^{1/m}
	+\Tail_{m,sp}(u_j;x,r)
	\right]<\infty,
	\qquad
	\sup_j|\mu_j|(B_r(x))<\infty.
	\]
	Since $u_j\in W^{s,p}(\R^n)$,
	Lemma~\ref{lem:uniform-replacement-comparison} applies to $u_j$.
	Taking $\ell=0$ in \eqref{eq:energy-mean-to-uniform} gives
	\begin{equation}\label{eq:sola-local-uniform-bound}
		\begin{aligned}
			\|u_j\|_{L^\infty(B_{r/8}(x))}
			&\le C\left[
			\left(
			\fint_{B_r(x)}|u_j|^m\,\diff y
			\right)^{1/m}
			+\Tail_{m,sp}(u_j;x,r)
			\right]\\
			&\quad+
			Cr^{(sp-n)/m}|\mu_j|(B_r(x))^{1/m}.
		\end{aligned}
	\end{equation}
	A finite covering therefore gives
	\[
	\sup_j\|u_j\|_{L^\infty(K)}<\infty
	\qquad\text{for every }K\Subset\Omega.
	\]
	Fix $B_{3\rho}(z)\Subset\Omega$. By the local uniform bound above
	and \eqref{eq:sola-tail-class}, there exist $A,T<\infty$ such that
	\[
	\sup_j\|u_j\|_{L^\infty(B_{3\rho}(z))}\le A,
	\qquad
	\sup_j\Tail_{m,sp}(u_j;z,3\rho)\le T.
	\]
	Let $\eta\in C_c^\infty(B_{2\rho}(z))$ satisfy
	\[
	0\le\eta\le1,\qquad
	\eta=1\quad\text{on }B_\rho(z),
	\qquad
	|D\eta|\le \frac{C}{\rho}.
	\]
	Then $\eta^p u_j\in W_0^{s,p}(B_{2\rho})$ is an admissible
	test function for the equation of $u_j$.

Note that for $a,b\in\R$ and $\xi,\zeta\in[0,1]$, we have
	\[
	\begin{aligned}
		J_p(a-b)(a\xi^p-b\zeta^p)
		&=
		\frac12|a-b|^p(\xi^p+\zeta^p)
		+\frac12J_p(a-b)(a+b)(\xi^p-\zeta^p)\\
		&\ge
		\frac12|a-b|^p(\xi^p+\zeta^p)
		-C|a-b|^{p-1}(|a|+|b|)
		(\xi^p+\zeta^p)^{\frac{p-1}{p}}
		|\xi-\zeta|\\
		&\ge
		\frac14|a-b|^p(\xi^p+\zeta^p)
		-C(|a|+|b|)^p|\xi-\zeta|^p.
	\end{aligned}
	\]
Testing the equation of $u_j$ with $\eta^p u_j$ gives
	\[
	\begin{aligned}
		\int_{B_{2\rho}}\eta^p u_j\,\diff\mu_j
		&=
		\iint_{B_{3\rho}\times B_{3\rho}}
		\frac{
			J_p(u_j(x)-u_j(y))
			\bigl(\eta(x)^pu_j(x)-\eta(y)^pu_j(y)\bigr)
		}
		{|x-y|^{n+sp}}
		\,\diff x\diff y\\
		&\quad+
		2\int_{B_{2\rho}}
		\int_{\R^n\setminus B_{3\rho}}
		\frac{
			J_p(u_j(x)-u_j(y))
			\eta(x)^pu_j(x)
		}
		{|x-y|^{n+sp}}
		\,\diff y\diff x .
	\end{aligned}
	\]
	For $x,y\in B_{3\rho}$,  we have
	\[
	\begin{aligned}
		&J_p(u_j(x)-u_j(y))
		\bigl(\eta(x)^pu_j(x)-\eta(y)^pu_j(y)\bigr)\\
		&\qquad\ge
		\frac14
		|u_j(x)-u_j(y)|^p
		\bigl(\eta(x)^p+\eta(y)^p\bigr)
		-C\bigl(|u_j(x)|+|u_j(y)|\bigr)^p
		|\eta(x)-\eta(y)|^p .
	\end{aligned}
	\]
	Since $	\eta=1$ on $B_\rho$ and $\|u_j\|_{L^\infty(B_{3\rho})}\le A$,
	we obtain
	\[
	\begin{aligned}
		&\iint_{B_{3\rho}\times B_{3\rho}}
		\frac{
			J_p(u_j(x)-u_j(y))
			\bigl(\eta(x)^pu_j(x)-\eta(y)^pu_j(y)\bigr)
		}
		{|x-y|^{n+sp}}
		\,\diff x\diff y\\
		&\qquad\ge
		c[u_j]_{W^{s,p}(B_\rho)}^p
		-
		CA^p
		\iint_{B_{3\rho}\times B_{3\rho}}
		\frac{
			|\eta(x)-\eta(y)|^p
		}
		{|x-y|^{n+sp}}
		\,\diff x\diff y .
	\end{aligned}
	\]
	Therefore,
	\[
	\begin{aligned}
		c[u_j]_{W^{s,p}(B_\rho)}^p
		&\le
		\int_{B_{2\rho}}\eta^p u_j\,\diff\mu_j+
		CA^p
		\iint_{B_{3\rho}\times B_{3\rho}}
		\frac{
			|\eta(x)-\eta(y)|^p
		}
		{|x-y|^{n+sp}}
		\,\diff x\diff y\\
		&\quad+
		2\int_{B_{2\rho}}
		\int_{\R^n\setminus B_{3\rho}}
		\frac{
			|J_p(u_j(x)-u_j(y))|
			\eta(x)^p|u_j(x)|
		}
		{|x-y|^{n+sp}}
		\,\diff y\diff x .
	\end{aligned}
	\]
	The first term is at most $A|\mu_j|(B_{2\rho})$.
	For the cutoff term, $p-sp>0$ gives
	\[
	\begin{aligned}
		\iint_{B_{3\rho}\times B_{3\rho}}
		\frac{|\eta(x)-\eta(y)|^p}{|x-y|^{n+sp}}
		\,\diff x\diff y\le
		C\rho^{n-p}\int_0^{6\rho}t^{p-sp-1}\,\diff t
		\le C\rho^{n-sp}.
	\end{aligned}
	\]
	For $x\in B_{2\rho}$ and $y\notin B_{3\rho}$,
	$|x-y|\ge |y-z|/3$. Hence
	\[
	\begin{aligned}
		\int_{B_{2\rho}}\int_{\R^n\setminus B_{3\rho}}
		\frac{
			|J_p(u_j(x)-u_j(y))|
			\eta(x)^p|u_j(x)|
		}
		{|x-y|^{n+sp}}\,\diff y\diff x	&\le
		CA\int_{B_{2\rho}}\int_{\R^n\setminus B_{3\rho}}
		\frac{A^m+|u_j(y)|^m}{|y-z|^{n+sp}}
		\,\diff y\diff x\\
		&\le
		CA\rho^{n-sp}
		\left[
		A^m+\Tail_{m,sp}(u_j;z,3\rho)^m
		\right]\\
		&\le
		C\rho^{n-sp}(A^p+AT^m).
	\end{aligned}
	\]
	Combining these estimates proves
	\begin{equation}\label{eq:sola-local-energy-bound}
		[u_j]_{W^{s,p}(B_\rho(z))}^p
		\le C\left[
		A|\mu_j|(B_{2\rho}(z))
		+\rho^{n-sp}(A^p+AT^m)
		\right].
	\end{equation}
	The right-hand side is bounded uniformly in $j$.
	Since $u_j\to u$ almost everywhere, Fatou's lemma gives
	\[
	\int_{B_\rho}|u|^p\,\diff x
	\le A^p|B_\rho|,
	\qquad
	[u]_{W^{s,p}(B_\rho)}^p
	\le
	\liminf_{j\to\infty}
	[u_j]_{W^{s,p}(B_\rho)}^p
	<\infty.
	\]
	As the ball was arbitrary,
	$u\in W^{s,p}_{\mathrm{loc}}(\Omega)$.
	
	Since $sp>n$, local Morrey embedding provides a
	$C^{0,\tau}$ representative on every interior ball,
	where $\tau=s-n/p$.
	On two overlapping balls, these representatives are continuous
	and equal almost everywhere, so they agree everywhere on the
	overlap. They therefore define a unique representative in
	$C^{0,\tau}_{\mathrm{loc}}(\Omega)$.
	
Finally, by density of $C_c^\infty(B)$ in $W_0^{s,p}(B)$,
\eqref{eq:local-weak-pairing-continuity},
\eqref{eq:fractional-morrey-zero}, and the distributional identity,
\[
\mathcal B_u(\varphi)=\int_B\varphi\,\diff\mu
\qquad
\text{for every }\varphi\in W_0^{s,p}(B).
\]
This proves \eqref{eq:sola-local-weak-form}.
\end{proof}

In what follows,  we use the continuous local weak solution supplied by
Lemma~\ref{lem:sola-local-weak}.
The statement is local and does not assert that $u$ belongs to
$W_0^{s,p}(\Omega)$ or identify all SOLA with a global variational
solution. The local comparison lemma now applies directly to $u$.
In particular, it follows from  \eqref{eq:energy-mean-to-uniform}  that for every affine function $\ell$ and every
$B_{2r}(x_0)\Subset\Omega$,
\begin{align}\label{eq:sola-mean-to-uniform}
\begin{split}
\frac1r\|u-\ell\|_{L^\infty(B_{r/8}(x_0))}
&\le\frac Cr\left[
\left(\fint_{B_r(x_0)}|u-\ell|^m\,\diff x\right)^{1/m}
+\Tail_{m,sp}(u-\ell;x_0,r)\right]\\
&\quad+CD_\mu(x_0,2r).
\end{split}
\end{align}

\subsection{Proof of classical differentiability}
\begin{proof}[Proof of Theorem~\ref{thm:wolff}]
Fix $x_0,R$ as in the statement. Use the radii $r_k=(R/2)\theta^k$
and the affine functions $\ell_k$ from
Lemma~\ref{lem:mean-affine-expansion}. With $E_k$ and
$D_k=D_\mu(x_0,2r_k)$ as in that lemma, we have
\[
E_k+D_k\longrightarrow0,
\qquad a_k\longrightarrow a_*,
\qquad \frac{|b_k-b_*|}{r_k}\longrightarrow0.
\]
Since $\ell_k$ minimizes $E_k$, \eqref{eq:sola-mean-to-uniform} gives
\[
\frac1{r_k}\|u-\ell_k\|_{L^\infty(B_{r_k/8}(x_0))}
\le C(E_k+D_k).
\]
For $\ell_*(x)=b_*+a_*\cdot(x-x_0)$, it follows that
\begin{equation}\label{eq:uniform-from-mean-expansion}
\frac1{r_k}\|u-\ell_*\|_{L^\infty(B_{r_k/8}(x_0))}
\le C(E_k+D_k)+\frac{|b_k-b_*|}{r_k}+\frac18|a_k-a_*|
=:q_k\longrightarrow0.
\end{equation}
For the continuous representative, this bound holds at every point.
At $x_0$ it gives $|u(x_0)-b_*|\le r_kq_k$, hence $u(x_0)=b_*$.
If $r_{k+1}/8\le |x-x_0|<r_k/8$, then
\[
\frac{|u(x)-u(x_0)-a_*\cdot(x-x_0)|}{|x-x_0|}
\le \frac{r_kq_k}{r_{k+1}/8}
=8\theta^{-1}q_k\longrightarrow0.
\]
Thus $u$ is Fr\'echet differentiable at $x_0$, with classical gradient
$a_*$. The estimate \eqref{eq:mean-all-scale-bound} gives
\eqref{eq:wolff-gradient}.
At a Lebesgue point of the weak gradient,
Theorem~\ref{thm:mean-gradient} and its proof identify $a_*$ with the
Lebesgue value of $\nabla u$. Thus the classical gradient agrees with
the weak gradient at those points, and hence almost everywhere.
\end{proof}

	\subsection*{Acknowledgments}
	This work was supported by the National Natural Science Foundation of China (No. 12471128).
	
	\subsection*{Conflict of interest}
	The authors declare that there is no conflict of interest.
	
	\subsection*{Data availability}
	Data sharing is not applicable to this article because no datasets were
	generated or analyzed during the current study.

	\subsection*{AI disclosure} During the preparation of this manuscript, the authors used ChatGPT-6 Astra as an auxiliary tool for computations, editorial assistance, and preliminary exploration of proof ideas. All computations, proofs, references, and conclusions were independently verified by the authors, who take full responsibility for the content of this paper.
	

\begin{thebibliography}{10}
		
		\bibitem{Baroni2015}
		Paolo Baroni, \emph{{Riesz potential estimates for a general class of
				quasilinear equations}}, Calc. Var. Partial Differential Equations
		\textbf{53} (2015), no.~3--4, 803--846.
		
		\bibitem{BiswasTopp2025}
		Anup Biswas and Erwin Topp, \emph{{Lipschitz regularity of fractional
				$p$-Laplacian}}, Ann. PDE \textbf{11} (2025), no.~2, 27.
		
		\bibitem{BogeleinDuzaarLiaoMolicaBisciServadei2025}
		Verena B\"ogelein, Frank Duzaar, Naian Liao, Giovanni {Molica Bisci}, and
		Raffaella Servadei, \emph{{Regularity for the fractional $p$-Laplace
				equation}}, J. Funct. Anal. \textbf{289} (2025), no.~9, 111078, 69 pp.
		
		\bibitem{BogeleinDuzaarLiaoMoring2025}
		Verena B\"ogelein, Frank Duzaar, Naian Liao, and Kristian Moring,
		\emph{{Gradient estimates for the fractional $p$-Poisson equation}}, J. Math.
		Pures Appl. (9) \textbf{204} (2025), 103764, 25 pp.
		
		\bibitem{BogeleinDuzaarLiaoMoring2026}
		\bysame, \emph{{Sharp gradient integrability for $(s,p)$-Poisson type
				equations}}, arXiv:2602.08944, 2026.
		
		\bibitem{BrascoLindgren2017}
		Lorenzo Brasco and Erik Lindgren, \emph{{Higher Sobolev regularity for the
				fractional $p$-Laplace equation in the superquadratic case}}, Adv. Math.
		\textbf{304} (2017), 300--354.
		
		\bibitem{BrascoLindgrenSchikorra2018}
		Lorenzo Brasco, Erik Lindgren, and Armin Schikorra, \emph{{Higher H\"older
				regularity for the fractional $p$-Laplacian in the superquadratic case}},
		Adv. Math. \textbf{338} (2018), 782--846.
		
		\bibitem{ByunSongYoun2023}
		Sun-Sig Byun, Kyeong Song, and Yeonghun Youn, \emph{{Potential estimates for
				elliptic measure data problems with irregular obstacles}}, Math. Ann.
		\textbf{387} (2023), no.~1--2, 745--805.
		
		\bibitem{ChlebickaGiannettiZatorskaGoldstein2024}
		Iwona Chlebicka, Flavia Giannetti, and Anna Zatorska-Goldstein, \emph{{Wolff
				potentials and local behavior of solutions to elliptic problems with Orlicz
				growth and measure data}}, Adv. Calc. Var. \textbf{17} (2024), no.~4,
		1293--1321.
		
		\bibitem{ChlebickaKimWeidner2026}
		Iwona Chlebicka, Minhyun Kim, and Marvin Weidner, \emph{{Gradient Riesz
				potential estimates for a general class of measure data quasilinear
				systems}}, Adv. Calc. Var. \textbf{19} (2026), no.~2, 237--269.
		
		\bibitem{DiCastroKuusiPalatucci2014}
		Agnese Di~Castro, Tuomo Kuusi, and Giampiero Palatucci, \emph{{Nonlocal Harnack
				inequalities}}, J. Funct. Anal. \textbf{267} (2014), no.~6, 1807--1836.
		
		\bibitem{DiCastroKuusiPalatucci2016}
		\bysame, \emph{{Local behavior of fractional $p$-minimizers}}, Ann. Inst. H.
		Poincar\'e C Anal. Non Lin\'eaire \textbf{33} (2016), no.~5, 1279--1299.
		
		\bibitem{DiNezzaPalatucciValdinoci2012}
		Eleonora Di~Nezza, Giampiero Palatucci, and Enrico Valdinoci,
		\emph{{Hitchhiker's guide to the fractional Sobolev spaces}}, Bull. Sci.
		Math. \textbf{136} (2012), no.~5, 521--573.
		
		\bibitem{DieningKimLeeNowak2025a}
		Lars Diening, Kyeongbae Kim, Ho-Sik Lee, and Simon Nowak, \emph{{Gradient
				estimates for parabolic nonlinear nonlocal equations}}, Calc. Var. Partial
		Differential Equations \textbf{64} (2025), no.~3, 98, 86 pp.
		
		\bibitem{DieningKimLeeNowak2025b}
		\bysame, \emph{{Higher differentiability for the fractional $p$-Laplacian}},
		Math. Ann. \textbf{391} (2025), no.~4, 5631--5693.
		
		\bibitem{DieningKimLeeNowak2025c}
		\bysame, \emph{{Nonlinear nonlocal potential theory at the gradient level}}, J.
		Eur. Math. Soc. (2025), Published online, doi:10.4171/JEMS/1706.
		
		
		\bibitem{DieningNowak2025}
		Lars Diening and Simon Nowak, \emph{{Calder\'on--Zygmund estimates
				for the fractional $p$-Laplacian}}, Ann. PDE \textbf{11} (2025),
		6, 33 pp.

	
		\bibitem{DongZhu2022}
		Hongjie Dong and Hanye Zhu, \emph{{Gradient estimates for singular parabolic
				$p$-Laplace type equations with measure data}}, Calc. Var. Partial
		Differential Equations \textbf{61} (2022), no.~3, 86, 41 pp.
		
		\bibitem{DongZhu2024}
		\bysame, \emph{{Gradient estimates for singular $p$-Laplace type equations with
				measure data}}, J. Eur. Math. Soc. (JEMS) \textbf{26} (2024), no.~10,
		3939--3985.
		
		\bibitem{DuzaarMingione2010}
		Frank Duzaar and Giuseppe Mingione, \emph{{Gradient estimates via linear and
				nonlinear potentials}}, J. Funct. Anal. \textbf{259} (2010), no.~11,
		2961--2998.
		
		\bibitem{DuzaarMingione2011}
		\bysame, \emph{{Gradient estimates via non-linear potentials}}, Amer. J. Math.
		\textbf{133} (2011), no.~4, 1093--1149.
		
		\bibitem{FernandezRealRosOton2024}
		Xavier Fern\'andez-Real and Xavier Ros-Oton, \emph{{Schauder and
				Cordes--Nirenberg estimates for nonlocal elliptic equations with singular
				kernels}}, Proc. Lond. Math. Soc. (3) \textbf{129} (2024), no.~3, e12629, 47
		pp.
		
		\bibitem{GiovagnoliJesusSilvestre2025}
		Davide Giovagnoli, David Jesus, and Luis Silvestre, \emph{{$C^{1+\alpha}$
				regularity for fractional $p$-harmonic functions}}, arXiv:2509.26565v1, 2025.
		
		\bibitem{HeinonenKilpelainenMartio2006}
		Juha Heinonen, Tero Kilpel\"ainen, and Olli Martio, \emph{{Nonlinear potential
				theory of degenerate elliptic equations}}, Dover Publications, Inc., Mineola,
		NY, 2006, Unabridged republication of the 1993 original.
		
		\bibitem{IannizzottoMosconiSquassina2016}
		Antonio Iannizzotto, Sunra Mosconi, and Marco Squassina, \emph{{Global H\"older
				regularity for the fractional $p$-Laplacian}}, Rev. Mat. Iberoam. \textbf{32}
		(2016), no.~4, 1353--1392.
		
		\bibitem{KilpelainenKuusiTuholaKujanpaa2011}
		Tero Kilpel\"ainen, Tuomo Kuusi, and Anna Tuhola-Kujanp\"a\"a,
		\emph{{Superharmonic functions are locally renormalized solutions}}, Ann.
		Inst. H. Poincar\'e C Anal. Non Lin\'eaire \textbf{28} (2011), no.~6,
		775--795.
		
		\bibitem{KilpelainenMaly1992}
		Tero Kilpel\"ainen and Jan Mal\'y, \emph{{Degenerate elliptic equations with
				measure data and nonlinear potentials}}, Ann. Scuola Norm. Sup. Pisa Cl. Sci.
		(4) \textbf{19} (1992), no.~4, 591--613.
		
		\bibitem{KilpelainenMaly1994}
		\bysame, \emph{{The Wiener test and potential estimates for quasilinear
				elliptic equations}}, Acta Math. \textbf{172} (1994), no.~1, 137--161.
		
		\bibitem{KimLeeLee2023}
		Minhyun Kim, Ki-Ahm Lee, and Se-Chan Lee, \emph{{The Wiener criterion for
				nonlocal Dirichlet problems}}, Comm. Math. Phys. \textbf{400} (2023), no.~3,
		1961--2003.
		
		\bibitem{KimLeeLee2025}
		\bysame, \emph{{Wolff potential estimates and Wiener criterion for nonlocal
				equations with Orlicz growth}}, J. Funct. Anal. \textbf{288} (2025), no.~1,
		110690, 51 pp.
		
		\bibitem{KorteKuusi2010}
		Riikka Korte and Tuomo Kuusi, \emph{{A note on the Wolff potential estimate for
				solutions to elliptic equations involving measures}}, Adv. Calc. Var.
		\textbf{3} (2010), no.~1, 99--113.
		
		\bibitem{KuusiMingione2013}
		Tuomo Kuusi and Giuseppe Mingione, \emph{{Linear potentials in nonlinear
				potential theory}}, Arch. Ration. Mech. Anal. \textbf{207} (2013), no.~1,
		215--246.
		
		\bibitem{KuusiMingione2014}
		\bysame, \emph{{Guide to nonlinear potential estimates}}, Bull. Math. Sci.
		\textbf{4} (2014), no.~1, 1--82.
		
		\bibitem{KuusiMingioneSire2015}
		Tuomo Kuusi, Giuseppe Mingione, and Yannick Sire, \emph{{Nonlocal equations
				with measure data}}, Comm. Math. Phys. \textbf{337} (2015), no.~3,
		1317--1368.
		
		\bibitem{KuusiNowakSire2024}
		Tuomo Kuusi, Simon Nowak, and Yannick Sire, \emph{{Gradient regularity and
				first-order potential estimates for a class of nonlocal equations}},
		arXiv:2212.01950v2, 2024, To appear in American Journal of Mathematics.
		
		\bibitem{LeeLee2021}
		Ki-Ahm Lee and Se-Chan Lee, \emph{{The Wiener criterion for elliptic equations
				with Orlicz growth}}, J. Differential Equations \textbf{292} (2021),
		132--175.
		
		\bibitem{Mingione2011}
		Giuseppe Mingione, \emph{{Gradient potential estimates}}, J. Eur. Math. Soc.
		(JEMS) \textbf{13} (2011), no.~2, 459--486.
		
		\bibitem{NguyenOkSong2026}
		Quoc-Hung Nguyen, Jihoon Ok, and Kyeong Song, \emph{{Wolff potentials and
				nonlocal equations of Lane--Emden type}}, arXiv:2405.11747v2, 2026, To appear
		in Transactions of the American Mathematical Society.
		
		\bibitem{NguyenPhuc2020}
		Quoc-Hung Nguyen and Nguyen~Cong Phuc, \emph{{Pointwise gradient estimates for
				a class of singular quasilinear equations with measure data}}, J. Funct.
		Anal. \textbf{278} (2020), no.~5, 108391, 35 pp.
		
		\bibitem{NguyenPhuc2023}
		\bysame, \emph{{A comparison estimate for singular $p$-Laplace equations and
				its consequences}}, Arch. Ration. Mech. Anal. \textbf{247} (2023), no.~3, 49,
		24 pp.
		
		\bibitem{Nowak2023}
		Simon Nowak, \emph{{Improved Sobolev regularity for linear nonlocal equations
				with VMO coefficients}}, Math. Ann. \textbf{385} (2023), 1323--1378.
		
		\bibitem{Scheven2012}
		Christoph Scheven, \emph{{Gradient potential estimates in non-linear elliptic
				obstacle problems with measure data}}, J. Funct. Anal. \textbf{262} (2012),
		no.~6, 2777--2832.
		
		\bibitem{Schikorra2016}
		Armin Schikorra, \emph{{Nonlinear commutators for the fractional $p$-Laplacian
				and applications}}, Math. Ann. \textbf{366} (2016), no.~1--2, 695--720.
		
		\bibitem{TrudingerWang2002}
		Neil~S. Trudinger and Xu-Jia Wang, \emph{{On the weak continuity of elliptic
				operators and applications to potential theory}}, Amer. J. Math. \textbf{124}
		(2002), no.~2, 369--410.
		
		\bibitem{TrudingerWang2009}
		\bysame, \emph{{Quasilinear elliptic equations with signed measure data}},
		Discrete Contin. Dyn. Syst. \textbf{23} (2009), no.~1--2, 477--494.
		
		\bibitem{XiongZhangMa2026}
		Qi~Xiong, Zhenqiu Zhang, and Lingwei Ma, \emph{{Riesz potential estimates for
				double obstacle problems with Orlicz growth}}, J. Differential Equations
		\textbf{464} (2026), 114192, 43 pp.
		
		\bibitem{XuZhao2026}
		Longjuan Xu and Yirui Zhao, \emph{{Gradient continuity estimates for elliptic
				equations of $p$-Laplace type with measure data}}, Calc. Var. Partial
		Differential Equations \textbf{65} (2026), no.~7, 219, 44 pp.
		
	\end{thebibliography}
%
%

\end{document}